\documentclass[preprint,pdftex,english]{imsart}
\usepackage{amsthm,amsmath,amssymb,amsfonts,epic,multirow,ulem}

\usepackage[numbers]{natbib}
\RequirePackage{natbib}
\RequirePackage[colorlinks,citecolor=blue,urlcolor=blue]{hyperref}

\usepackage{wrapfig}
\usepackage[english]{babel}
 \usepackage[T1]{fontenc} 

 \usepackage{makeidx}
 \usepackage{fancyhdr}
\usepackage{epsf}
\usepackage[dvips]{epsfig}  
\usepackage{subfig}
 \usepackage{graphicx}
\usepackage[nice]{nicefrac}
\usepackage{dsfont}
\usepackage{calrsfs} %Lettres plus jolies

\usepackage{pifont}
\usepackage{bm}

\usepackage{wasysym}
\usepackage{stmaryrd}
\usepackage{aeguill}
\usepackage{mathrsfs}
\usepackage{enumerate}
\usepackage{enumitem}
\usepackage[dvipsnames,table,xcdraw]{xcolor}
\usepackage{cancel}
\usepackage{ulem}
\usepackage{caption}
\usepackage{tikz}

\startlocaldefs

\newtheorem{theo}{Theorem}[section]

\newtheorem{coro}[theo]{Corollary}

\newtheorem{lemm}[theo]{Lemma}

\def\R{\mathbb{R}}
\def \N{\mathbb{N}}

\def \P{\mathbb{P}} % proba 

\def \E{\mathbb{E}} % esperance 
\newcommand{ \un }{\mathds{1}}

\def\w{\omega}

\def\T{\mathbb{T}}

\newenvironment{merci}{\textbf{Acknowledgments.}}{ }
\endlocaldefs

\newtheorem{theorem}{Theorem}

\newtheorem{proposition}[theorem]{Proposition}
\newtheorem{Ass}{Assumption}
\newtheorem{remark}{Remark}

\newtheorem{example}[remark]{Example}

\renewcommand{\T}{\mathbb{T}}
\newcommand{\X}{\mathbb{X}}

\def \Eb{\mathbf{E}}
\def \Pb{\mathbf{P}}

\renewcommand{\P}{\mathbb{P}}

\newtheorem{postita}{Post-it}

\renewcommand{\P}{\mathbb{P}}

\def\T{\mathbb{T}}

\makeatletter
\newcommand*\bigcdot{\mathpalette\bigcdot@{.5}}
\newcommand*\bigcdot@[2]{\mathbin{\vcenter{\hbox{\scalebox{#2}{$\m@th#1\bullet$}}}}}
\makeatother

\newcommand{\Vect}[1]{\bm{#1}}
\newcommand{\VectCoord}[2]{#1^{(#2)}}
\newcommand{\fInf}{\|f\|_{\infty}}

\newcommand{\sqrtBis}[1]{#1^{1/2}}

\newcommand{\ProdSet}[2]{#1^{\times #2}}

\begin{document}

\begin{frontmatter}

%%%%%%%%%
%%%%%%%%%%

\title{Coalescence in small generations for the diffusive randomly biased walk on Galton-Watson trees}

\author{\fnms{Alexis} \snm{Kagan}\ead[label=e2]{alexis.kagan@auckland.co.nz}}
\address{Institut Denis Poisson,   UMR CNRS 7013,
Universit\'e d'Orl\'eans, Orl\'eans, France. \printead{e2}} \vspace{0.5cm}

\runauthor{Kagan}

%Version : \today{}

\runtitle{Coalescence in small generations for the diffusive randomly biased walk on Galton-Watson trees}

\begin{abstract}
We investigate the range $\mathcal{R}_T$ of the diffusive biased walk $\X$ on a Galton-Watson tree $\T$ in random environment, that is to say the sub-tree of $\T$ of all distinct vertices visited by this walk up to the time $T$. We study the volume of the range with constraints and more precisely the number of $k$-tuples ($k\geq 2$) of distinct vertices in this sub-tree, in small generations and satisfying an hereditary condition. A special attention is paid to the vertices visited during distinct excursions of $\X$ above the root of the Galton-Watson tree as we observe they give the major contribution to this range. As an application, we study the genealogy of $k\geq 2$ distinct vertices of the tree $\mathcal{R}_T$ picked uniformly from those in small generations. It turns out that two or more vertices among them share a common ancestor for the last time in the remote past. We also point out an hereditary character in their genealogical tree due to the random environment. 

\end{abstract}

% \begin{keyword}[class=AenMS]
 % \kwd[MSC : Primary ] 
% \end{keyword}
% % 62M05 Markov processes: estimation
% % 62F12 Asymptotic properties of estimators
% % 60J25 Markov processes with continuous parameter
% % 60J27 Markov chains with continuous parameter
% % 60J35 Transition functions, generators and resolvents


\end{frontmatter}

%%%%%%%%
%%%%%%%%

\section{Introduction}

\subsection{Randomly biased random walk on trees}\label{RBRWT}
Given, under a probability measure $\Pb$, a $\bigcup_{j\in\N}\R^j$-valued random variable $\mathcal{P}$ ($\mathbb{R}^0$ only contains the sequence with length $0$) with $N:=\#\mathcal{P}$ denoting the cardinal of $\mathcal{P}$, we consider the following Galton-Watson marked tree $(\T,(A_x;x\in\T))$ rooted at $e$: the generation $0$ contains one marked individual $(e,A_e)=(e,0)$. For any $n\in\N^*$, assume the generation $n-1$ has been built. If it is empty, then the generation $n$ is also empty. Otherwise, for any vertex $x$ in the generation $n-1$, let $\mathcal{P}^x:=\{A_{x^{1}},\ldots,A_{x^{N(x)}}\}$ be a random variable distributed as $\mathcal{P}$ where $N(x):=\#\mathcal{P}^x$. The vertex $x$ gives progeny to $N(x)$ marked children $(x^{1},A_{x^{1}}),\ldots,(x^{N(x)},A_{x^{N(x)}})$ independently of the other vertices in generation $n-1$, thus forming the generation $n$, denoted by $\T_n$. We assume $\Eb[N]>1$ so that $\T$ is a supercritical Galton-Watson tree, that is $\Pb(\textrm{non-extinction of }\T)>0$ and we define $\Pb^*(\cdot):=\Pb(\cdot|\textrm{non-extinction of }\T)$, $\Eb$ (resp. $\Eb^*$) denotes the expectation with respect to $\Pb$ (resp. $\Pb^*$). \\
For any vertex $x\in\T$, we denote by $|x|$ the generation of $x$, by $x_i$ its ancestor in generation $i\in\{0,\ldots,|x|\}$ and $x^*:=x_{|x|-1}$ stands for the parent of $x$. In particular, $x_0=e$ and $x_{|x|}=x$. For any $x,y\in\T$, we write $x\leq y$ if $x$ is an ancestor of $y$ ($y$ is said to be a descendent of $x$) and $x<y$ if $x\leq y$ and $x\not=y$. We then write $\llbracket x_i,x\rrbracket:=\{x_j; j\in\{i,\ldots,|x|\}\}$.  Finally, for any $x,y\in\T$, we denote by $x\land y$ the most recent common ancestor of $x$ and $y$, that is the ancestor $u$ of $x$ and $y$ such that $\max\{|z|;\; z\in\llbracket e,x\rrbracket\cap\llbracket e,y\rrbracket\}=|u|$. \\
Let us introduce the branching potential $V$: let $V(e)=A_e=0$ and for any $x\in\T\setminus{\{e}\}$
$$ V(x):=\sum_{e<z\leq x}A_z=\sum_{i=1}^{|x|} A_{x_i}. $$
Under $\Pb$, $\mathcal{E}:=(\T,(V(x);x\in\T))$ is a real valued branching random walk such that $(V(x)-V(x^*))_{|x|=1}$ is distributed as $\mathcal{P}$. We will then refer to $\mathcal{E}$ as the random environment. \\
For convenience, we add a parent $e^*$ to the root $e$ and we introduce the $\T\cup\{e^*\}$-valued random walk $\X:=(X_j)_{j\in\N}$ reflected in $e^*$ such that under the quenched probabilities $\{\P^{\mathcal{E}}_z; z\in\T\cup\{e^*\}\}$ (that is $\P^{\mathcal{E}}_z(X_0=z)=1$), the transition probabilities are given by: for any $x\in\T$
\begin{align*}
    p^{\mathcal{E}}(x,x^*)=\frac{e^{-V(x)}}{e^{-V(x)}+\sum_{i=1}^{N_x}e^{-V(x^i)}}\;\;\textrm{ and for all } 1\leq j\leq N_x,\;\; p^{\mathcal{E}}(x,x^j)=\frac{e^{-V(x^j)}}{e^{-V(x)}+\sum_{i=1}^{N_x}e^{-V(x^i)}}.
\end{align*}
Otherwise, $p^{\mathcal{E}}(x,u)=0$ and $p^{\mathcal{E}}(e^*,e)=1$. Let $\P^{\mathcal{E}}:=\P^{\mathcal{E}}_e$ and we finally define the following annealed probabilities
$$ \P(\cdot):=\Eb[\P^{\mathcal{E}}(\cdot)]\;\;\textrm{ and }\;\;\P^*(\cdot):=\Eb^*[\P^{\mathcal{E}}(\cdot)]. $$ 
R. Lyons and R. Pemantle \cite{LyonPema} initiated the study of the randomly biased random walk $\X$. \\
When, for all $x\in\T$, $V(x)=|x|\log\lambda$ for a some constant $\lambda>0$, the walk $\X$ is known as the $\lambda$-biased random walk on $\T\cup\{e^*\}$ and was first introduced by R. Lyons (see \cite{Lyons} and \cite{Lyons2}). The $\lambda$-biased random walk is transient unless the bias is strong enough: if $\lambda\geq\Eb[N]$ then, $\Pb^*$-almost surely, $\X$ is recurrent (positive recurrent if $\lambda>\Eb[N]$). It is known since Y. Peres and O. Zeitouni \cite{PeresZeitouni}, G. Faraud \cite{Faraud} and E. Aïdékon and L. de Raphélis \cite{AidRap} that when $\lambda=\Eb[N]$, $\X$ is diffusive: there exists $\bm{\sigma}^2\in(0,\infty)$ such that $(|X_{\lfloor nt\rfloor}|/\sqrt{\bm{\sigma}^2n})_{t\geq 0}$ converges in law to a standard reflected Brownian motion. R. Lyons, R. Pemantle and Y. Peres (see \cite{LyonsRussellPemantle1} and \cite{LyonsRussellPemantle2}), later joined by E. Aïdékon \cite{AidekonSpeed} and G. Ben Arous, A. Fribergh, N. Gantert, A. Hammond \cite{BA_F_G_H} for example, studied the transient case and showed that $\X$ has a deterministic and explicit speed $v_{\lambda}:=\lim_{n\to\infty}|X_n|/n$ (see \cite{AidekonSpeed} for the expression of $v_{\lambda}$ in the case of positive speed and \cite{BA_F_G_H} for details about the behaviour of $|X_n|$ when $v_{\lambda}=0$). \\
When the bias is random, the behavior of $\X$ depends on the fluctuations of the following $\log$-Laplace transform $\psi(t):=\log\Eb[\sum_{|x|=1}e^{-tV(x)}]$ which we assume to be well defined on $[0,1]$: as stated by R. Lyons and R. Pemantle \cite{LyonPema}, if $\inf_{t\in[0,1]}\psi(t)>0$, then $\Pb^*$-almost surely, $\X$ is transient and we refer to the work of E. Aïdékon \cite{Aidekon2008} for this case. Otherwise, it is recurrent. More specifically, G. Faraud \cite{Faraud} proved that the random walk $\X$ is $\Pb^*$-almost surely positive recurrent either if $\inf_{t\in[0,1]}\psi(t)<0$ or if $\inf_{t\in[0,1]}\psi(t)=0$ and $\psi'(1)>0$. It is null recurrent if $\inf_{t\in[0,1]}\psi(t)=0$ and $\psi'(1)\leq 0$. When $\psi'(1)=0$, the largest generation reached by the walk $\X$ up to time $n$ is of order $(\log n)^3$ (and it is usually referred to as the slow regime for the random random $\X$, see \cite{HuShi10a} and \cite{HuShi10b}) but surprisingly, the generation of the vertex $X_n$ is of order $(\log n)^2$ as $n\to\infty$, see \cite{HuShi15} .\\
In the present paper, we focus on the null recurrent randomly biased walk $\X$ and assume
\begin{Ass}\label{Assumption1}
\begin{align}\label{DiffCase}
    \inf_{t\in[0,1]}\psi(t)=\psi(1)=0\;\;\textrm{ and }\;\;\psi'(1)<0.
\end{align}
\end{Ass}
\noindent Let us introduce
\begin{align}\label{DefKappa}
    \kappa:=\inf\{t>1;\; \psi(t)=0\},
\end{align}
and assume $\kappa\in(1,\infty)$. Under \eqref{DiffCase} and some integrability conditions, it has been proven that $|X_n|$ and $\max_{1\leq j\leq n}|X_j|$ are of order $n^{1-1/\min(\kappa,2)}$(see \cite{HuShi10}, \cite{Faraud}, \cite{AidRap} and \cite{deRaph1}). In other words, the random walk $\X$ is sub-diffusive for $\kappa\in(1,2]$ and diffusive for $\kappa>2$. In this paper, we put ourselves in the latter case. \\
We now define the range of the random walk $\X$. Let $T\in\N^*$. The range $\mathcal{R}_T$ of the random walk $\X$ is the set of distinct vertices of $\T$ visited by $\X$ up to time $T$: if $\mathcal{L}_u^T:=\sum_{j=1}^T\un_{\{X_j=u\}}$ denotes the local time of a vertex $u\in\T$ at time $T$ then
\begin{align}
    \mathcal{R}_T=\{u\in\T;\; \mathcal{L}_u^T\geq 1\},
\end{align}
its cardinal is denoted by $R_T$ and we also called it range. It has been proved by E. Aïdékon and L. de Raphélis that $R_n$ is of order $n$ (see the proofs of Theorem
1.1 and Theorem 6.1 in \cite{AidRap}). Moreover, $(\mathcal{R}_n)$ is a sequence of finite sub-trees of $\T$ and still according to E. Aïdékon and L. de Raphélis (Theorem 6.1 in \cite{AidRap}), after being properly renormalized, this sequence converges in law under both annealed and quenched probabilities to a random real tree when $n$ goes to infinity. \\
Introduce $T^j$, the $j$-th return time to $e^*$: $T^0=0$ and for any $j\geq 1$, $T^j=\inf\{i>T^{j-1};\; X_i=e^*\}$. Thanks to a result of Y. Hu (\cite{Hu2017}, Corollary 1.2), we know that $T^{\sqrtBis{n}}$ is of order $n$. We will be focusing our attention on the range $\mathcal{R}_{T^{\sqrtBis{n}}}$ and we shall finally present an extension of the range $\mathcal{R}_n$. For this purpose, it is convenient to split the tree $\mathcal{R}_{T^{\sqrtBis{n}}}$ in three: the vertices located in what we call the tiny generations, that is those smaller than $\Bar{\gamma}\log n$ for some constant $\Bar{\gamma}>0$ defined below (see the subsection \ref{FurtherDiscusion}), the critical generations, that is to say of order $\sqrtBis{n}$ and corresponding to the typical generations but also to the largest reached by the diffusive random walk $\X$ up to the time $T^{\sqrtBis{n}}$ and finally, the vertices located in what we are going to be calling the small generations. Let $(\mathfrak{L}_n)$ be a sequence of positive integers such that $\mathfrak{L}_n\geq\delta_0^{-1}\log n$ (see Lemma \ref{MinPotential} for the definition of $\delta_0$). A vertex $x\in\mathcal{R}_{T^{\sqrtBis{n}}}$ is said to be in a small generation if it is located above the tiny generations but below the critical generations of the diffusive random walk $\X$, that is if $|x|=\mathfrak{L}_n$ and satisfies
\begin{Ass}[The small generations]\label{AssumptionSmallGenerations}
Let $(\Lambda_i)_{i\in\N}$ be the sequence of functions defined recursively by: for all $t>0$, $\Lambda_0(t)=t$ and for any $i\in\N^*$, $\Lambda_{i-1}(t)=e^{\Lambda_i(t)}$. There exists $l_0\in\N$ such that 
\begin{align}\label{SmallGenerations}
    \lim_{n\to\infty}\frac{\mathfrak{L}_n}{\sqrtBis{n}}\Lambda_{l_0}(\mathfrak{L}_n)=0.
\end{align}
\end{Ass}
\noindent Assumption \ref{AssumptionSmallGenerations} ensures that $\mathfrak{L}_n/\sqrtBis{n}$, renormalized by a sequence that grows very slowly, goes to $0$ when $n$ goes to $\infty$.

\vspace{0.2cm}

\noindent Let us now define an extension of the volume $R_{T^{\sqrtBis{n}}}$: for any integer $k\geq 2$ and any subset $\mathfrak{D}$ of $\T$, let $\ProdSet{\mathfrak{D}}{k}:=\mathfrak{D}\times\cdots\times\mathfrak{D}$ ($k$ times) and $|\mathfrak{D}|$ stands for the cardinal of $\mathfrak{D}$. Introduce, on the set of non-extinction, the subset $\Delta^k$ of $\ProdSet{\T}{k}$ such that a $k$-tuple $(\VectCoord{x}{1},\ldots,\VectCoord{x}{k})$ belongs to $\Delta^k$ if and only if for any $i_1,i_2\in\{1,\ldots,k\}$, $i_1\not=i_2$, we have $\VectCoord{x}{i_1}\not\in\llbracket e,\VectCoord{x}{i_2}\rrbracket$ and $\VectCoord{x}{i_2}\not\in\llbracket e,\VectCoord{x}{i_1}\rrbracket$. In other words, neither $\VectCoord{x}{i_1}$ is an ancestor of $\VectCoord{x}{i_2}$, nor $\VectCoord{x}{i_2}$ is an ancestor of $\VectCoord{x}{i_1}$. Also introduce the set $\Delta^k(\mathfrak{D}):=\Delta^k\cap\ProdSet{\mathfrak{D}}{k}$. For any $n\in\N^*$, any subset $\mathcal{D}_n$ of $\mathcal{R}_{T^{\sqrtBis{n}}}$ and for any bounded function $f:\Delta^k\longrightarrow\R^+$, if $|\Delta^k(\mathcal{D}_n)|>0$, we define the range $\mathcal{A}^k(\mathcal{D}_n,f)$ by
\begin{align}\label{GenTrace}
    \mathcal{A}^k(\mathcal{D}_n,f):=\sum_{\Vect{x}\in\Delta^k(\mathcal{D}_n)}f(\Vect{x}).
\end{align}
Otherwise, $\mathcal{A}^k(\mathcal{D}_n,f)$ is equal to $0$. The aim of studying the range $\mathcal{A}^k(\mathcal{D}_n,f)$ is to understand the interactions between the vertices in the tree $\mathcal{R}_{T^{\sqrtBis{n}}}$ and to give a description of the genealogy of the vertices in $\mathcal{R}_{T^{\sqrtBis{n}}}$. Note that the range we investigate here differs from the range studied in \cite{AndAKHightPotential}, where authors focus on the interactions between the trajectories of the random walk $\X$ and on the trajectories of the underlying branching potential $V$.

\subsection{Genealogy of uniformly chosen vertices in the range}\label{GenVertices}

The range $\mathcal{A}^k(\mathcal{D}_n,f)$ has a very natural interpretation in terms of vertices picked at random in the set $\mathcal{D}_n$. If $\P^*(|\Delta^k(\mathcal{D}_n)|>0)>0$, introduce, on the set of non-extinction 
\begin{align}\label{LawVertices2}
    \bm{\mu}^n(f):=\frac{1}{\P^*(|\Delta^k(\mathcal{D}_n)|>0)}\E^{\mathcal{E}}\Big[\frac{\mathcal{A}^k(\mathcal{D}_n,f)}{\mathcal{A}^k(\mathcal{D}_n,1)}\un_{\{|\Delta^k(\mathcal{D}_n)|>0\}}\Big],
\end{align}
and $\bm{\mu}^n(f)=0$ otherwise. In particular, if $\P^*(|\Delta^k(\mathcal{D}_n)|>0)>0$, then
\begin{align}\label{LawVertices1}
    \Eb^*\big[\bm{\mu}^n(\un_{\{\Vect{x}\}})\big]=\frac{1}{\P^*(|\Delta^k(\mathcal{D}_n)|>0)}\E^*\Big[\frac{\un_{\{\Vect{x}\in\Delta^k(\mathcal{D}_n)\}}}{|\Delta^k(\mathcal{D}_n)|}\un_{\{|\Delta^k(\mathcal{D}_n)|>0\}}\Big],
\end{align}
and $\Eb^*[\bm{\mu}^n(\un_{\{\cdot\}})]$ can be seen as the (annealed) law of a $k$-tuple $\VectCoord{\mathcal{X}}{n}=(\VectCoord{\mathcal{X}}{1,n},\ldots,\VectCoord{\mathcal{X}}{k,n})$ picked uniformly in the set $\Delta^k(\mathcal{D}_n)$, conditionally on the event $\{|\Delta^k(\mathcal{D}_n)|>0\}$. In the particular case $\mathcal{D}_n=\{x\in\mathcal{R}_{T^{\sqrtBis{n}}};\; |x|=\ell_n\}$ for some sequence $(\ell_n)$, we have $|\Delta^k(\mathcal{D}_n)|=|\mathcal{D}_n|(|\mathcal{D}_n|-1)\times\cdots\times(|\mathcal{D}_n|-k+1)$ and the vertices $\VectCoord{\mathcal{X}}{1,n},\ldots,\VectCoord{\mathcal{X}}{k,n}$ are nothing but $k$ vertices picked uniformly and without replacement in the generation $\ell_n$ of the tree $\mathcal{R}_{T^{\sqrtBis{n}}}$, conditionally on the event $\{|\mathcal{D}_n|\geq k\}$. \\
Our main interest is the genealogy of the $k$ vertices $\VectCoord{\mathcal{X}}{1,n},\ldots,\VectCoord{\mathcal{X}}{k,n}$ so let us define the genealogical tree of these $k$ vertices. First, introduce the largest generation $M_n:=\max_{x\in\mathcal{D}_n}|x|$ of the set $\mathcal{D}_n$. Recall that in the diffusive regime (see \eqref{DiffCase} and \eqref{DefKappa} with $\kappa>2$), $\max_{x\in\mathcal{R}_{T^{\sqrtBis{n}}}}|x|$, the largest generation of the tree $\mathcal{R}_{T^{\sqrtBis{n}}}$, is of order $\sqrtBis{n}$ when $n\to\infty$. \\
If $|\Delta^k(\mathcal{D}_n)|>0$, then we define for any $m\in\{0,\ldots,M_n\}$ the equivalence relation $\sim_m$ on $\{1,\ldots,k\}$ by: $i_1\sim_m i_2$ if and only if $\VectCoord{\mathcal{X}}{i_1,n}$\textrm{ and }\; $\VectCoord{\mathcal{X}}{i_2,n}$\textrm{ share a common ancestor in generation }$m$. We denote by $\pi^{k,n}_m$ the partition of $\{1,\ldots,k\}$ whose blocks are given by the equivalent classes of the relation $\sim_m$. The process $\pi^{k,n}:=(\pi^{k,n}_m)_{0\leq m\leq M_n}$ is called the genealogical tree of $\VectCoord{\mathcal{X}}{1,n},\ldots,\VectCoord{\mathcal{X}}{k,n}$. Let $\VectCoord{\mathcal{G}}{i,n}=|\VectCoord{\mathcal{X}}{i,n}|$ be the generation of $\VectCoord{\mathcal{X}}{i,n}$. By definition, 
\begin{align*}
    \pi^{k,n}_0=\{\{1,\ldots,k\}\}\;\textrm{ and }\;\pi^{k,n}_{m}=\{\{1\},\ldots,\{k\}\}\;\textrm{ for any }\;m\in\Big\{\max_{1\leq i\leq k}\VectCoord{\mathcal{G}}{i,n},\ldots,M_n\Big\}.
\end{align*}
The generations at which the vertices $\VectCoord{\mathcal{X}}{1,n},\ldots, \VectCoord{\mathcal{X}}{k,n}$ are chosen have a major influence on their genealogical structure. The next three subsections are dedicated to the three regimes we observe: the tiny generations, the small generations, on which we spend most of our time and the critical generations. For the second regime, we provide a general result for the range $\mathcal{A}^k(\mathcal{D}_n,f)$ and we show that we can easily extend our results on $\mathcal{R}_{T^{\sqrtBis{n}}}$ to the range up to the time $n$, see subsection \ref{SmallGenerationsGR}. As corollaries, we then give a quite full description of the genealogy of $\VectCoord{\mathcal{X}}{1,n},\ldots, \VectCoord{\mathcal{X}}{k,n}$, displaying five examples we believe to be relevant, see subsection \ref{SmallGenerationsEx}. In particular, Corollary \ref{GenTh3} is dedicated to the study of the genealogical tree $\pi^{k,n}$ of the vertices $\VectCoord{\mathcal{X}}{1,n},\ldots,\VectCoord{\mathcal{X}}{k,n}$.

\vspace{0.2cm}

\noindent We end this subsection by introducing some related works on the genealogical tree of $k$ vertices sampled in random trees. Replacing $\mathcal{R}_{T^{\sqrtBis{n}}}$ by a regular Galton-Watson tree $\mathfrak{T}$ and $\mathcal{D}_n$ by $\{x\in\mathfrak{T};\; |x|=T\}$ (the $T$-th generation of $\mathfrak{T}$), the genealogy of $k$ vertices $\VectCoord{\mathcal{X}}{1}_T,\ldots,\VectCoord{\mathcal{X}}{k}_T$ uniformly chosen in $\{x\in\mathfrak{T};\; |x|=T\}$ has been deeply studied for fixed $T$ as well as for $T\to\infty$. First, when $k=2$, K.B. Athreya \cite{Athreya_SUPER_CT} proved that when $\mathfrak{T}$ is supercritical (the mean of the reproduction law in larger than $1$) $\VectCoord{\mathcal{X}}{1}_T$ and $\VectCoord{\mathcal{X}}{2}_T$ share a common ancestor for the last time in the remote past: if $\mathcal{M}_T:=|\VectCoord{\mathcal{X}}{1}_T\land\VectCoord{\mathcal{X}}{2}_T|$ denotes the generation of the most recent common ancestor of $\VectCoord{\mathcal{X}}{1}_T$ and $\VectCoord{\mathcal{X}}{2}_T$ then $(\mathcal{M}_T)$ converges in law to a non-negative random variable depending on the reproduction law $N$ when $T$ goes to $\infty$. However, when $\mathfrak{T}$ is critical (the mean of the reproduction law is equal to 1), $\VectCoord{\mathcal{X}}{1}_T$ and $\VectCoord{\mathcal{X}}{2}_T$ share a common ancestor for the last time in the recent past: $(\mathcal{M}_T/T)$ converges in law to a $[0,1]$-valued random variable which does not depend on the reproduction law $N$ when $T$ goes to $\infty$, see \cite{Athreya_SUB_CT}.  K.B. Athreya also dealt with the sub-critical case (the mean of the reproduction law is smaller than $1$) in the latter paper and it is quite similar to the critical case. More recently S. Harris, S. Johnston and M. Roberts gave a full description of the genealogy of the vertices $\VectCoord{\mathcal{X}}{1}_T,\ldots,\VectCoord{\mathcal{X}}{k}_T$ for a given integer $k\geq 2$ for both fixed $T$ and $T\to\infty$, when the underlying process is a continuous-time Galton-Watson process (see \cite{HarrisJohnstonRoberts1} and \cite{Johnston1}). See also \cite{AbrDel1} for a study of the genealogy of randomly chosen individuals when the underlying process is a continuous-state branching process. \\
Hence, the most notable differences between our model and regular Galton-Watson trees are: first the influence of the underlying random environment and second the influence of the generations where vertices are sampled.

\subsection{The small generations: general results}\label{SmallGenerationsGR}
In this section, we present results for the range $\mathcal{A}^k(\mathcal{D}_n,f)$ with $f$ non-negative and bounded satisfying a very natural heredity condition we will discuss later. First, let $(\mathfrak{L}_n)$ be a sequence of positive integers as in \eqref{SmallGenerations}, let $(\ell_n)$ be a sequence of positive integers such that $\delta_0^{-1}\log n\leq\ell_n\leq \mathfrak{L}_n$ and introduce the set
$$ \mathcal{D}_n:=\{x\in\mathcal{R}_{T^{\sqrtBis{n}}};\; \ell_n\leq |x|\leq \mathfrak{L}_n\}, $$
with height $\bm{L}_n-1$ where $\bm{L}_n:=\mathfrak{L}_n-\ell_n+1$. Note that $\lim_{n\to\infty}\P^*(|\Delta^k(\mathcal{D}_n)|>0)=1$ (see for instance \eqref{ConvA}) so we will refer to the event $\{|\Delta^k(\mathcal{D}_n)|>0\}$ only if necessary. For any $m\in\N$, recall that $\T_m=\{x\in\T;\; |x|=m\}$ be the $m$-th generation of the tree $\T$ and let $\Delta^j_m:=\Delta^j(\T_m)$. We make the following technical assumption.
\begin{Ass}\label{Assumption2}
There exists $\delta_1>0$ such that $\psi(t)<\infty$ for all $t\in[1-\delta_1,\lceil\kappa+\delta_1\rceil]$ and for all $1\leq j\leq\lceil\kappa+\delta_1\rceil$, for all $\bm{\beta}=(\beta_1,\ldots,\beta_j)\in\ProdSet{(\N^*)}{j}$ such that $\sum_{i=1}^j\beta_i\leq\lceil\kappa+\delta_1\rceil$
\begin{align}\label{MomentsJoints}
    c_j(\bm{\beta}):=\Eb\Big[\sum_{\Vect{x}\in\Delta_1^j}e^{-\langle\Vect{\beta},V(\Vect{x})\rangle_j}\Big]<\infty,
\end{align}
where $\langle\Vect{\beta},V(\Vect{x})\rangle_j:=\sum_{i=1}^j\beta_iV(\VectCoord{x}{i})$.
\end{Ass}
\noindent In addition, we also require the following ellipticity condition.
\begin{Ass}\label{AssumptionIncrements}
There exists $\mathfrak{h}>0$ such that
\begin{align}
    \Pb\big(\inf_{x\in\T}(V(x)-V(x^*))\geq -\mathfrak{h}\big)=1.
\end{align}
\end{Ass}
\noindent Although we obtain quite general results, we however require an assumption on $f$. For any $\Vect{x}=(\VectCoord{x}{1},\ldots,\VectCoord{x}{k})\in\Delta^k$, let $\mathcal{S}^k(\Vect{x})$ be the first generation at which none of $\VectCoord{x}{1},\ldots,\VectCoord{x}{k}$ share a common ancestor:
\begin{align}\label{MultiMRCA}
    \mathcal{S}^k(\Vect{x}):=\min\{m\geq 1;\; \forall\;i_1\not=i_2, \; |\VectCoord{x}{i_1}\land\VectCoord{x}{i_2}|<m\},
\end{align}
where we recall that $\VectCoord{x}{i_1}\land\VectCoord{x}{i_2}$ is the most recent common ancestor of $\VectCoord{x}{i_1}$ and $\VectCoord{x}{i_2}$. For any $m\in\N^*$, introduce $\mathcal{C}^k_m:=\{\Vect{x}\in\Delta^k;\; \mathcal{S}^k(\Vect{x})\leq m\}$ (see Figure \ref{Coalescence times}). Assume the following:
\begin{Ass}\label{Assumption3}
There exists $\mathfrak{g}\in\N^*$ such that for all integer $p\geq\mathfrak{g}$ and all $\Vect{x}=(\VectCoord{x}{1},\ldots,\VectCoord{x}{k})\in\Delta^k$, if $\min_{1\leq i\leq k}|\VectCoord{x}{i}|\geq p$ and $\Vect{x}\in\mathcal{C}^k_{p}$ then
\begin{align}
    f((\VectCoord{x}{1},\ldots,\VectCoord{x}{k}))=f\big((\VectCoord{x}{1})_p,\ldots,(\VectCoord{x}{k})_p\big),
\end{align}
\end{Ass}
\noindent where we recall that $(\VectCoord{x}{i})_{p}$ is the ancestor of $\VectCoord{x}{i}$ in the generation $p$. In other words, we ask the constraint $f$ to be hereditary from a given generation $\mathfrak{g}$. By convention, a function $f$ satisfying Assumption \ref{Assumption3} is assumed to be non-negative and bounded.

\vspace{0.2cm}

\noindent Now, introduce $\mathcal{A}^k_l(f,\Vect{\beta}):=\sum_{\Vect{x}\in\Delta^k_l}f(\Vect{x})e^{-\langle\Vect{\beta},V(\Vect{x})\rangle_{k}}$, $\mathcal{A}_l(f):=\mathcal{A}_l(f,\Vect{1})$ and $\Vect{1}:=(1,\ldots,1)\in\N^{\times k}$. Note, by Assumption \ref{Assumption2}, that $(\mathcal{A}_l(f))_{l\geq 1}$ is bounded in $L^2(\Pb^*)$. Then, define
\begin{align}\label{Def_A_Inf_Th}
    \mathcal{A}^k_{\infty}(f):=\lim_{l\to\infty}\mathcal{A}^k_l(f),
\end{align}
where the convergence holds in $L^2(\Pb^*)$. We prove that this limit exists for any function $f$ satisfying Assumption \ref{Assumption3} in section \ref{QuasiMartingale}. \\
Let us introduce a few more definitions. Let $(S_{j}-S_{j-1})_{j\in\N^*}$ be a sequence of\textrm{ i.i.d } real-valued random variables under $\Pb$ such that $S_0=0$ and for any bounded and measurable function $\mathfrak{t}:\R\longrightarrow\R$
\begin{align}\label{RW}
    \Eb[\mathfrak{t}(S_1)]=\Eb\Big[\sum_{|x|=1}\mathfrak{t}(V(x))e^{-V(x)}\Big].
\end{align}
Note that by Assumption \ref{Assumption1}, $\Eb[S_1]>0$. Introduce $c_{\infty}:=\Eb[(\sum_{j\geq 0}e^{-S_j})^{-1}]$. We are now ready to state our first result.
\begin{theo}\label{GenTh5}
Let $k\geq 2$ and assume $\kappa>2k$. Under the Assumptions \ref{Assumption1}, \ref{AssumptionSmallGenerations}, \ref{Assumption2} and \ref{AssumptionIncrements}, if $f$ satisfies the hereditary Assumption \ref{Assumption3} then, in $\P^*$-probability 
\begin{align}\label{ConvA}
    \frac{\mathcal{A}^k(\mathcal{D}_n,f)}{(n^{1/2}\bm{L}_n)^k}&\underset{n\to\infty}{\longrightarrow}(c_{\infty})^k\mathcal{A}^k_{\infty}(f),
\end{align}
and if $g\not\equiv 0$ also satisfies Assumption \ref{Assumption3} then, in $\P^*$-probability 
\begin{align}\label{ConvAQuotient}
    \frac{\mathcal{A}^k(\mathcal{D}_n,f)}{\mathcal{A}^k(\mathcal{D}_n,g)}\un_{\{|\Delta^k(\mathcal{D}_n)|>0\}}&\underset{n\to\infty}{\longrightarrow}\frac{\mathcal{A}^k_{\infty}(f)}{\mathcal{A}^k_{\infty}(g)},
\end{align}
where $\bm{L}_n=\mathfrak{L}_n-\ell_n+1$. Note that a constraint satisfying Assumption \ref{Assumption3} does not have any influence on the normalization of the range. Moreover, $\mathcal{A}^k(\mathcal{D}_n,f)$ behaves like $(\bm{L}_n\max_{x\in\mathcal{R}_{T^{\sqrtBis{n}}}}|x|)^k$ and the limiting value $\mathcal{A}_{\infty}(f)$ contains all the information about the interactions between the vertices of the tree.
\end{theo}

\vspace{0.2cm}

\noindent Let us state an extension of Theorem \ref{GenTh5} to the range $\mathcal{R}_n$, the sub-tree of vertices visited by the random walk $\X$ up to time $n$. Before that, introduce $\Tilde{\mathcal{D}}_n:=\{x\in\mathcal{R}_n;\; \ell_n\leq|x|\leq\mathfrak{L}_n\}$. Let us now define $W_n:=\sum_{|x|=n}e^{-V(x)}$ and $\mathcal{F}_n:=\sigma((x,V(x));|x|\leq n)$. The random process $(W_n)_{n\geq 0}$ is a non-negative $(\mathcal{F}_n)_{n\geq 0}$-martingale, usually referred to as the additive martingale. It is known that under the Assumptions \ref{Assumption1} and \ref{Assumption2}, $\Pb(W_{\infty}>0)>0$, see \cite{Biggins1977}, \cite{Lyons1997}, \cite{Liu1} or \cite{Alsmeyer_Iksanov} for instance. Moreover, it is claimed in \cite{Biggins1977} that $\Pb$-almost surely, the event $\{W_{\infty}>0\}$ coincides with the event of non extinction of the underlying Galton-Watson tree $\T$. In particular, $\Pb^*(W_{\infty}>0)=1$.

\begin{theo}\label{GenTh7}
Let $k\geq 2$. There exists a non-increasing sequence of positive integers $(\mathfrak{q}(j))_j$, satisfying $\mathfrak{q}(j)\in(0,1/2)$ and $\mathfrak{q}(j)\to 0$ when $j\to\infty$ such that if $\kappa>2\xi k$ for some integer $\xi\geq 2$ and $\mathfrak{L}_n=o(n^{1/2-\mathfrak{q}({\xi})})$, then, in law, under $\P^*$
\begin{align}\label{ConvAbis}
    \frac{\mathcal{A}^k(\Tilde{\mathcal{D}}_n,f)}{(n^{1/2}\bm{L}_n)^k}&\underset{n\to\infty}{\longrightarrow}\frac{\mathcal{A}^k_{\infty}(f)}{(W_{\infty})^k}\big(c_{\infty}c_{0}^{1/2}|\mathcal{N}|\big)^k,
\end{align}
where $c_0:=\Eb[\sum_{x\not=y;|x|=|y|=1}e^{-V(x)-V(y)}]/(1-e^{\psi(2)})$ and $\mathcal{N}$ is a standard Gaussian random variable independent of the environment. Moreover, if $g\not\equiv 0$ also satisfies Assumption \ref{Assumption3} then in, $\P^*$-probability 
\begin{align}\label{ConvAQuotientbis}
    \frac{\mathcal{A}^k(\Tilde{\mathcal{D}}_n,f)}{\mathcal{A}^k(\Tilde{\mathcal{D}}_n,g)}\un_{\{|\Delta^k(\Tilde{\mathcal{D}}_n)|>0\}}&\underset{n\to\infty}{\longrightarrow}\frac{\mathcal{A}^k_{\infty}(f)}{\mathcal{A}^k_{\infty}(g)}.
\end{align}
\end{theo}

\subsection{The small generations: examples}\label{SmallGenerationsEx}

As promised, we now present five explicit examples that illustrate our general results. The two first examples are simple. Corollary \ref{GenTh1} is an asymptotic of the volume of small generations and Corollary \ref{GenTh1Bis} aims to convince ourselves that the hereditary Assumption \ref{Assumption3} is not really restrictive. The last three examples, Corollaries \ref{GenTh2}-\ref{GenTh4}, are definitely more important since they give a good description of the genealogical tree of vertices picked uniformly in the range $\mathcal{R}_{T^{\sqrtBis{n}}}$. 

\vspace{0.2cm}

\noindent Let $R_n(l):=\sum_{|z|=l}\un_{\{z\in\mathcal{R}_{T^{\sqrtBis{n}}}\}}$ be the volume of the $\ell$-th generation of the range $\mathcal{R}_{T^{\sqrtBis{n}}}$.

\begin{coro}[\textbf{Volume of small generations}]\label{GenTh1}
Let $\kappa>2$. Under the Assumptions \ref{Assumption1}, \ref{AssumptionSmallGenerations}, \ref{Assumption2} and \ref{AssumptionIncrements}, in $\P^*$-probability
\begin{align*}
    \frac{1}{\sqrtBis{n}\bm{L}_n}\sum_{l=\ell_n}^{\mathfrak{L}_n}R_n(l)\underset{n\to\infty}{\longrightarrow}c_{\infty}W_{\infty},
\end{align*}
As a consequence of Corollary \ref{GenTh1}, we have that if $\log n=o(\mathfrak{L}_n)$, then both $(R_n(\mathfrak{L}_n)/\sqrtBis{n})$ and $(\sum_{l=\ell_n}^{\mathfrak{L}_n}R_n(l)/(\sqrtBis{n}\mathfrak{L}_n))$ converge in $ \P^*$-probability to the same limit $c_{\infty}W_{\infty}$.
\end{coro}
\begin{remark}
    Although we assumed $k\geq 2$, the case $k=1$ is interesting. The convergence of $(\sum_{l=\ell_n}^{\mathfrak{L}_n}R_n(l)/(\sqrtBis{n}\bm{L}_n))_n$ does not require all the previous assumptions and holds for $\kappa>2$. Corollary \ref{GenTh1} can be seen as an easy consequence of Theorem \ref{GenTh5} with $k=2$ and $f=1$. 
\end{remark}

\vspace{0.2cm}

\noindent In view of Corollary \ref{GenTh1}, we deduce that whenever $\mathfrak{L}_n$ is large enough but not to close to the largest generation of the tree $\mathcal{R}_{T^{\sqrtBis{n}}}$, the range $R_n(\mathfrak{L}_n)$ is of order $n^{1/2}$. Moreover, $\bm{L}_n-1$ denotes the height of the set $\mathcal{D}_n$ in the tree $\mathcal{R}_{T^{\sqrtBis{n}}}$ and the volume of $\mathcal{D}_n$ behaves like $\bm{L}_n\times R_n(\mathfrak{L}_n)$.

\vspace{0.2cm}

\noindent The following corollaries are composed of two parts: the first part will be a convergence of the range $\mathcal{A}^k(\mathcal{D}_n,f)$ for a given function $f$ and the second part will be an application of this convergence to the genealogy of the vertices $\VectCoord{\mathcal{X}}{1,n},\ldots,\VectCoord{\mathcal{X}}{k,n}$. 

\vspace{0.2cm}

\noindent In the second example, we present a range such that for a $k$-tuple $\Vect{x}\in\Delta^k$, some of the vertices are free while others are obliged to interact with each other. Let $\bm{\lambda}=(\lambda_2,\ldots,\lambda_k)\in\ProdSet{(\N^*)}{(k-1)}$ and introduce
\begin{align*}
    f_{\bm{\lambda}}(\VectCoord{x}{1},\ldots,\VectCoord{x}{k}):=\prod_{i=2}^k\un_{\{|\VectCoord{x}{i-1}\land\VectCoord{x}{i}|<\lambda_i\}}.
\end{align*}
Note that there is no constraint between $\VectCoord{x}{i_1}$ and $\VectCoord{x}{i_2}$ if $i_2\not\in\{i_1-1,i_1+1\}$, $i_1\geq 2$.
\begin{coro}[\textbf{A constraint for consecutive vertices}]\label{GenTh1Bis}
    Let $k\geq 2$ and assume $\kappa>2k$. Under the Assumptions \ref{Assumption1}, \ref{AssumptionSmallGenerations}, \ref{Assumption2} and \ref{AssumptionIncrements}, in $\P^*$-probability
    \begin{align*}
        \frac{1}{(\sqrtBis{n}\bm{L}_n)^k}\mathcal{A}^k(\mathcal{D}_n,f_{\bm{\lambda}})\underset{n\to\infty}{\longrightarrow}(c_{\infty})^k\mathcal{A}^k_{\infty}(f_{\bm{\lambda}}),
    \end{align*}
    where $\mathcal{A}^k_{\infty}(f_{\bm{\lambda}})=\lim_{l\to\infty}\sum_{\Vect{x}\in\Delta^k_l}e^{-V(\VectCoord{x}{1})}\prod_{i=2}^ke^{-V(\VectCoord{x}{i})}\un_{\{|\VectCoord{x}{i-1}\land\VectCoord{x}{i}|<\lambda_i\}}$ and this limit holds in $L^2(\Pb^*)$, see \eqref{Def_A_Inf_Th}.
\end{coro}
\noindent In the next example, we are interested in the number of $k$-tuples of distinct vertices of $\mathcal{D}_n$ such that any most recent common ancestor of two vertices among them is located close to the root of $\mathcal{R}_{T^{\sqrtBis{n}}}$. Recall that for all $k\geq 2$, $\Vect{x}=(\VectCoord{x}{1},\ldots,\VectCoord{x}{k})\in\Delta^k$, $\mathcal{C}^k_m=\{\Vect{x}\in\Delta^k;\; \mathcal{S}^k(\Vect{x})\leq m\}$ where $\mathcal{S}^k(\Vect{x})-1$ denotes the last generation at which two or more vertices among $\VectCoord{x}{1},\ldots,\VectCoord{x}{k}$ share a  common ancestor, see \eqref{MultiMRCA}. It turns out that the number of vertices visited by the random walk $\X$ belonging to $\mathcal{C}^k_m$ for any $m\in\N^*$ is large and as a consequence, the sequence of random times $(\mathcal{S}^k(\VectCoord{\mathcal{X}}{n})$ converges in law.
\begin{coro}[\textbf{First coalescent time}]\label{GenTh2}
Let $k\geq 2$. Assume that $\kappa>2k$ and for any $m\in\N^*$, $\Vect{x}\in\Delta^k$, $f_m(\Vect{x})=\un_{\mathcal{C}^k_m}(\Vect{x})$. Recall that $\mathcal{A}^k(\mathcal{D}_n,f_m)$ is the number of $k$-tuples $\Vect{x}$ of distinct vertices of $\mathcal{D}_n$ such that $\mathcal{S}^k(\Vect{x})\leq m$. Under the Assumptions \ref{Assumption1}, \ref{AssumptionSmallGenerations}, \ref{Assumption2} and \ref{AssumptionIncrements}
\begin{enumerate}
    \item in $\P^*$-probability
        \begin{align*}%\label{Th2Trace}
            \frac{1}{(\sqrtBis{n}\bm{L}_n)^k}\mathcal{A}^k(\mathcal{D}_n,f_m)\underset{n\to\infty}{\longrightarrow}(c_{\infty})^k\mathcal{A}^k_{\infty}(f_m),
        \end{align*}
        where $\mathcal{A}^k_{\infty}(f_m)$ is an explicit random variable (see \eqref{Def_A_Inf_Th}) such that $\lim_{m\to\infty}\mathcal{A}_{\infty}^k(f_m)=(W_{\infty})^k$ in $L^2(\Pb^*)$. \\
    \item Moreover, the sequence of random times $(\mathcal{S}^k(\VectCoord{\mathcal{X}}{n}))$ converges in law, under $\P^*$: for any $m\in\N^*$
    \begin{align}\label{ConvMRCA}
        \P^*\big(\mathcal{S}^k(\VectCoord{\mathcal{X}}{n})\leq m\big)\underset{n\to\infty}{\longrightarrow}\Eb^*\Big[\frac{\mathcal{A}^k_{\infty}(f_m)}{(W_{\infty})^k}\Big].
    \end{align}
\end{enumerate}
\end{coro}
\noindent The convergence in \eqref{ConvMRCA} is somewhat reminiscent of the result of K.B Athreya (\cite{Athreya_SUPER_CT}, Theorem 2) for a supercritical Galton-Watson tree stated earlier: each coalescence occurs in a generation close to the root.

\vspace{0.2cm}

\noindent In the following result, we compute the law of $\pi^{k,n}$. Recall that if $|\Delta^k(\mathcal{D}_n)|>0$, then for any $m\in\{0,\ldots,M_n\}$ (where $M_n=\max_{x\in\mathcal{D}_n}|x|$), $\pi^{k,n}_m$ is the partition of $\{1,\ldots,k\}$ whose blocks are given by the equivalent classes of the relation $\sim_m$ defined by: $i_1\sim_m i_2$ if and only if $\VectCoord{\mathcal{X}}{i_1,n}$\textrm{ and }\; $\VectCoord{\mathcal{X}}{i_2,n}$\textrm{ share a common ancestor in generation }$m$. We have $\pi^{k,n}=(\pi^{k,n}_m)_{0\leq m\leq M_n}$. Also recall that, by definition
\begin{align*}
    \pi^{k,n}_0=\{\{1,\ldots,k\}\}\;\textrm{ and }\;\pi^{k,n}_{m}=\{\{1\},\ldots,\{k\}\}\;\textrm{ for any }\;m\in\Big\{\max_{1\leq i\leq k}\big|\VectCoord{\mathcal{X}}{i,n}\big|,\ldots,M_n\Big\}.
\end{align*}
Before stating our next result, we add, for convenience, a collection $\{\VectCoord{e}{i};\; i\in\N^*\}$ of distinct leafs in the generation $0$. Let $q\geq 2$ be an integer and $\bm{\pi}$ be a partition of $\{1,\ldots,q\}$. We denote by $|\bm{\pi}|$ the total number of blocks of $\bm{\pi}$. For any $m\in\N^*$, define the set $\Upsilon_{m,\bm{\pi}}$ by: $\Vect{x}=(\VectCoord{x}{1},\ldots,\VectCoord{x}{q})\in\Upsilon_{m,\bm{\pi}}$ if and only if $\Vect{x}\in\Delta^q$ and 
\begin{align*}
   \forall\bm{B}\in\bm{\pi}, \forall i_1,i_2\in\bm{B}:(\VectCoord{x}{i_1})_{m}=(\VectCoord{x}{i_2})_{m},
\end{align*}
and if $|\bm{\pi}|\geq 2$
\begin{align*}
    \forall\bm{B}\not=\tilde{\bm{B}}\in\bm{\pi}, \forall i_1\in\bm{B}, i_2\in\tilde{\bm{B}}:(\VectCoord{x}{i_1})_{m}\not=(\VectCoord{x}{i_2})_{m},
\end{align*}
where we recall that, when $|\VectCoord{x}{i}|\geq m$, $(\VectCoord{x}{i})_m$ denotes the ancestor of $\VectCoord{x}{i}$ in generation $m$. Otherwise, if $|\VectCoord{x}{i}|<m$, we set $(\VectCoord{x}{i})_m:=\VectCoord{e}{i}$ so $\Upsilon_{m,\bm{\pi}}$ is well defined. 

\vspace{0.2cm}

\begin{figure}[h]
\centering

   \begin{tikzpicture}

        %%%%%%%% Gen 0 %%%%%%%%
        \draw[black,fill] (-1,0) circle [radius=0.02];
        \filldraw (-1,0) node[anchor=west] {$e$};

        \draw[black,fill] (-2,0) circle [radius=0.02];
        \filldraw (-2,0) node[anchor=west] {$\VectCoord{e}{1}$};

        \draw[black,fill] (-3,0) circle [radius=0.02];
        \filldraw (-3,0) node[anchor=west] {$\VectCoord{e}{2}$};

        \draw[black,fill] (-4,0) circle [radius=0.02];
        \filldraw (-4,0) node[anchor=west] {$\VectCoord{e}{3}$};

        \draw[DarkOrchid,fill] (-5,0) circle [radius=0.02];
        \filldraw[DarkOrchid] (-6.54,-0.2) node[anchor=west] {$(\VectCoord{x}{4})_{m'}=\VectCoord{e}{4}$};
        
        %%%%%%%%%%%%%%%%%%%%%%%%
            
        %%%%%%%% Gen Inter 1 %%%%%%%%
        %\draw[dotted] (0,0)--(-0.4,0.5);
        %\draw[black,fill] (-0.4,0.5) circle [radius=0.02];
        
        %\draw[dotted] (0,0)--(0,0.5);
        %\draw[black,fill] (0,0.5) circle [radius=0.02];
        
        %\draw[dotted] (0,0)--(0.4,0.5);
        %\draw[black,fill] (0.4,0.5) circle [radius=0.02];
        %%%%%%%%%%%%%%%%%%%%%%%%

         %%%%%%%% Gen 1 %%%%%%%%
        \draw (-1,0) -- (-1,2);
        \draw[black,fill] (-1,2) circle [radius=0.02];
        
        %%%%%%%%%%%%%%%%%%%%%%%%

        %%%%%%%% Gen 2 %%%%%%%%
        \draw (-1,2) -- (0,3.5);
        \draw[black,fill] (0,3.5) circle [radius=0.02];

        %%%%%%%%%%%%%%%%%%%%%%%%

        %%%%%%%% Gen 3 %%%%%%%%
        \draw (-1,2) -- (-2.5,4.5);
        \draw[BurntOrange,fill] (-2.5,4.5) circle [radius=0.02];
        %%%%%%%%%%%%%%%%%%%%%%%%

        %%%%%%%% Gen 3 to gen m %%%%%%%%
        \draw[BurntOrange] (-2.5,4.5) -- (-1.42,2.7) ;
        \draw[BurntOrange,fill] (-2.5,4.5) circle [radius=0.02];

        \draw[DarkOrchid] (0,3.5) -- (-0.53,2.7);
        \draw[DarkOrchid,fill] (0,3.5) circle [radius=0.02];

        \draw[DarkOrchid,fill] (-0.53,2.7) circle [radius=0.02];
        \draw[BurntOrange,fill] (-1.42,2.7) circle [radius=0.02];

        %%%%%%%%%%%%%%%%%%%%%%%% 

        \draw[BurntOrange] (-2.5,4.5) -- (-3.7,7);
        \draw[BurntOrange,fill] (-3.7,7) circle [radius=0.02];

        \draw[BurntOrange] (-2.5,4.5) -- (-2.2,6);
        \draw[BurntOrange,fill] (-2.2,6) circle [radius=0.02];

        \draw[DarkOrchid] (0,3.5) -- (-0.5,4.2);
        \draw[DarkOrchid,fill] (-0.5,4.2) circle [radius=0.02];

        \draw[DarkOrchid] (0,3.5) -- (2.1,6.85);
        \draw[DarkOrchid,fill] (2.1 ,6.85) circle [radius=0.02];

        %%%%%%%% Ligne gen + flèche %%%%%%%%        

        \filldraw (3.5,0) node[anchor=west] {$0$};
        
        \draw[dotted] (-3.7,2.7) -- (3.5,2.7);
        \filldraw (3.5,2.7) node[anchor=west] {$m$};

        \draw[dotted] (-3.7,4.3) -- (3.5,4.3);
        \filldraw (3.5,4.3) node[anchor=west] {$m'$};

        \draw[DarkOrchid,fill] (-0.53,2.7) circle [radius=0.02];
        \draw[BurntOrange,fill] (-1.42,2.7) circle [radius=0.02];

        \draw[latex-] (3.5,8) -- (3.5,0);

        \filldraw (3.5,8) node[anchor=west] {\textrm{ generation }};

        %%%%%%%%%%%%%%%%%%%%%%%%%%%%%%%%%%%%%%%%
        
        %%%%%%%% sommets x et z %%%%%%%%

        \filldraw[BurntOrange] (-3.7,7.2) node[anchor=west] {$\VectCoord{x}{3}$};

        \filldraw[BurntOrange] (-2.2,6.2) node[anchor=west] {$\VectCoord{x}{1}$};

        \filldraw[DarkOrchid] (-0.5,4.5) node[anchor=west] {$\VectCoord{x}{4}$};
        
        \filldraw[DarkOrchid] (2.1,7.05) node[anchor=west] {$\VectCoord{x}{2}$};

        \filldraw[BurntOrange] (-1.42,2.9) node[anchor=west] {$\VectCoord{z}{1}$};
        
        \filldraw[DarkOrchid] (-0.53,2.9) node[anchor=west] {$\VectCoord{z}{2}$};
        
        %%%%%%%%%%%%%%%%%%%%%%%%%%%%%%%%%%%%%%%%

        \end{tikzpicture}

\label{Coalescence times 1}
\caption{In the present illustration, the $4$-tuple of vertices $(\textcolor{BurntOrange}{\VectCoord{x}{1}},\textcolor{DarkOrchid}{\VectCoord{x}{2}},\textcolor{BurntOrange}{\VectCoord{x}{3}}\textcolor{DarkOrchid}{\VectCoord{x}{4}})$ belongs to $\Upsilon_{m,\bm{\pi}}$ with $\bm{\pi}=\{\{1,3\},\{2,4\}\}$, since $\textcolor{BurntOrange}{\VectCoord{z}{1}}=\textcolor{BurntOrange}{(\VectCoord{x}{1})_m}=\textcolor{BurntOrange}{(\VectCoord{x}{3})_m}, \textcolor{DarkOrchid}{\VectCoord{z}{2}}=\textcolor{DarkOrchid}{(\VectCoord{x}{2})_m}=\textcolor{DarkOrchid}{(\VectCoord{x}{4})_m}$ and $\textcolor{BurntOrange}{\VectCoord{z}{1}}\not=\textcolor{DarkOrchid}{\VectCoord{z}{2}}$. However, it does not belong to $\Upsilon_{m',\bm{\pi}}.$}
\end{figure}

\newpage

Now, let $1\leq d<q$ be two integers. A collection $\Xi:=(\bm{\Xi}_i)_{0\leq i\leq d}$ of partitions of $\{1,\ldots,q\}$ is said to be increasing if it satisfies $\bm{\Xi}_0=\{\{1,\ldots,q\}\}$, $\bm{\Xi}_d=\{\{1\},\ldots,\{q\}\}$ and for all $i\in\{1,\ldots,d\}$, $|\bm{\Xi}_{i-1}|<|\bm{\Xi}_i|$, where we recall that $|\bm{\Xi}_i|$ is the total number of blocks of the partition $\bm{\Xi}_i$. \\
Let $\Xi:=(\bm{\Xi}_i)_{0\leq i\leq d}$ be an increasing collection of partitions of $\{1,\ldots,q\}$ and let $\bm{t}=(t_1,\ldots,t_{d})\in\ProdSet{\N}{d}$ such that $t_1<\cdots<t_{d}$. Introduce the set $\Gamma^i_{\bm{t},\Xi}:=\Upsilon_{t_i-1,\bm{\Xi}_{i-1}}\cap\Upsilon_{t_i,\bm{\Xi}_i}$. We then define the function $f^{d}_{\bm{t},\Xi}$ by: for all $\Vect{x}\in\Delta^q$
\begin{align}\label{Def_f_partition}
    f^{d}_{\bm{t},\Xi}(\Vect{x})=\prod_{i=1}^{d}\un_{\Gamma^i_{\bm{t},\Xi}}(\Vect{x}).
\end{align}
The function defined in \eqref{Def_f_partition} plays a key role in our study: $f^{d}_{\bm{t},\Xi}(\Vect{x})$ characterizes the genealogy of $\Vect{x}:=(\VectCoord{x}{1},\ldots,\VectCoord{x}{q})$. Indeed, for any $i\in\{1,\ldots,d\}$, the partition $\bm{\Xi}_i$ corresponds to the $i$-th generation of the genealogical tree of $\VectCoord{x}{1},\ldots,\VectCoord{x}{q}$ while $t_i-1$ denotes the $i$-th generation at which at least two branches of this genealogical tree split ($t_i-1$ therefore corresponds to a coalescent/split time, see Figure \ref{Genealogical tree 2} for instance). \\
Before stating our result, we need a few more notation and definitions. Let $\Xi:=(\bm{\Xi}_i)_{0\leq i\leq d}$ be an increasing collection of partitions of $\{1,\ldots,q\}$. For $p\in\{1,\ldots, d\}$ , the $j$-th block $\bm{B}^{p-1}_j$ of the partition $\bm{\Xi}_{p-1}$ (blocks are ordered by their least element) is the union of $b_{p-1}(\bm{B}^{p-1}_j)\geq 1$ (we will write $b_{p-1}(\bm{B}_j)$ instead) block(s) $\bm{B}^{p}_{l_1},\ldots,\bm{B}^{p}_{l_{b_{p-1}(\bm{B}_j)}}$, $1\leq l_1<\ldots<l_{b_{p-1}(\bm{B}_j)}\leq|\bm{\Xi}_{p}|$, of the partition $\bm{\Xi}_{p}$ and for any $i\in\{1,\ldots,{b_{p-1}(\bm{B}_j)}\}$, define
\begin{align}\label{betaDef}
    \beta^{p-1}_{j,i}:=|\bm{B}^{p}_{l_i}|,   
\end{align}
to be the cardinal of the block $\bm{B}^{p}_{l_i}$. We are now ready to state our result:
\begin{coro}[\textbf{Full genealogy}]\label{GenTh3}
Let $k\geq 2$ and assume that $\kappa>2k$. Under the Assumptions \ref{Assumption1}, \ref{AssumptionSmallGenerations}, \ref{Assumption2} and \ref{AssumptionIncrements}, for any $\ell\in\N^*$ such that $\ell<k$, any $\bm{s}=(s_1,\ldots,s_{\ell})\in\ProdSet{\N}{\ell}$ such that $s_1<\cdots<s_{\ell}$ and any increasing collection $\Pi=(\bm{\pi}_i)_{0\leq i\leq\ell}$ of partitions of $\{1,\ldots,k\}$
\begin{enumerate}
    \item in $\P^*$-probability
    \begin{align}\label{Th3Trace}
        \frac{1}{(\sqrtBis{n}\bm{L}_n)^k}\mathcal{A}^k(\mathcal{D}_n,f^{\ell}_{\bm{s},\Pi})\underset{n\to\infty}{\longrightarrow}(c_{\infty})^k\mathcal{A}^k_{\infty}(f^{\ell}_{\bm{s},\Pi}),
        \end{align}
        where $(\mathcal{A}^k_{\infty}(f^{d}_{\bm{t},\Xi});\; 1\leq d<k,\; \Xi\textrm{ increasing },\; \bm{t}=(t_1,\ldots,t_d))$ is a collection of random variables satisfying
        \begin{align*}
            \sum_{d=1}^{k-1}\;\;\sum_{\Xi\textrm{ increasing }}\sum_{m_0<\cdots<m_d}\underset{m_{i-1}<t_i\leq m_i}{\sum_{\bm{t}=(t_1,\ldots,t_{d})}}\mathcal{A}^k_{\infty}(f^{d}_{\bm{t},\Xi})=(W_{\infty})^k,
        \end{align*}
    where $\Xi\textrm{ increasing}$ means here that $\Xi=(\bm{\Xi}_i)_{0\leq i\leq d}$ is an increasing collection of partitions of $\{1,\ldots,k\}$. Moreover
    \begin{align}\label{EspPartionLaw}
         \Eb\big[\mathcal{A}_{\infty}^k(f^{\ell}_{\bm{s},\Pi})\big]=e^{\psi(k)}\prod_{i=1}^{\ell}\prod_{j=1}^{|\bm{\pi}_{i-1}|}c_{b_{i-1}(\bm{B}_j)}(\bm{\beta}^{i-1}_j)\underset{|\mathfrak{B}|\geq 2}{\prod_{\mathfrak{B}\in\bm{\pi}_{i}}}e^{s_{i+1}^*\psi(|\mathfrak{B}|)},
    \end{align}
    with $s_{i+1}^*=s_{i+1}-s_i-1$, $s^*_{\ell+1}=1$, $\bm{\beta}^p_j:=(\beta^p_{j,1},\ldots,\beta^p_{j,b_p(\bm{B}_j)})$, see \eqref{betaDef}. We also use the convention $\prod_{\varnothing}=1$ and see the Assumption \ref{Assumption2} for the definition of $c_l(\bm{\beta})$. \\
    \item Moreover, for any non-negative integers $m_0<m_1<\cdots<m_{\ell}$
    \begin{align}\label{Th3partition}
       \P^*(\pi^{k,n}_{m_0}=\bm{\pi}_0,\ldots,\pi^{k,n}_{m_{\ell}}=\bm{\pi}_{\ell})
       \underset{n\to\infty}{\longrightarrow}\Eb^*\Big[\frac{1}{(W_{\infty})^k}\underset{m_{i-1}<s_i\leq m_i}{\sum_{\bm{s}=(s_1,\ldots,s_{\ell})}}\mathcal{A}^k_{\infty}(f^{\ell}_{\bm{s},\Pi})\Big]. 
    \end{align}

\end{enumerate}
\end{coro}
\begin{remark}[An hereditary character]
There is an hereditary character hidden in the previous formula \eqref{EspPartionLaw} due to the random environment. The fact is, unlike the case of regular supercritical Galton-Watson trees depending on $(b_i(\bm{B});\; \bm{B}\in\bm{\pi}_i,\; 0\leq i\leq\ell-1)$ (see \cite{Johnston1}, Theorem 3.5), the limit law of the present genealogical tree depends on the collection $(\bm{\beta}^i_j;\; 0\leq i\leq\ell-1,\; 1\leq j\leq|\bm{\pi}_{i-1}|)$ and on $(|\mathfrak{B}|;\;\mathfrak{B}\in\bm{\pi}_i,\; 1\leq i\leq\ell)$, making a huge difference. Indeed, by definition, the latter take more account of the genealogical structure than $(b_i(\bm{B});\; \bm{B}\in\bm{\pi}_i,\; 0\leq i\leq\ell-1)$. For instance, let $k=4$, $\ell=3$ and define the increasing collection of partitions $\Pi=(\bm{\pi}_i)_{0\leq i\leq\ell}$ by $\bm{\pi}_3=\{\{1\},\{2\},\{3\},\{4\}\}$, $\bm{\pi}_2=\{\{1,3\},\{2\},\{4\}\}$, $\bm{\pi}_1=\{\{1,3\},\{2,4\}\}$ and $\bm{\pi}_0=\{1,2,3,4\}$. We have $\bm{\beta}^2_1=(1,1)$, $\bm{\beta}^2_2=1$, $\bm{\beta}^2_1=1$; $\bm{\beta}^1_1=2$, $\bm{\beta}^1_2=(1,1)$; $\bm{\beta}^0_1=(2,2)$ and thanks to \eqref{EspPartionLaw}, for any $\bm{t}=(t_1,t_2,t_3)\in\ProdSet{\N}{3}$ such that $t_1<t_2<t_3$
\begin{align*}
     &\Eb[\mathcal{A}^4_{\infty}(f^{3}_{\bm{t},\Pi})] \\ & =\Eb\Big[\sum_{|x|=1}e^{-2V(x)}\Big]\Eb\Big[\underset{|x|=|y|=1}{\sum_{x\not=y}}e^{-V(x)-V(y)}\Big]^2\Eb\Big[\underset{|x|=|y|=1}{\sum_{x\not=y}}e^{-2V(x)-2V(y)}\Big]e^{t_3^*\psi(2)+2t_2^*\psi(2)+\psi(4)}.
\end{align*}
Also introduce the increasing collection of partitions $\Pi'=(\bm{\pi}'_i)_{1\leq i\leq\ell}$ such that $\bm{\pi}'_3=\bm{\pi}_3$, $\bm{\pi}'_2=\bm{\pi}_2$, $\bm{\pi}'_1=\{\{1,3,4\},\{2\}\}$ and $\bm{\pi}'_0=\bm{\pi}_0$. We have $\bm{\beta}^2_1=(1,1)$, $\bm{\beta}^2_2=1$, $\bm{\beta}^2_3=1$; $\bm{\beta}^1_1=(2,1)$, $\bm{\beta}^1_2=1$; $\bm{\beta}^0_1=(3,1)$ and thanks to \eqref{EspPartionLaw}, for any $\bm{t}=(t_1,t_2,t_3)\in\ProdSet{\N}{3}$ such that $t_1<t_2<t_3$ 
\begin{align*}
     &\Eb[\mathcal{A}^4_{\infty}(f^{3}_{\bm{s},\Pi'})] \\ & =\Eb\Big[\underset{|x|=|y|=1}{\sum_{x\not=y}}e^{-V(x)-V(y)}\Big]\Eb\Big[\underset{|x|=|y|=1}{\sum_{x\not=y}}e^{-2V(x)-V(y)}\Big]\Eb\Big[\underset{|x|=|y|=1}{\sum_{x\not=y}}e^{-3V(x)-V(y)}\Big]e^{s_3^*\psi(2)+s_2^*\psi(3)+\psi(4)}.
\end{align*}
The difference between these two examples is that in the second one, we ask $(\VectCoord{\mathcal{X}}{4,n})_{t_2-1}$ (the ancestor of $\VectCoord{\mathcal{X}}{4,n}$ of in generation $t_2-1$) to belong to both genealogical line $\llbracket (\VectCoord{\mathcal{X}}{1,n})_{t_1-1},\VectCoord{\mathcal{X}}{1,n}\rrbracket$ and $\llbracket (\VectCoord{\mathcal{X}}{3,n})_{t_1-1},\VectCoord{\mathcal{X}}{3,n}\rrbracket$. This constraint can be satisfied only if the vertex $(\VectCoord{\mathcal{X}}{4,n})_{t_1-1}$ is often visited by the random walk $\X$, inducing more dependence in the trajectories of $\X$ thus giving the factor $t_2^*\psi(3)$ instead of $2t_2^*\psi(2)=t_2^*\psi(2)+t_2^*\psi(2)$. \\
However, in the case of regular supercritical Galton-Watson trees, the events $\cap_{i=0}^{3}\{\bm{\pi}_i\}$ and $\cap_{i=0}^{3}\{\Tilde{\bm{\pi}}_i\}$ have the same probability under the limit law of the genealogical tree. Indeed, one can notice (see Figure \ref{Genealogical tree 2}) that for all $i\in\{1,2,3\}$ and all $j\in\{1,\ldots,|\bm{\pi}_i|\}$ ($|\bm{\pi}_i|=|\bm{\pi}'_i|$ by definition), $b_{i,\Pi}(\bm{B}_{j})=b_{i,\Pi'}(\bm{B}_{\mathfrak{p}(j)})$ for some permutation $\mathfrak{p}$ on $\llbracket 1,|\bm{\pi}_i|\rrbracket$, but this not the case when replacing $b_{i,\Pi}(\bm{B}_{\cdot})$ by $\bm{\beta}^{i,\Pi}_{\cdot}$ and $b_{i,\Pi'}(\bm{B}_{\cdot})$ by $\bm{\beta}^{i,\Pi'}_{\cdot}$.

\begin{figure}[h]
        \centering
       \begin{tikzpicture}

    %%%%%%%%%%% Arbre 1 %%%%%%%%%%%
    %%%%%%%% Gen 0 %%%%%%%%
        \draw[black,fill] (-1,0) circle [radius=0.02];
        \filldraw (-1,0) node[anchor=west] {$e$};
        %%%%%%%%%%%%%%%%%%%%%%%%
            
        %%%%%%%% Gen Inter 1 %%%%%%%%
        %\draw[dotted] (0,0)--(-0.4,0.5);
        %\draw[black,fill] (-0.4,0.5) circle [radius=0.02];
        
        %\draw[dotted] (0,0)--(0,0.5);
        %\draw[black,fill] (0,0.5) circle [radius=0.02];
        
        %\draw[dotted] (0,0)--(0.4,0.5);
        %\draw[black,fill] (0.4,0.5) circle [radius=0.02];
        %%%%%%%%%%%%%%%%%%%%%%%%

         %%%%%%%% Gen 1 %%%%%%%%
        \draw (-1,0) -- (-1,2);
        \draw[black,fill] (-1,2) circle [radius=0.02];

        \filldraw (-2.5,2.2) node[anchor=west] {$[2,(2,2)]$};

        \filldraw (3.8,2.2) node[anchor=west] {$[2,(3,1)]$};
        
        %%%%%%%%%%%%%%%%%%%%%%%%

        %%%%%%%% Gen 2 %%%%%%%%
        \draw (-1,2) -- (0,3.5);
        %\draw[black,fill] (0,3.5) circle [radius=0.02];

        \filldraw (-2.9,3.7) node[anchor=west] {\textcolor{BrickRed}{$[\textcolor{Peach}{1},2]$}};

        \filldraw (0.05,3.7) node[anchor=west] {\textcolor{BrickRed}{$[\textcolor{Peach}{2},(1,1)]$}};

        \filldraw (2.9,3.7) node[anchor=west] {\textcolor{BrickRed}{$[\textcolor{Peach}{2},(2,1)]$}};

        \filldraw (5.9,3.7) node[anchor=west] {\textcolor{BrickRed}{$[\textcolor{Peach}{1},1]$}};

        %%%%%%%%%%%%%%%%%%%%%%%%

        %%%%%%%% Gen 3 %%%%%%%%
        \draw (-1,2) -- (-2.5,4.5);
        \draw[black,fill] (-2.5,4.5) circle [radius=0.02];

        \filldraw (-4,4.7) node[anchor=west] {$[2,(1,1)]$};

        \filldraw (-1.1,4.7) node[anchor=west] {$[1,1]$};

        \filldraw (0.25,4.7) node[anchor=west] {$[1,1]$};

        \filldraw (2.3,4.7) node[anchor=west] {$[2,(1,1)]$};

        \filldraw (4.9,4.7) node[anchor=west] {$[1,1]$};

        \filldraw (6.25,4.7) node[anchor=west] {$[1,1]$};
        
        %%%%%%%%%%%%%%%%%%%%%%%%

        \draw (-2.5,4.5) -- (-3.7,7);
        \draw[black,fill] (-3.7,7) circle [radius=0.02];

        \draw (-2.5,4.5) -- (-2.2,6);
        \draw[black,fill] (-2.2,6) circle [radius=0.02];

        \draw (0,3.5) -- (-0.7,6.5);
        \draw[black,fill] (-0.7,6.5) circle [radius=0.02];

        \draw (0,3.5) -- (0.7,6.85);
        \draw[black,fill] (0.7,6.85) circle [radius=0.02];

        %%%%%%%% sommets x %%%%%%%%

        \filldraw (-3.7,7.2) node[anchor=west] {$\VectCoord{\mathcal{X}}{3,n}$};

        \filldraw (-2.2,6.2) node[anchor=west] {$\VectCoord{\mathcal{X}}{1,n}$};

        \filldraw (-0.7,6.7) node[anchor=west] {$\VectCoord{\mathcal{X}}{4,n}$};
        
        \filldraw (0.7,7.05) node[anchor=west] {$\VectCoord{\mathcal{X}}{2,n}$}; 
        
        %%%%%%%%%%%%%%%%%%%%%%%%%%%%%%%%%%%%%%%%

        %%%%%%%%%%% Arbre 2 %%%%%%%%%%%

        %%%%%%%% Gen 0 %%%%%%%%
        \draw[black,fill] (5.3,0) circle [radius=0.02];
        \filldraw (5.3,0) node[anchor=west] {$e$};
        %%%%%%%%%%%%%%%%%%%%%%%%
            
        %%%%%%%% Gen Inter 1 %%%%%%%%
        %\draw[dotted] (0,0)--(-0.4,0.5);
        %\draw[black,fill] (-0.4,0.5) circle [radius=0.02];
        
        %\draw[dotted] (0,0)--(0,0.5);
        %\draw[black,fill] (0,0.5) circle [radius=0.02];
        
        %\draw[dotted] (0,0)--(0.4,0.5);
        %\draw[black,fill] (0.4,0.5) circle [radius=0.02];
        %%%%%%%%%%%%%%%%%%%%%%%%

         %%%%%%%% Gen 1 %%%%%%%%
        \draw (5.3,0) -- (5.3,2);
        \draw[black,fill] (5.3,2) circle [radius=0.02];
        
        %%%%%%%%%%%%%%%%%%%%%%%%

        %%%%%%%% Gen 2 %%%%%%%%
        %\draw (-1,2) -- (0,3.5);
        %\draw[black,fill] (0,3.5) circle [radius=0.02];

        %%%%%%%%%%%%%%%%%%%%%%%%

        %%%%%%%% Gen 3 %%%%%%%%
        \draw (5.3,2) -- (3.8,4.5);
        \draw[black,fill] (3.8,4.5) circle [radius=0.02];
        %%%%%%%%%%%%%%%%%%%%%%%%

        \draw (3.8,4.5) -- (2.6,7);
        \draw[black,fill] (2.6,7) circle [radius=0.02];

        \draw (3.8,4.5) -- (4.1,6);
        \draw[black,fill] (4.1,6) circle [radius=0.02];

        %\draw (0,3.5) -- (-0.7,6.5);
        \draw[black,fill] (5.6,6.5) circle [radius=0.02];

        %\draw (0,3.5) -- (0.7,6.85);
        \draw[black,fill] (7,6.85) circle [radius=0.02];

        %\draw[black,fill] (4.4,3.5) circle [radius=0.02];
        \draw (4.4,3.5) -- (5.6,6.5);

        \draw (5.3,2) -- (7,6.85);

        %%%%%%%% Ligne gen + flèche %%%%%%%%

        \filldraw (9.3,0) node[anchor=west] {$0$};

        \draw[latex-] (9.3,8) -- (9.3,0);

        \filldraw (9.3,8) node[anchor=west] {\textrm{ generation }};

        \draw[dotted] (-3.7,2) -- (9.3,2);
        \filldraw (9.3,2) node[anchor=west] {$t_1-1$};
        
        %\draw[dotted] (-3.7,3.5) -- (9.3,3.5);
        %\filldraw (9.3,3.5) node[anchor=west] {$t_2-1$};

        \draw[dotted] (-3.7,4.5) -- (9.3,4.5);
        \filldraw (9.3,4.5) node[anchor=west] {$t_3-1$};

        %%%%%%%% sommets x %%%%%%%%

        \filldraw (2.6,7.2) node[anchor=west] {$\VectCoord{\mathcal{X}}{3,n}$};

        \filldraw (4.1,6.2) node[anchor=west] {$\VectCoord{\mathcal{X}}{1,n}$};

        \filldraw (5.6,6.7) node[anchor=west] {$\VectCoord{\mathcal{X}}{4,n}$};
        
        \filldraw (7,7.05) node[anchor=west] {$\VectCoord{\mathcal{X}}{2,n}$}; 
        
        %%%%%%%%%%%%%%%%%%%%%%%%%%%%%%%%%%%%%%%%

        %%%%%%%% Gen  couleur %%%%%%%%

        \draw[BrickRed,dotted] (-3.7,3.5) -- (9.3,3.5);
        \filldraw (9.3,3.5) node[anchor=west] {\textcolor{BrickRed}{$t_2-1$}};
        \draw[BrickRed,fill] (0,3.5) circle [radius=0.02];
        \draw[BrickRed,fill] (4.4,3.5) circle [radius=0.02];
    
        %%%%%%%%%%%%%%%%%%%%%%%%%%%%%%%%%%%%%%%%
        
\end{tikzpicture}
\caption{An example of a genealogical tree of the four vertices $\VectCoord{\mathcal{X}}{1,n}$, $\VectCoord{\mathcal{X}}{2,n}$, $\VectCoord{\mathcal{X}}{3,n}$, $\VectCoord{\mathcal{X}}{4,n}$ associated to $\Pi$ (left) and associated to $\Pi'$ (right). 
\textcolor{BrickRed}{$[\textcolor{Peach}{1},2]$} means that $b_{1,\Pi}(\{1,3\})=\textcolor{Peach}{1}$ and $\bm{\beta}^{2,\Pi}_1=\textcolor{BrickRed}{2}$, \textcolor{BrickRed}{$[\textcolor{Peach}{2},(1,1)]$} means that $b_{2,\Pi}(\{2,4\})=\textcolor{Peach}{2}$ and $\bm{\beta}^{2,\Pi}_2=\textcolor{BrickRed}{(1,1)}$. In the same way, \textcolor{BrickRed}{$[\textcolor{Peach}{2},(2,1)]$} means that $b_{1,\Pi'}(\{1,3\})=\textcolor{Peach}{2}$ and $\bm{\beta}^{2,\Pi'}_1=\textcolor{BrickRed}{(2,1)}$, \textcolor{BrickRed}{$[\textcolor{Peach}{1},1]$} means that $b_{2,\Pi'}(\{2\})=\textcolor{Peach}{1}$ and $\bm{\beta}^{2,\Pi'}_2=\textcolor{BrickRed}{1}$.}\label{Genealogical tree 2}
\end{figure}
\end{remark}

\noindent Since all the coalescences  of the genealogical lines of $\VectCoord{\mathcal{X}}{1,n},\ldots,\VectCoord{\mathcal{X}}{k,n}$ occur in the remote past with large probability, one could focus on these particular vertices of the tree $\mathcal{R}_{T^{\sqrtBis{n}}}$. To do that, we pick a $k$-tuple $\VectCoord{\mathcal{Y}}{n}=(\VectCoord{\mathcal{Y}}{1,n},\ldots,\VectCoord{\mathcal{Y}}{k,n})$ uniformly in the set $\ProdSet{\mathcal{D}_n}{k}\cap\mathcal{C}^k_{\mathfrak{s}}$ for $\mathfrak{s}\in\N^*$ where we recall that $\mathcal{C}^k_{\mathfrak{s}}=\{\Vect{x}\in\Delta^k;\; \mathcal{S}^k(\Vect{x})\leq\mathfrak{s}\}$, see \eqref{MultiMRCA} for the definition of $ \mathcal{S}^k(\Vect{x})$. In other words, the law of $\VectCoord{\mathcal{Y}}{n}$ is given in \eqref{LawVertices2} and \eqref{LawVertices1} by replacing $\Delta^k(\mathcal{D}_n)$ with $\Delta^k(\mathcal{D}_n)\cap\mathcal{C}^k_{\mathfrak{s}}$. We keep the same notation for $\VectCoord{\mathcal{Y}}{n}$ as for $\VectCoord{\mathcal{X}}{n}$. In particular, $\pi^{k,n}$ denotes here the genealogical tree of the vertices $\VectCoord{\mathcal{Y}}{1,n},\ldots,\VectCoord{\mathcal{Y}}{k,n}$. \\
Let us also introduce the coalescent times (or split times) of the vertices $\VectCoord{\mathcal{Y}}{1,n},\ldots,\VectCoord{\mathcal{Y}}{k,n}$. Define the coalescent times by: $\mathcal{S}^{k,n}_0:=0$ and for all $j\in\N^*$, $k\geq 2$
\begin{align}\label{CoaTimes}
    \mathcal{S}^{k,n}_j:=\min\big\{m\geq\mathcal{S}^{k,n}_{j-1};\; |\pi^{k,n}_{m}|>|\pi^{k,n}_{\mathcal{S}^{k,n}_{j-1}}|\land(k-1)\big\}.
\end{align}
Note that there exists $\mathcal{J}^{k,n}\in\N$ such that for any $j\geq\mathcal{J}^{k,n}$, $\mathcal{S}^{k,n}_{j}=\mathcal{S}^k(\VectCoord{\mathcal{Y}}{n})$ and by definition, $2\leq|\{\mathcal{S}^n_j;\; j\in\N\}|\leq k$. One can notice that seen backwards in time, each random time $\mathcal{S}^{k,n}_j-1$ with $0<j\leq\mathcal{J}^{k,n}$ corresponds to a generation at which two or more vertices among $\VectCoord{\mathcal{Y}}{1,n},\ldots,\VectCoord{\mathcal{Y}}{k,n}$ share a common ancestor for the first time. $\mathcal{S}^{k,n}_j$ is usually referred to as the $j$-th split time while $\mathcal{S}^{k,n}_{\mathcal{J}^{k,n}-j+1}$ is the $j$-th coalescent time.

\begin{figure}[h]
\centering
    \begin{tikzpicture}

        %%%%%%%% Gen 0 %%%%%%%%
        \draw[black,fill] (-1,0) circle [radius=0.02];
        \filldraw (-1,0) node[anchor=west] {$e$};
        %%%%%%%%%%%%%%%%%%%%%%%%
            
        %%%%%%%% Gen Inter 1 %%%%%%%%
        %\draw[dotted] (0,0)--(-0.4,0.5);
        %\draw[black,fill] (-0.4,0.5) circle [radius=0.02];
        
        %\draw[dotted] (0,0)--(0,0.5);
        %\draw[black,fill] (0,0.5) circle [radius=0.02];
        
        %\draw[dotted] (0,0)--(0.4,0.5);
        %\draw[black,fill] (0.4,0.5) circle [radius=0.02];
        %%%%%%%%%%%%%%%%%%%%%%%%

         %%%%%%%% Gen 1 %%%%%%%%
        \draw (-1,0) -- (-1,2);
        \draw[black,fill] (-1,2) circle [radius=0.02];
        
        %%%%%%%%%%%%%%%%%%%%%%%%

        %%%%%%%% Gen 2 %%%%%%%%
        \draw (-1,2) -- (0,3.5);
        \draw[black,fill] (0,3.5) circle [radius=0.02];

        %%%%%%%%%%%%%%%%%%%%%%%%

        %%%%%%%% Gen 3 %%%%%%%%
        \draw (-1,2) -- (-2.5,4.5);
        \draw[black,fill] (-2.5,4.5) circle [radius=0.02];
        %%%%%%%%%%%%%%%%%%%%%%%%

        \draw (-2.5,4.5) -- (-3.7,7);
        \draw[black,fill] (-3.7,7) circle [radius=0.02];

        \draw (-2.5,4.5) -- (-2.2,6);
        \draw[black,fill] (-2.2,6) circle [radius=0.02];

        \draw (0,3.5) -- (-0.5,4.2);
        \draw[black,fill] (-0.5,4.2) circle [radius=0.02];

        \draw (0,3.5) -- (2.1,6.85);
        \draw[black,fill] (2.1,6.85) circle [radius=0.02];

        %%%%%%%% Lignes gen + flèche %%%%%%%%        

        \draw[dotted] (-3.7,2) -- (3.5,2);

        \draw[dotted] (-3.7,3.5) -- (3.5,3.5);

        \draw[dotted] (-3.7,4.5) -- (3.5,4.5);
        
        \draw[dotted] (-3.7,5.5) -- (3.5,5.5);
        
        \draw[latex-] (3.5,8) -- (3.5,0);
        \filldraw (3.5,8) node[anchor=west] {\textrm{ generation }};

        %%%%%%%% Coa times + gen m %%%%%%%% 
        
        \filldraw (3.5,0) node[anchor=west] {$0$};

        \filldraw (3.5,2) node[anchor=west] {$\mathcal{S}^{4,n}_1-1$};
        
        \filldraw (3.5,3.5) node[anchor=west] {$\mathcal{S}^{4,n}_2-1$};

        \filldraw (3.5,4.5) node[anchor=west] {$\mathcal{S}^{4,n}_3-1=\mathcal{S}^{4,n}(\VectCoord{\mathcal{X}}{n})-1$};

        \filldraw (3.5,5.5) node[anchor=west] {$m$};
        
        %%%%%%%%%%%%%%%%%%%%%%%%%%%%%%%%%%%%%%%%
        
        %%%%%%%% sommets x %%%%%%%%

        \filldraw (-3.7,7.2) node[anchor=west] {$\VectCoord{\mathcal{X}}{3,n}$};

        \filldraw (-2.2,6.2) node[anchor=west] {$\VectCoord{\mathcal{X}}{1,n}$};

        \filldraw (-0.5,4.4) node[anchor=west] {$\VectCoord{\mathcal{X}}{4,n}$};
        
        \filldraw (2.1,7.05) node[anchor=west] {$\VectCoord{\mathcal{X}}{2,n}$}; 
        
        %%%%%%%%%%%%%%%%%%%%%%%%%%%%%%%%%%%%%%%%

        \end{tikzpicture}
\caption{An example of four vertices belonging to $\mathcal{C}^4_m$ together with their three coalescent times.}\label{Coalescence times}
\end{figure}

\noindent The last example gives the law of the coalescent times $(\mathcal{S}^{k,n})_{1\leq j\leq\mathcal{J}^{k,n}}$ of $\VectCoord{\mathcal{Y}}{1,n},\ldots,\VectCoord{\mathcal{Y}}{k,n}$:
\begin{coro}[\textbf{Coalescent times}]\label{GenTh4}
Let $k\geq 2$ and assume that $\kappa>2k$. Let $1\leq\ell<k$, $\mathfrak{s}\in\N^*$ be two integers, and $\bm{s}=(s_1,\ldots,s_{\ell})\in\ProdSet{\N}{\ell}$ such that $s_1<\ldots<s_{\ell}\leq\mathfrak{s}$. Assume that for all $\Vect{x}\in\Delta^k$,
$$F^{\ell}_{\bm{s}}(\Vect{x})=\sum_{\Xi\textrm{ increasing }}f^{\ell}_{\bm{s},\Xi}(\Vect{x}), $$ \\
Under the Assumptions \ref{Assumption1}, \ref{AssumptionSmallGenerations}, \ref{Assumption2} and \ref{AssumptionIncrements}, 
\begin{enumerate}
    \item in $\P^*$-probability
        \begin{align}\label{Th4Trace}
        \frac{\mathcal{A}^k(\mathcal{D}_n,F^{\ell}_{\bm{s}})}{\mathcal{A}^k(\mathcal{D}_n,\un_{\mathcal{C}^k_{\mathfrak{s}}})}\un_{\{|\Delta^k(\mathcal{D}_n)|>0\}}\underset{n\to\infty}{\longrightarrow}\frac{\mathcal{A}^k_{\infty}(F^{\ell}_{\bm{s}})}{\mathcal{A}^k_{\infty}(\un_{\mathcal{C}^k_{\mathfrak{s}}})},
    \end{align}
    where $\mathcal{A}^k_{\infty}(\un_{\mathcal{C}^k_{\mathfrak{s}}})$ is defined in Corollary \ref{GenTh2} and $(\mathcal{A}^k_{\infty}(F^{d}_{\bm{t}});\; 1\leq d<k,\; \bm{t}=(t_1,\cdots,t_d))$ is a collection of random variables satisfying
    \begin{align*}
        \sum_{d=1}^{k-1}\;\underset{t_1<\cdots<t_{d}\leq\mathfrak{s}}{\sum_{\bm{t}=(t_1,\ldots,t_{d})}}\mathcal{A}^k_{\infty}(F^{d}_{\bm{t}})= \mathcal{A}^k_{\infty}(\un_{\mathcal{C}^k_{\mathfrak{s}}}).
    \end{align*}
    \item Moreover
    \begin{align}
        \P^*(\mathcal{S}^{k,n}_1=s_1,\ldots,\mathcal{S}^{k,n}_{\ell}=s_{\ell},\mathcal{J}^{k,n}=\ell)\underset{n\to\infty}{\longrightarrow}\Eb^*\Big[\frac{\mathcal{A}^k_{\infty}(F^{\ell}_{\bm{s}})}{\mathcal{A}^k_{\infty}(\un_{\mathcal{C}^k_{\mathfrak{s}}})}\Big].
    \end{align}
\end{enumerate}
\end{coro}

\begin{remark}
    Note, by Theorem \ref{GenTh7}, that all the previous results on $\mathcal{D}_n=\{x\in\mathcal{R}_{T^{\sqrtBis{n}}};\; \ell_n\leq|x|\leq\mathfrak{L}_n\}$ hold for $\Tilde{\mathcal{D}}_n=\{x\in\mathcal{R}_n;\; \ell_n\leq|x|\leq\mathfrak{L}_n\}$ with $\mathfrak{L}_n=o(n^{1/2-\mathfrak{q}({\xi})})$. In particular, we are able to obtain information about the genealogical structure of a $k$-tuple $\VectCoord{\Tilde{\mathcal{X}}}{n}$ picked uniformly in the set $\Delta^k(\Tilde{\mathcal{D}}_n)$, conditionally on the event $|\Delta^k(\Tilde{\mathcal{D}}_n)|>0$. 
\end{remark}

\subsection{The tiny and the critical generations}\label{FurtherDiscusion}

Recall that $\psi(t)=\log\Eb[\sum_{|x|=1}e^{-tV(x)}]$ and introduce $\tilde{\gamma}:=\sup\{a\in\R;\; \inf_{t\geq 0}(\psi(-t)-at)>0\}$. By tiny generations, we mean those of order $\ell_n$ where $\ell_n\to\infty$ when $n\to\infty$ and $\ell_n\leq G\log n$ with $G\in(0,(2\tilde{\gamma})^{-1})$. The fact is that for these generations, the random environment has a uniform impact. Indeed, P. Andreoletti and P. Debs proved in \cite{AndDeb1} that with high probability, $\{x\in\mathcal{R}_n;\; |x|\leq G\log n\}=\{x\in\T;\; |x|\leq G\log n\}$ for all $G\in(0,(2\tilde{\gamma})^{-1})$. Moreover, the value $(2\tilde{\gamma})^{-1}$ is optimal: if $G_n$ denotes the largest generation entirely visited by the random walk $\X$ up to the time $n$, then $\P^*$-almost surely
$$ \frac{G_n}{\log n}\underset{n\to\infty}{\longrightarrow}\frac{1}{2\tilde{\gamma}}. $$
For this case, we are therefore capable of giving a description of the genealogy of $k\geq 2$ vertices uniformly chosen by adapting the results on the genealogical structure of continuous-time Galton-Watson trees of S. Harris, S. Johnston and M. Roberts (see \cite{HarrisJohnstonRoberts1} and \cite{Johnston1}) to discrete supercritical Galton-Watson trees.

\noindent The critical generations, that is to say of order $\sqrtBis{n}$, correspond to the typical generations but also to the largest reached by the diffusive random walk $\X$ up to the time $n$. E. Aïdékon and L. de Raphélis \cite{AidRap} showed that $\sqrtBis{n}$ is also the right normalization for the tree $\mathcal{R}_n$: in law, under $\P^*$
$$ \frac{\sqrtBis{c_0}}{\sqrtBis{n}}\mathcal{R}_n\underset{n\to\infty}{\longrightarrow}\mathcal{T}_{|B|},$$
where for any $\mathfrak{c}>0$, $\mathfrak{c}\mathcal{R}_T$ is tree $\mathcal{R}_T$ with edge lengths equal to $\mathfrak{c}$ and $\mathcal{T}_{|B|}$ is the real tree coded by the standard reflected Brownian motion $|B|=(|B_t|)_{t\in[0,1]}$ on $[0,1]$ (see \cite{JFLeGall_RT}). $\mathcal{T}_{|B|}$ is what we can call a Brownian forest thus suggesting that two vertices $\VectCoord{\mathcal{X}}{1,n}$ and $\VectCoord{\mathcal{X}}{2,n}$ chosen uniformly in the range $\mathcal{R}_n$ at a generation of order $\sqrtBis{n}$ can share a common ancestor in both remote past and recent past. That is actually what is happening when considering two vertices $\VectCoord{\tilde{\mathcal{X}}}{1,n}$ and $\VectCoord{\tilde{\mathcal{X}}}{2,n}$ picked uniformly at generation $\sqrtBis{n}$ in the tree $\mathcal{R}_{T^{\sqrtBis{n}}}$, where we recall that $T^{\sqrtBis{n}}$ is the $\sqrtBis{n}$-th return time of $\X$ to $e^*$ (which is quite similar to $\mathcal{R}_n$): let $\tilde{\mathcal{M}}_n$ be the most recent common ancestor of $\VectCoord{\tilde{\mathcal{X}}}{1,n}$ and $\VectCoord{\tilde{\mathcal{X}}}{2,n}$. First observe that 
\begin{align}\label{CoaCritical}
    \lim_{\varepsilon\to 0}\liminf_{n\to\infty}\P^*(\tilde{\mathcal{M}}_n< 1/\varepsilon)>0\;\;\textnormal{ and }\;\;  \lim_{\varepsilon\to 0}\limsup_{n\to\infty}\P^*(\varepsilon\sqrtBis{n}\leq\tilde{\mathcal{M}}_n<\sqrtBis{n})>0.
\end{align}
Moreover, coalescence can not occur anywhere else:
\begin{align*}
    \lim_{\varepsilon\to 0}\limsup_{n\to\infty}\E^*\Big[\frac{1}{n}\underset{|x|=|y|=\sqrtBis{n}}{\sum_{x\not=y}}\un_{\{x,y\in\mathcal{R}_{T^{\sqrtBis{n}}},\; 1/\varepsilon\leq|x\land y|<\varepsilon\sqrtBis{n}\}}\Big]=0. 
\end{align*}
Although $\T$ is a supercritical Galton-Watson tree, the genealogy of $\mathcal{R}_{T^{\sqrtBis{n}}}$ (or $\mathcal{R}_n$) is a mix of the supercritical case and the critical case for a regular Galton-Watson trees (see subsection \ref{GenVertices}). \\
The fact is using standard techniques for randomly biased random walks and branching random walks, we are able to deal with the quenched mean of $(D_{T^{\sqrtBis{n}}})^{p_1}$ for $p_1\leq\lfloor\kappa\rfloor$ and $(\mathcal{A}^2(\mathcal{D}_{T^{\sqrtBis{n}}},f))^{p_2}$ with $p_2\leq\lfloor\kappa/2\rfloor$ but not with the actual random variables. \\
The computation for any $m>0$ and any $0<a<b<1$ of $\P^*(\tilde{\mathcal{M}}_n<m)$ and $\P^*(a\sqrtBis{n}\leq\tilde{\mathcal{M}}_n<b\sqrtBis{n})$ is part of an ongoing work. \\
The present paper aims in some way to describe the interaction between the vertices of the tree $\mathcal{R}_{T^{\sqrtBis{n}}}$ in the set of generations $\textrm{\guillemotleft squashed\guillemotright}$ when rescaling the tree by $\sqrtBis{n}$.

\begin{remark}\label{RemarqueRange}
    The curiosity here is the fact that critical generations and small generations equally contributed to the range. Indeed, whether $\mathfrak{L}_n$ is negligible with respect to $\sqrtBis{n}$ (with $\mathfrak{L}_n\geq \delta_0^{-1}\log n$) or not, $\sum_{|u|=\mathfrak{L}_n}\un_{\{u\in\mathcal{R}_{T^{\sqrtBis{n}}}\}}$ is of order $\sqrtBis{n}$. This fact makes a deep difference with the slow regime in which only the critical generations (that is typical generations, of order $(\log n)^2$) contribute significantly to the range (see \cite{AndChen}, Theorem 1.2 and Proposition 1.4).
\end{remark}

\vspace{0.3cm}

\begin{remark}[The sub-diffusive and the slow regimes]
    In the sub-diffusive case for the random $\mathbb{X}$, that is when $\kappa\in(1,2]$, there is no reason to believe that the genealogical structure of the range is different from the diffusive case. Indeed, as in the case $\kappa>2$, we have the convergence of the rescaled range, no longer to a Brownian forest but towards a L\'evy forest (see Theorem 1 in \cite{deRaph1}), suggesting that if we sample two vertices uniformly in a critical generation (that is a generation of order $n^{1-1/\kappa}$ for $\kappa\in(1,2)$ and $(n/\log n)^{1/2}$ for $\kappa=2$) in the range up to $n$, the coalescence happens either in the recent past or in the remote past. When we sample two vertices uniformly in a small generation, again, the coalescence should happen close to the root. However, in the slow regime for the random walk $\mathbb{X}$, that is when $\psi(1)=\psi'(1)=0$, it surprisingly turns out that the most recent common ancestor of two vertices sampled uniformly in a generation of order $(\log n)^2$ in the range up to the time $n$ (the critical generation in the slow regime, see Remark \ref{RemarqueRange}) is located close to the root of the Galton-Watson tree $\T$. These results are part of an ongoing work.
\end{remark}

\vspace{0.2cm}

\noindent The rest of the paper is organized as follows: we first prove Theorem \ref{GenTh5} and Theorem \ref{GenTh7}, see subsection \ref{ProofsTH}. Proofs are mostly based on two important propositions: Proposition \ref{GENPROPCONV2} and Proposition \ref{GENPROPCONV1}. Roughly speaking, we claim in these propositions that only $k$-tuples of vertices visited during $k$ distinct excursions above the root of $\T$ up to time $T^{\sqrtBis{n}}$ give a significant contribution to $\mathcal{A}^k(\mathcal{D}_n,f)$ and this will be a key ingredient to show that the range $\mathcal{A}^k(\mathcal{D}_n,f)$ concentrates around the quenched mean of this latter range restricted to $k$-tuples of vertices visited during $k$ distinct excursions. We then prove our five examples, that is Corollaries \ref{GenTh1} -\ref{GenTh4}, see subsection \ref{ProofsExamples}. Finally, section \ref{ProofsProp} is devoted to the proof of Propositions \ref{GENPROPCONV2} and \ref{GENPROPCONV1} .

\section{Proofs of our theorems and corollaries}

In this section, we prove Theorem \ref{GenTh5} and Theorem \ref{GenTh7} and end it with the proofs of our corollaries presented as examples.

\subsection{Proofs of Theorem \ref{GenTh5} and Theorem \ref{GenTh7}}\label{ProofsTH}

Before proving our theorems, let us state two very important propositions. Recall that $T^j$ is the $j$-th return time  to $e^*$: $T^0=0$ and for any $j\geq 1$, $T^j=\inf\{i>T^{j-1};\; X_i=e^*\}$. \\
Let $s\in\N^*$ and introduce $\mathcal{D}_{n,T^s}:=\{x\in\mathcal{R}_{T^s};\;\ell_n\leq|x|\leq\mathfrak{L}_n\}$. We denote by $\mathfrak{E}^{k,s}$ the set defined by: for a given $\Vect{x}=(\VectCoord{x}{1},\ldots,\VectCoord{x}{k})\in\Delta^k$, $\Vect{x}\in\mathfrak{E}^{k,s}$ if and only if the vertices of $\VectCoord{x}{1},\ldots,\VectCoord{x}{k}$ are visited during $k$ distinct excursions before the instant $T^s$:
\begin{align}\label{DistinctExcur}
    \mathfrak{E}^{k,s}:=\bigcup_{\Vect{j}\in\llbracket 1,s\rrbracket_k}\bigcap_{i=1}^k\{\Vect{x}=(\VectCoord{x}{1},\ldots,\VectCoord{x}{k})\in\Delta^k;\;  \mathcal{L}^{T^{j_i}}_{\VectCoord{x}{i}}-\mathcal{L}^{T^{j_i-1}}_{\VectCoord{x}{i}}\geq 1\},
\end{align}
where we denote by $\llbracket 1,s\rrbracket_k$ the set of $k$-tuples $\bm{j}$ of $\{1,\ldots,s\}$ such that for all $i_1\not=i_2\in\{1,\ldots,s\}$, $j_{i_1}\not=j_{i_2}$. Our first proposition is a convergence of the range $\mathcal{A}^k(\mathcal{D}_{n,T^s},f\un_{\mathfrak{E}^{k,s}})$ for any $\varepsilon_1\sqrtBis{n}\leq s\leq\sqrtBis{n}/\varepsilon_1$, $\varepsilon_1\in(0,1)$.
\begin{proposition}\label{GENPROPCONV2}
Let $k\geq 2$ and assume $\kappa>2k$. Under the Assumptions \ref{Assumption1}, \ref{AssumptionSmallGenerations}, \ref{Assumption2} and \ref{AssumptionIncrements}, if $f$ satisfies the hereditary Assumption \ref{Assumption3} then for all $\varepsilon,\varepsilon_1\in(0,1)$, $\varepsilon_1\sqrtBis{n}\leq s\leq\sqrtBis{n}/\varepsilon_1$
\begin{align*}
    \P^*\Big(\Big|\frac{1}{(s\bm{L}_n)^k}\mathcal{A}^k(\mathcal{D}_{n,T^s},f\un_{\mathfrak{E}^{k,s}})-(c_{\infty})^k\mathcal{A}^k_{\infty}(f)\Big|>\varepsilon\Big)\underset{n\to\infty}{\longrightarrow} 0,
\end{align*}
where we recall that in $L^2(\Pb^*)$, $\mathcal{A}^k_{\infty}(f)=\lim_{l\to\infty}\mathcal{A}^k_l(f)$ with $\mathcal{A}^k_l(f,\Vect{\beta})=\sum_{\Vect{x}\in\Delta^k_l}f(\Vect{x})e^{-\langle\Vect{\beta},V(\Vect{x})\rangle_{k}}$, $\mathcal{A}_l(f)=\mathcal{A}_l(f,\Vect{1})$ and $\Vect{1}=(1,\ldots,1)\in\N^{\times k}$.
\end{proposition}

\vspace{0.2cm}

\noindent In the next proposition, we claim $k$-tuples in $\Delta^k\setminus\mathfrak{E}^{k,s}$ with $s\leq\sqrtBis{n}/\varepsilon_1$ and $\varepsilon_1\in(0,1)$, that is $k$-tuples of vertices such that at least two among them are visited during the same excursion above $e^*$ and before $T^s$, have a minor contribution to the range $\mathcal{A}^k(\mathcal{D}_n,1)$.

\begin{proposition}\label{GENPROPCONV1}
Let $\varepsilon\in(0,1)$, $k\geq 2$ and assume $\kappa>2k$. Under the Assumptions \ref{Assumption1}, \ref{AssumptionSmallGenerations}, \ref{Assumption2} and \ref{AssumptionIncrements}
\begin{align}
    \P^*\big(\sup_{s\leq\sqrtBis{n}/\varepsilon_1}\mathcal{A}^k(\mathcal{D}_{n,T^{s}},\un_{\Delta^k\setminus\mathfrak{E}^{k,s}})>\varepsilon(\sqrtBis{n}\bm{L}_n)^k\big)\underset{n\to\infty}{\longrightarrow} 0.
\end{align}
\end{proposition}

\vspace{0.2cm}

\noindent We are now ready to prove Theorem \ref{GenTh5}.

\begin{proof}[Proof of Theorem \ref{GenTh5}]
    First, $\mathcal{A}^k(\mathcal{D}_{n,T^{\sqrtBis{n}}},f)=\mathcal{A}^k(\mathcal{D}_{n},f\un_{\mathfrak{E}^{k,\sqrtBis{n}}})+\mathcal{A}^k(\mathcal{D}_{n},f\un_{\Delta^k\setminus\mathfrak{E}^{k,\sqrtBis{n}}})$ and then for any $\varepsilon\in(0,1)$
    \begin{align*}
    &\P^*\Big(\Big|\frac{1}{(\sqrtBis{n}\bm{L}_n)^k}\mathcal{A}^k(\mathcal{D}_{n},f)-(c_{\infty})^k\mathcal{A}^k_{\infty}(f)\Big|>\varepsilon\Big) \\ & \leq\P^*\Big(\Big|\frac{1}{(\sqrtBis{n}\bm{L}_n)^k}\mathcal{A}^k(\mathcal{D}_{n},f\un_{\mathfrak{E}^{k,\sqrtBis{n}}})-(c_{\infty})^k\mathcal{A}^k_{\infty}(f)\Big|>\frac{\varepsilon}{2}\Big) \\ & +\P^*\big(\mathcal{A}^k(\mathcal{D}_{n},\un_{\Delta^k\setminus\mathfrak{E}^{k,\sqrtBis{n}}})>\frac{\varepsilon}{2}(\sqrtBis{n}\bm{L}_n)^k\big).
\end{align*}
Noticing that $\mathcal{D}_n=\mathcal{D}_{n,T^{\sqrtBis{n}}}$, the first probability in this sum goes to $0$ when $n\to\infty$ thanks to Proposition \ref{GENPROPCONV2} with $s=\sqrtBis{n}$ and the second one also goes to $0$ thanks to Proposition \ref{GENPROPCONV1} thus giving \eqref{ConvA}. For the convergence in $\P^*$-probability \eqref{ConvAQuotient}, note that 
\begin{align*}
    &\P^*\Big(\Big|\frac{\mathcal{A}^k(\mathcal{D}_{n},f)}{\mathcal{A}^k(\mathcal{D}_{n},g)}\un_{\{|\Delta^k(\mathcal{D}_n)|>0\}}-\frac{\mathcal{A}^k_{\infty}(f)}{\mathcal{A}^k_{\infty}(g)}\Big|>\varepsilon\Big) \\ & \leq\P^*\Big(\Big|\frac{\mathcal{A}^k(\mathcal{D}_{n},f)}{\mathcal{A}^k(\mathcal{D}_{n},g)}-\frac{\mathcal{A}^k_{\infty}(f)}{\mathcal{A}^k_{\infty}(g)}\Big|>\varepsilon,\;|\Delta^k(\mathcal{D}_n)|>0\Big)+\P^*(|\Delta^k(\mathcal{D}_n)|=0),
\end{align*}
these two probabilities go to $0$ when $n\to\infty$ and the proof is completed.
\end{proof}

\noindent We now prove Theorem \ref{GenTh7}. Recall that $\Tilde{\mathcal{D}}_n=\{x\in\mathcal{R}_n;\; \ell_n\leq|x|\leq\mathfrak{L}_n\}$. The main idea of the proof is to show that, when $\kappa>2\xi k$, $\xi\geq 2$, and $\mathfrak{L}_n=o(n^{1/2-\mathfrak{q}({\xi})})$ for some non-increasing sequence $\mathfrak{q}$ such that $\mathfrak{q}(j)\to 0$ when $j\to\infty$, the volume $\Tilde{D}_{n}$ of the range $\Tilde{\mathcal{D}}_n$ behaves like the volume of the range up to the last complete excursion of $(X)_{j\leq n}$ above the parent $e^*$ of the root $e$. \\
For that, one can notice that for this choice of $\kappa$, Proposition \ref{GENPROPCONV2} holds uniformly in $s$ (in the sense of \eqref{ConvRangeUnif}): there exists a non-increasing sequence of positive integers $(\mathfrak{q}(j))_j$, satisfying $\mathfrak{q}(j)\in(0,1/2)$ and $\mathfrak{q}(j)\to 0$ when $j\to\infty$ such that if $\kappa>2\xi k$ for some integer $\xi\geq 2$ and $\mathfrak{L}_n=o(n^{1/2-\mathfrak{q}({\xi})})$ then, for any $\varepsilon_1\in(0,1)$
\begin{align}\label{ConvRangeUnif}
    \P^*\Big(\bigcup_{s=\varepsilon_1\sqrtBis{n}}^{\sqrtBis{n}/\varepsilon_1}\Big|\frac{1}{(s\bm{L}_n)^k}\mathcal{A}^k(\mathcal{D}_{n,T^s},f\un_{\mathfrak{E}^{k,s}})-(c_{\infty})^k\mathcal{A}^k_{\infty}(f)\Big|>\varepsilon\Big)\underset{n\to\infty}{\longrightarrow} 0.
\end{align}
The proof of \eqref{ConvRangeUnif} is the same as the proof of Proposition \ref{GENPROPCONV2} but for any $\varepsilon,\varepsilon_1\in(0,1)$, by Markov inequality
\begin{align*}
    &\P\Big(\bigcup_{s=\varepsilon_1\sqrtBis{n}}^{\sqrtBis{n}/\varepsilon_1}\Big\{\Big|\sum_{\bm{j}\in\llbracket 1,s\rrbracket_k}\mathcal{A}^{k,n}(\Vect{j},f\un_{\mathcal{C}_{a_n}^k})-\E^{\mathcal{E}}\big[\sum_{\bm{j}\in\llbracket 1,s\rrbracket_k}\mathcal{A}^{k,n}(\Vect{j},f\un_{\mathcal{C}_{a_n}^k})\big]\Big|>\varepsilon(s\bm{L}_n)^k/16\Big\}\Big) \\ & \leq\sum_{s=\varepsilon_1\sqrtBis{n}}^{\sqrtBis{n}/\varepsilon_1}\frac{16^{2\xi k}}{\varepsilon^{2\xi k}(s\bm{L}_n)^{2\xi k}}\E\Big[\Big(\sum_{\bm{j}\in\llbracket 1,s\rrbracket_k}\mathcal{A}^{k,n}(\Vect{j},f\un_{\mathcal{C}_{a_n}^k})-\E^{\mathcal{E}}\big[\sum_{\bm{j}\in\llbracket 1,s\rrbracket_k}\mathcal{A}^{k,n}(\Vect{j},f\un_{\mathcal{C}_{a_n}^k})\big]\Big)^{2\xi k}\Big] \\ & \leq 16^{2\xi k}\mathfrak{C}_{\ref{GenVariance}}\sum_{s=\varepsilon_1\sqrtBis{n}}^{\sqrtBis{n}/\varepsilon_1}\Big(\frac{\mathfrak{L}_n}{s}\Big)^{\Tilde{\mathfrak{q}}({\xi})}\leq\mathfrak{C}_{\ref{GENPROPCONV2}}\frac{(\mathfrak{L}_n)^{\Tilde{\mathfrak{q}}({\xi})}}{n^{(\Tilde{\mathfrak{q}}({\xi})-1)/2}},
\end{align*}
where we have used Lemma \ref{GenVariance} with $\mathfrak{a}=\xi$ for second inequality. Note that $\Tilde{\mathfrak{q}}({\xi})\geq 2$ since $\xi\geq 2$ so, as in the proof of Proposition \ref{GENPROPCONV2}, we obtain \eqref{ConvRangeUnif} by taking $\mathfrak{q}(j):=(2\Tilde{\mathfrak{q}}(j))^{-1}$.

\begin{proof}[Proof of Theorem \ref{GenTh7}]
\noindent First, let us state the following fact, proved by Y. Hu (\cite{Hu2017}, Corollary 1.2): in law, under $\P^*$
\begin{align*}
    \frac{1}{\sqrtBis{n}}\sum_{j=1}^n\un_{\{X_k=e\}}\underset{n\to\infty}{\longrightarrow}\frac{1}{p^{\mathcal{E}}(e,e^*)}\frac{c_0^{1/2}}{W_{\infty}}|\mathcal{N}|.
\end{align*}
We can actually adapt this result to the local time $\mathcal{L}^n:=\mathcal{L}^n_{e^*}=\sum_{j=1}^n\un_{\{X_j=e^*\}}$ of the parent $e^*$ of the root $e$: in law, under $\P^*$
\begin{align}\label{LocalTime}
    \frac{1}{\sqrtBis{n}}\mathcal{L}^n\underset{n\to\infty}{\longrightarrow}\frac{c_0^{1/2}}{W_{\infty}}|\mathcal{N}|,
\end{align}
where $c_{0}$ is defined in \eqref{ConvAbis}. Moreover, recall that $\mathcal{N}$ denotes a standard Gaussian variable. Then, we show that $\mathcal{A}^k(\mathcal{D}_{n,T^{\mathcal{L}^n}},f)$ and $\mathcal{A}^k(\Tilde{\mathcal{D}}_n,f)$ are close in the following:
\begin{align}\label{DiffRange}
    \P^*\Big(\frac{1}{(\mathcal{L}^n\bm{L}_n)^k}\big|\mathcal{A}^k(\mathcal{D}_{n,T^{\mathcal{L}^n}},f)-\mathcal{A}^k(\Tilde{\mathcal{D}}_n,f)\big|>\varepsilon\Big)\underset{n\to\infty}{\longrightarrow} 0.
\end{align}
For that, introduce $T_z:=\inf\{i\geq 1\; X_i=z\}$, the hitting time of the vertex $z\in\T$ and for any $\Vect{x}=(\VectCoord{x}{1},\ldots,\VectCoord{x}{k})\in\Delta^k$, $T_{\Vect{x}}:=\max_{1\leq i\leq k}T_{\VectCoord{x}{i}}$. Assume $\mathcal{L}^n=s\in\{\varepsilon_1\sqrtBis{n},\ldots,\sqrtBis{n}/\varepsilon_1\}$.  By definition, $\mathcal{L}^n=\sup\{j\geq 1;\; T^j\leq n\}$ so on the set $\{D_{n,T^s}\geq k\}$, where $D_{n,T^s}$ is the cardinal of $\mathcal{D}_{n,T^s}$, both $\mathcal{D}_{n,T^s}$ and $\Tilde{\mathcal{D}}_n$ are nonempty and note that
\begin{align*}
    \mathcal{A}^k(\Tilde{\mathcal{D}}_n,f)-\mathcal{A}^k(\mathcal{D}_{n,T^{s}},f)&=\sum_{\Vect{p}\in\ProdSet{\{\ell_n,\ldots,\mathfrak{L}_n\}}{k}}\sum_{\Vect{x}\in\Delta^k}f(\Vect{x})\un_{\{|\Vect{x}|=\bm{p},\;T^s<T_{\Vect{x}}\leq n\}} \\ & \leq\|f\|_{\infty}\sum_{\Vect{p}\in\ProdSet{\{\ell_n,\ldots,\mathfrak{L}_n\}}{k}}\sum_{\Vect{x}\in\Delta^k}\un_{\{|\Vect{x}|=\bm{p},\; T^s<T_{\Vect{x}}<T^{s+1}\}},
\end{align*}
where $\|f\|_{\infty}:=\sup_{\Vect{z}\in\Delta^k}f(\Vect{z})$ and $|\Vect{x}|=|\bm{p}|$ means that for all $1\leq i\leq k$, $\VectCoord{x}{i}=p_i$. We now aim to provide an upper bound to $\E[(\sum_{\Vect{p}\in\ProdSet{\{\ell_n,\ldots,\mathfrak{L}_n\}}{k}}\sum_{\Vect{x}\in\Delta^k}\un_{\{|\Vect{x}|=\bm{p},\; T^s<T_{\Vect{x}}<T^{s+1}\}})^2]$. We have
\begin{align*}
    &\E^{\mathcal{E}}\Big[\Big(\sum_{\Vect{p}\in\ProdSet{\{\ell_n,\ldots,\mathfrak{L}_n\}}{k}}\sum_{\Vect{x}\in\Delta^k}\un_{\{|\Vect{x}|=\bm{p},\; T^s<T_{\Vect{x}}<T^{s+1}\}}\Big)^2\Big] \\ & =\underset{\Vect{p}'\in\ProdSet{\{\ell_n,\ldots,\mathfrak{L}_n\}}{k}}{\sum_{\Vect{p}\in\ProdSet{\{\ell_n,\ldots,\mathfrak{L}_n\}}{k}}}\;\underset{\Vect{y}\in\Delta^k}{\sum_{\Vect{x}\in\Delta^k}}\un_{\{|\Vect{x}|=\Vect{p},\; |\Vect{y}|=\Vect{p}'\}}\P^{\mathcal{E}}\big(T^s<T_{\Vect{x}}<T^{s+1},\; T^s<T_{\Vect{y}}<T^{s+1}\big).
\end{align*}
Without loss of generality, we only deal with the case $\VectCoord{x}{i}\not=\VectCoord{y}{i}$ for all $i\in\{1,\ldots,k\}$, that is the case such that the concatenation $\Vect{x}\Vect{y}$ of $\Vect{x}$ and $\Vect{y}$ belongs to $\Delta^{2k}$. One can see that for any $k$-tuple $\Vect{u}=(\VectCoord{u}{1},\ldots,\VectCoord{u}{k})\in\Delta^k$ such that $T^s<T_{\Vect{u}}<T^{s+1}$, we have, for any $i\in\{1,\ldots,k\}$, that either $\VectCoord{u}{i}$ is visited during that $s$-th excursion or $T_{\VectCoord{u}{i}}<T^s$ (at least one vertex among $(\VectCoord{u}{1},\ldots,\VectCoord{u}{k})$ must be visited during the $s$-th excursion). Hence ${\sum_{\Vect{x},\Vect{y}\in\Delta^k,\Vect{x}\Vect{y}\in\Delta^{2k},|\Vect{x}|=\Vect{p},|\Vect{y}|=\Vect{p}'}}\P^{\mathcal{E}}\big(T^s<T_{\Vect{x}}<T^{s+1},\; T^s<T_{\Vect{y}}<T^{s+1}\big)$ is equal to 
\begin{align*}
    &\underset{|I|\leq k-1}{\sum_{I\subset\{1,\ldots,k\}}}\;\;\underset{|J|\leq k-1}{\sum_{J\subset\{1,\ldots,k\}}}\;\;\underset{|\Vect{x}|=\Vect{p},|\Vect{y}|=\Vect{p}'}{\sum_{\Vect{x},\Vect{y}\in\Delta^k,\Vect{x}\Vect{y}\in\Delta^{2k}}} \\ & \times\P^{\mathcal{E}}\Big(\max_{i\in I,j\in J}T_{\VectCoord{x}{i}}\lor T_{\VectCoord{y}{j}}<T^s,\;T^s<T_{\VectCoord{x}{i'}}<T^{s+1},T^s<T_{\VectCoord{y}{j'}}<T^{s+1}\;\forall\;i'\not\in I,\forall\;j'\not\in J\Big)
\end{align*}
where $t\lor s=\max(t,s)$ and $i'\not\in I$ (resp. $j'\not\in J$) means $i'\in\{1,\ldots,k\}\setminus I$ (resp. $j'\in\{1,\ldots,k\}\setminus J$), with $I$ and $J$ possibly empty. Thanks to the strong Markov property at time $T^s$, the latter probability is smaller than 
\begin{align*}
    \P^{\mathcal{E}}\Big(\max_{i\in I,j\in J}T_{\VectCoord{x}{i}}\lor T_{\VectCoord{y}{j}}<T^s\Big)\prod_{i'\not\in I}\P^{\mathcal{E}}\big(T_{\VectCoord{x}{i'}}<T^1\big)\prod_{j'\not\in J}\P^{\mathcal{E}}\big(T_{\VectCoord{y}{j'}}<T^1\big).
\end{align*}
By Lemma \ref{LemmeExc1}, we can restrict to vertices visited during a single excursion before $T^s$. Moreover, for any $i\in I$ and $j\in J$, $\VectCoord{x}{i}$ and $\VectCoord{y}{j}$ are possibly visited during the same excursion. Hence
\begin{align*}
    \P^{\mathcal{E}}\Big(\max_{i\in I,j\in J}T_{\VectCoord{x}{i}}\lor T_{\VectCoord{y}{j}}<T^s,\; &\Vect{x},\Vect{y}\in\mathfrak{S}^{k,s}\Big)\leq\sum_{e_1,\ldots,e_{|I|}=1}^s\;\sum_{e'_1,\ldots,e'_{|J|}=1}^s \\ & \times\P^{\mathcal{E}}\Big(T^{e_i-1}<T_{\VectCoord{x}{i}}<T^{e_i},\;T^{e'_j-1}<T_{\VectCoord{y}{j}}<T^{e'_j}\;\forall\; i\in I, j\in J\Big).
\end{align*}
where $|I|$ (resp. $|J|$) denotes the cardinal of $I$ (resp. $J$), we use the convention $\sum_{\varnothing}=0$ and see \eqref{SingleExcur} for the definition of $\mathfrak{S}^{k,s}$. Note that, if two distinct vertices $u$ and $v$ among $((\VectCoord{x}{i})_{i\in I},(\VectCoord{y}{j})_{j\in J})$ are visited during the same excursion, then we can assume that $|u\land v|<a_n$ (see for example the proof of Lemma \ref{CauchyCarre}). Hence, thanks to Lemma \ref{EspProbaTAPartition} and the fact that $\mathfrak{L}_n\leq s$ for $n$ large enough, we have that
\begin{align*}
    \Eb\Big[&\underset{|\Vect{x}|=\Vect{p},|\Vect{y}|=\Vect{p}'}{\sum_{\Vect{x},\Vect{y}\in\Delta^k,\Vect{x}\Vect{y}\in\Delta^{2k}}} \\ & \times\P^{\mathcal{E}}\Big(\max_{i\in I,j\in J}T_{\VectCoord{x}{i}}\lor T_{\VectCoord{y}{j}}<T^s,\;T^s<T_{\VectCoord{x}{i'}}<T^{s+1},T^s<T_{\VectCoord{y}{j'}}<T^{s+1}\;\forall\;i'\not\in I,\forall\;j'\not\in J\Big)\Big]
\end{align*}
is smaller than $\mathfrak{C}_{\ref{GenTh7}}(a_n)^{2k}s^{|I|+|J|}$ for some constant $\mathfrak{C}_{\ref{GenTh7}}$ only depending on $k$. Since $|I|$ and $|J|$ are smaller than $k-1$, we finally obtain for $n$ large enough and any $s\in\{\varepsilon_1\sqrtBis{n},\ldots,\sqrtBis{n}/\varepsilon_1\}$
\begin{align*}
    \E^{\mathcal{E}}\Big[\Big(\sum_{\Vect{p}\in\ProdSet{\{\ell_n,\ldots,\mathfrak{L}_n\}}{k}}\sum_{\Vect{x}\in\Delta^k}\un_{\{|\Vect{x}|=\bm{p},\; T^s<T_{\Vect{x}}<T^{s+1}\}}\Big)^2\Big]\leq\mathfrak{C}_{\ref{GenTh7},1}(\bm{L}_na_n)^{2k}s^{2k-2},
\end{align*}
for some constant $\mathfrak{C}_{\ref{GenTh7},1}>0$, thus giving, thanks to Markov inequality 
\begin{align*}
    \P^*\Big(\frac{1}{(s\bm{L}_n)^k}\big|\mathcal{A}^k(\Tilde{\mathcal{D}}_n,f)-\mathcal{A}^k(\mathcal{D}_{n,T^{s}},f)\big|>\varepsilon,D_{n,T^s}\geq k,\mathcal{L}^n=s\Big)\leq\|f\|_{\infty}^2\mathfrak{C}_{\ref{GenTh7},1}(a_n)^{2k}s^{-2}\varepsilon^{-2}.
\end{align*}
Hence, for all $\varepsilon_1\in(0,1)$ and $n$ large enough, $\P^*(\{\frac{1}{(\mathcal{L}^n\bm{L}_n)^k}|\mathcal{A}^k(\Tilde{\mathcal{D}}_n,f)-\mathcal{A}^k(\mathcal{D}_{n,T^{\mathcal{L}^n}},f)\big|>\varepsilon\})$ is smaller than
\begin{align*}
     &\P^*\big(D_{n,T^{\mathcal{L}^n}}<k\big)+\P^*\big(\mathcal{L}^n<\varepsilon_1\sqrtBis{n}\big)+\P^*\big(\mathcal{L}^n>\sqrtBis{n}/\varepsilon_1\big) \\ & + \sum_{s=\varepsilon_1\sqrtBis{n}}^{\sqrtBis{n}/\varepsilon_1}\P^*\Big(\frac{1}{(s\bm{L}_n)^k}\big|\mathcal{A}^k(\Tilde{\mathcal{D}}_n,f)-\mathcal{A}^k(\mathcal{D}_{n,T^{s}},f)\big|>\varepsilon,D_{n,T^s}\geq k,\mathcal{L}^n=s\Big) \\ & \leq \P^*\big(D_{n,T^{\mathcal{L}^n}}<k\big)+\P^*\big(\mathcal{L}^n<\varepsilon_1\sqrtBis{n}\big)+\P^*\big(\mathcal{L}^n>\sqrtBis{n}/\varepsilon_1\big)+\|f\|_{\infty}^2\mathfrak{C}_{\ref{GenTh7},1}\sum_{s=\varepsilon_1\sqrtBis{n}}^{\sqrtBis{n}/\varepsilon_1}\frac{(a_n)^{2k}}{s^2\varepsilon^2}.
\end{align*}
$\|f\|_{\infty}^2\mathfrak{C}_{\ref{GenTh7},1}\sum_{s=\varepsilon_1\sqrtBis{n}}^{\sqrtBis{n}/\varepsilon_1}\frac{(a_n)^{2k}}{s^2\varepsilon^2}$ is smaller than $\mathfrak{C}_{\ref{GenTh7},2}(a_n)^{2k}/\sqrtBis{n}$ for some constant $\mathfrak{C}_{\ref{GenTh7},2}>0$. Note that $\lim_{n\to\infty}\P^*(D_{n,T^{\mathcal{L}^n}}<k)=0$ and using \eqref{LocalTime} with the definition of $a_n=(2\delta_0)^{-1}\log n$, we have $\lim_{\varepsilon_1\to 0}\limsup_{n\to\infty}((a_n)^{2k}/\sqrtBis{n}+\P^*(\mathcal{L}^n<\varepsilon_1\sqrtBis{n})+\P^*(\mathcal{L}^n>\sqrtBis{n}/\varepsilon_1))=0$, which yields \eqref{DiffRange}. \\
Now, since $\mathcal{A}^k(\mathcal{D}_{n,T^{\mathcal{L}^n}},f\un_{\Delta^k\setminus\mathfrak{E}^{k,\mathcal{L}^n}})/(\mathcal{L}^n\bm{L}_n)^k\to 0$ when $n\to\infty$, in $\P^*$-probability and thanks to \eqref{DiffRange}, we can focus our attention on $\mathcal{A}^k(\mathcal{D}_{n,T^{\mathcal{L}^n}},f\un_{\mathfrak{E}^{k,\mathcal{L}^n}})/(\mathcal{L}^n\bm{L}_n)^k$. \\
Note the $\mathcal{A}^k(\mathcal{D}_{n,T^{\mathcal{L}^n}},f\un_{\mathfrak{E}^{k,\mathcal{L}^n}})$ concentrates around $(c_{\infty})^k\mathcal{A}^k_{\infty}(f)$. Indeed, for any $\varepsilon,\varepsilon_1\in(0,1)$
\begin{align*}
    &\P^*\Big(\Big|\frac{1}{(\mathcal{L}^n\bm{L}_n)^k}\mathcal{A}^k(\mathcal{D}_{n,T^{\mathcal{L}^n}},f\un_{\mathfrak{E}^{k,\mathcal{L}^n}})-(c_{\infty})^k\mathcal{A}^k_{\infty}(f)\Big|>\varepsilon\Big) \\ & \leq\P^*\Big(\bigcup_{s=\varepsilon_1\sqrtBis{n}}^{\sqrtBis{n}/\varepsilon_1}\Big\{\Big|\frac{1}{(s\bm{L}_n)^k}\mathcal{A}^k(\mathcal{D}_{n,T^s},f\un_{\mathfrak{E}^{k,s}})-(c_{\infty})^k\mathcal{A}^k_{\infty}(f)\Big|>\varepsilon\Big\}\Big)+\P^*(\mathcal{L}^n<\varepsilon_1\sqrtBis{n}) \\ & +\P^*(\mathcal{L}^n>\sqrtBis{n}/\varepsilon_1).
\end{align*}
Thanks to equation \eqref{ConvRangeUnif}, the first probability above goes to $0$ when $n$ goes to $\infty$ and by \eqref{LocalTime}, $\lim_{\varepsilon_1\to 0}\lim_{n\to\infty}(\P^*(\mathcal{L}^n<\varepsilon_1\sqrtBis{n})+\P^*(\mathcal{L}^n>\sqrtBis{n}/\varepsilon_1))=0$ thus giving
\begin{align}\label{RangeLocalTime}
    \lim_{n\to\infty}\P^*\Big(\Big|\frac{1}{(\mathcal{L}^n\bm{L}_n)^k}\mathcal{A}^k(\mathcal{D}_{n,T^{\mathcal{L}^n}},f\un_{\mathfrak{E}^{k,\mathcal{L}^n}})-(c_{\infty})^k\mathcal{A}^k_{\infty}(f)\Big|>\varepsilon\Big)=0.
\end{align}
We obtain from \eqref{RangeLocalTime}, together with \eqref{DiffRange} that 
\begin{align*}
    \lim_{n\to\infty}\P^*\Big(\Big|\frac{1}{(\mathcal{L}^n\bm{L}_n)^k}\mathcal{A}^k(\Tilde{\mathcal{D}}_{n},f)-(c_{\infty})^k\mathcal{A}^k_{\infty}(f)\Big|>\varepsilon\Big)=0,
\end{align*}
which gives \eqref{ConvAbis} by using \eqref{LocalTime}. For the convergence in $\P^*$-probability \eqref{ConvAQuotientbis}, note that 
\begin{align*}
    &\P^*\Big(\Big|\frac{\mathcal{A}^k(\Tilde{\mathcal{D}}_{n},f)}{\mathcal{A}^k(\Tilde{\mathcal{D}}_{n},g)}\un_{\{|\Delta^k(\mathcal{D}_n)|>0\}}-\frac{\mathcal{A}^k_{\infty}(f)}{\mathcal{A}^k_{\infty}(g)}\Big|>\varepsilon\Big) \\ & \leq\P^*\Big(\Big|\frac{\mathcal{A}^k(\Tilde{\mathcal{D}}_{n},f)}{\mathcal{A}^k(\Tilde{\mathcal{D}}_{n},g)}-\frac{\mathcal{A}^k_{\infty}(f)}{\mathcal{A}^k_{\infty}(g)}\Big|>\varepsilon,\; |\Delta^k(\mathcal{D}_n)|>0\Big)+\P^*(|\Delta^k(\mathcal{D}_n)|=0),
\end{align*}
which goes to $0$ when $n$ goes to $\infty$ and the proof is completed.
\end{proof}

\subsection{Proofs of Corollaries \ref{GenTh1} to \ref{GenTh4}}\label{ProofsExamples}

In this subsection, we give a proof of each example stated above except for the Corollary \ref{GenTh1} which is the simplest application of Theorem \ref{GenTh5}, taking $f=1$. For each example, the procedure is as follows: we first prove the function $f$ we consider satisfies the hereditary Assumption \ref{Assumption3} and we then give useful precisions on $\mathcal{A}_{\infty}^k(f)$ for the description of the genealogy of the vertices $\VectCoord{\mathcal{X}}{1,n},\ldots,\VectCoord{\mathcal{X}}{k,n}$.

\vspace{0.5cm} 

\begin{proof}[Proof of Corollary \ref{GenTh1Bis}]
Recall that for $\bm{\lambda}=(\lambda_2,\ldots,\lambda_k)\in\ProdSet{(\N^*)}{(k-1)}$ and $\Vect{x}=(\VectCoord{x}{1},\ldots,\VectCoord{x}{k})\in\Delta^k$ such that $\min_{1\leq i\leq k}|\VectCoord{x}{i}|\geq\max_{2\leq i\leq k}\lambda_i$
\begin{align*}
    f_{\bm{\lambda}}(\VectCoord{x}{1},\ldots,\VectCoord{x}{k}):=\prod_{i=2}^k\un_{\{|\VectCoord{x}{i-1}\land\VectCoord{x}{i}|<\lambda_i\}}.
\end{align*}
Let us prove that the hereditary Assumption \ref{Assumption3} is satisfied by $f_{\bm{\lambda}}$. Recall that for $\Vect{x}=(\VectCoord{x}{1},\ldots,\VectCoord{x}{k})\in\Delta^k$, $\mathcal{S}^k(\Vect{x})-1$ denotes the last generation at which two or more vertices among $\VectCoord{x}{1},\ldots,\VectCoord{x}{k}$ share a common ancestor. Let $p\geq\max_{2\leq i\leq k}\lambda_i$ and $\Vect{x}\in\Delta^k$ such that $p\leq\min_{1\leq i\leq k}|\VectCoord{x}{i}|$. If $\mathcal{S}^k(\Vect{x})\leq p$ then, for any $\Vect{z}\in\llbracket(\VectCoord{x}{1})_{p},\VectCoord{x}{1}\rrbracket\times\cdots\times\llbracket(\VectCoord{x}{k})_{p},\VectCoord{x}{k}\rrbracket$ and $i\in\{2,\ldots,k\}$, $|\VectCoord{x}{i-1}\land\VectCoord{x}{i}|<\lambda_i$ if and only if $|\VectCoord{z}{i-1}\land\VectCoord{z}{i}|<\lambda_i$, meaning that $f_{\lambda}(\Vect{x})=f_{\lambda}(\Vect{z})$. Consequently, Assumption \ref{Assumption3} holds for $\mathfrak{g}=\max_{2\leq i\leq k}\lambda_i$. We conclude using Theorem \ref{GenTh5}. 
\end{proof}

\noindent We now prove Corollary \ref{GenTh2}:

\begin{proof}[Proof of Corollary \ref{GenTh2}]
Recall that for $\Vect{x}=(\VectCoord{x}{1},\ldots,\VectCoord{x}{k})\in\Delta^k$, $\mathcal{S}^k(\Vect{x})-1$ denotes the last generation at which two or more vertices among $\VectCoord{x}{1},\ldots,\VectCoord{x}{k}$ share a common ancestor and for $m\in\N^*$, recall that
$$ f_m(\Vect{x})=\un_{\{\mathcal{S}^k(\Vect{x})\leq m\}}. $$ \\
First, note that the hereditary Assumption \ref{Assumption3} is satisfied by $f_m$. Indeed, if $p\geq m$ and $\Vect{x}\in\Delta^k$ such that $p\leq\min_{1\leq i\leq k}|\VectCoord{x}{i}|$, then $\mathcal{S}^k(\Vect{x})\leq p$ implies that for any $\Vect{z}\in\llbracket(\VectCoord{x}{1})_{p},\VectCoord{x}{1}\rrbracket\times\cdots\times\llbracket(\VectCoord{x}{k})_{p},\VectCoord{x}{k}\rrbracket$, we have $\mathcal{S}^k(\Vect{z})=\mathcal{S}^k(\Vect{x})$. Thus, $\mathcal{S}^k((\VectCoord{x}{1})_p,\ldots,(\VectCoord{x}{k})_p)\leq m$. Moreover, by definition, $\mathcal{S}^k((\VectCoord{x}{1})_p,\ldots,(\VectCoord{x}{k})_p)\leq m$ implies $\mathcal{S}^k(\Vect{x})\leq m$. Consequently, Assumption \ref{Assumption3} holds for $\mathfrak{g}=m$. \\
We then deduce the converge of the trace in \eqref{Th3Trace} by using Theorem \ref{GenTh5}. \\
We now move to the limit law of $(\mathcal{S}^k(\VectCoord{\mathcal{X}}{n}))$ in \eqref{ConvMRCA}. Note, by definition, that
\begin{align*}
    \P^*\big(\mathcal{S}^k(\VectCoord{\mathcal{X}}{n})\leq m\big)=\frac{1}{\P^*(|\Delta^k(\mathcal{D}_n)|>0)}\E^*\Big[\frac{\mathcal{A}^k(\mathcal{D}_n,f_m)}{\mathcal{A}^k(\mathcal{D}_n,1)}\un_{\{|\Delta^k(\mathcal{D}_n)|>0\}}\Big],
\end{align*}
so $\P^*(\mathcal{S}^k(\VectCoord{\mathcal{X}}{n})\leq m)$ goes to $\Eb^*[\mathcal{A}_{\infty}^k(f_m)/(W_{\infty})^k]$ when $n$ goes to $\infty$ thanks to Theorem \ref{GenTh5} with $f=f_m$ and $g=1$ together with the fact that $\lim_{n\to\infty}\P^*(|\Delta^k(\mathcal{D}_n)|>0)=1$. It is left to show that $\lim_{m\to\infty}\mathcal{A}_{\infty}^k(f_m)=(W_{\infty})^k$. For that, we use Lemma \ref{CauchyCarre} with $f=1$ and $\Vect{p}=(l,\ldots,l)\in\ProdSet{(\N^*)}{k}$
\begin{align*}
    \sup_{l>m}\Eb^*\big[\big|\mathcal{A}^k_{l}(f_m)-\mathcal{A}_l^k(1)\big|^2\big]\underset{m\to\infty}{\longrightarrow}0.
\end{align*}
Moreover, $\lim_{l\to\infty}\mathcal{A}^k_{l}(1)=(W_{\infty})^k$ and $\lim_{l\to\infty}\mathcal{A}^k_{l}(f_m)=\mathcal{A}_{\infty}^k(f_m)$ so $(\mathcal{A}_{\infty}^k(f_m))_m$ converges to $(W_{\infty})^k$ in $L^2(\Pb^*)$, which allows to end the proof.
\end{proof}
\noindent We now turn to the proof of Corollary \ref{GenTh3}.

\begin{proof}[Proof of Corollary \ref{GenTh3}]
Recall that for any $1\leq d<q\in\N^*$, for an increasing collection $\Xi=(\bm{\Xi}_i)_{0\leq i\leq d}$ of partitions of $\{1,\ldots,q\}$, for all $\Vect{x}=(\VectCoord{x}{1},\ldots,\VectCoord{x}{q})\in\Delta^q$ and all $\bm{t}=(t_1,\ldots,t_{d})\in\ProdSet{\N}{d}$ such that $t_1<t_2<\cdots<t_d$,  
$$f^{d}_{\bm{t},\Xi}(\Vect{x})=\prod_{i=1}^{d}\un_{\Gamma^i_{\bm{t},\Xi}}(\Vect{x}), $$
where $\Gamma^i_{\bm{t},\Xi}=\Upsilon_{t_i-1,\bm{\Xi}_{i-1}}\cap\Upsilon_{t_i,\bm{\Xi}_i}$ and for any $r\in\{1,\ldots,d\}$ and any $m\in\N^*$, $\Vect{x}$ belongs to $\Upsilon_{m,\bm{\Xi}_{r}}$ if and only if
\begin{align*}
   \forall\bm{B}\in\bm{\Xi}_r, \forall i_1,i_2\in\bm{B}:(\VectCoord{x}{i_1})_{m}=(\VectCoord{x}{i_2})_{m},
\end{align*}
and for $r\not=0$
\begin{align*}
    \forall\bm{B}\not=\tilde{\bm{B}}\in\bm{\Xi}_r, \forall i_1\in\bm{B}, i_2\in\tilde{\bm{B}}:(\VectCoord{x}{i_1})_{m}\not=(\VectCoord{x}{i_2})_{m},
\end{align*}
where we recall that $(\VectCoord{x}{i})_m$ denotes the ancestor of $\VectCoord{x}{i}$ in generation $m$ if exists, $(\VectCoord{x}{i})_m=\VectCoord{e}{i}$ otherwise. Recall that $\mathcal{C}^k_{\mathfrak{g}}=\{\Vect{y}\in\Delta^q;\; \mathcal{S}^q(\Vect{y})\leq\mathfrak{g}\}$ where $\mathcal{S}^q(\Vect{y})-1$ is the last generation at which two or more vertices among $\VectCoord{y}{1},\ldots,\VectCoord{y}{q}$ share a common ancestor. Let $p\geq t_d$ such that $\min_{1\leq i\leq q}\VectCoord{x}{i}\geq p$ and $\Vect{x}\in\mathcal{C}^k_{p}$. If $\Vect{x}\in\cap_{j=1}^d\Gamma^j_{\bm{t},\Xi}$, then $(\VectCoord{z}{i})_{t}=(\VectCoord{x}{i})_{t}$ for all $\Vect{z}\in\llbracket(\VectCoord{x}{1})_{p},\VectCoord{x}{1}\rrbracket\times\cdots\times\llbracket(\VectCoord{x}{q})_{p},\VectCoord{x}{q}\rrbracket$, $1\leq i\leq q$ and $t\in\{0,\ldots,p\}$ thus giving $
((\VectCoord{x}{1})_p,\ldots,(\VectCoord{x}{q})_p)\in\cap_{j=1}^d\Gamma^j_{\bm{t},\Xi}$. Moreover, by definition, $((\VectCoord{x}{1})_p,\ldots,(\VectCoord{x}{q})_p))\in\cap_{j=1}^d\Gamma^j_{\bm{t},\Xi}$ implies $\Vect{x}\in\cap_{j=1}^d\Gamma^j_{\bm{t},\Xi}$. Consequently, $f^{d}_{\bm{t},\Xi}$ satisfies Assumption \ref{Assumption3} with $\mathfrak{g}=t_d$ and this prove that the convergence in \eqref{Th3Trace} holds. \\
We move to the limit law of $(\pi^{k,n})$ in \eqref{Th3partition}. Recall the definition of $\mathcal{S}^{k,n}_i$ in \eqref{CoaTimes}. First, note that
\begin{align*}
    \P^*(\pi^{k,n}_{m_0}=\bm{\pi}_0,\ldots,\pi^{k,n}_{m_{\ell}}=\bm{\pi}_{\ell}) =\P^*\Big(\bigcap_{i=1}^{\ell}\big\{\pi^{k,n}_{m_{i-1}}=\bm{\pi}_{i-1},\pi^{k,n}_{m_{i}}=\bm{\pi}_{i},m_{i-1}<\mathcal{S}^{k,n}_i\leq m_i\big\}\Big).
\end{align*}
Indeed, for all $1\leq i\leq\ell$, $|\bm{\pi}_{i-1}|<|\bm{\pi}_i|$ so the interval $(m_{i-1},m_i]$ necessarily contains at least one coalescent time. But since $\bm{\pi}_{0}=\{\{1,\ldots,k\}\}$ and $\bm{\pi}_{\ell}=\{\{1\},\ldots,\{k\}\}$, $\cup_{i=1}^{\ell}(m_{i-1},m_i]$ can not contain more than $\ell$ coalescent times so $\mathcal{S}^{k,n}_{i}$ is the only one belonging to $(m_{i-1},m_i]$. We now write
\begin{align*}
    \P^*\Big(\bigcap_{i=1}^{\ell}\big\{\pi^{k,n}_{m_{i-1}}&=\bm{\pi}_{i-1},m_{i-1}<\mathcal{S}^{k,n}_i\leq m_i\big\}\Big) \\ & =\sum_{s_1=m_0+1}^{m_1}\cdots\sum_{s_{\ell}=m_{\ell-1}+1}^{m_{\ell}}\P^*\Big(\bigcap_{i=1}^{\ell}\big\{\pi^{k,n}_{m_{i-1}}=\bm{\pi}_{i-1},\pi^{k,n}_{m_{i}}=\bm{\pi}_{i},\mathcal{S}^{k,n}_i=s_i\big\}\Big) \\ & = \sum_{s_1=m_0+1}^{m_1}\cdots\sum_{s_{\ell}=m_{\ell-1}+1}^{m_{\ell}}\P^*\Big(\bigcap_{i=1}^{\ell}\big\{\pi^{k,n}_{s_{i}-1}=\bm{\pi}_{i-1},\pi^{k,n}_{s_{i}}=\bm{\pi}_{i}\big\}\Big),
\end{align*}
Moreover, $\pi^{k,n}_{s_{i}-1}=\bm{\pi}_{i-1},\pi^{k,n}_{s_{i}}=\bm{\pi}_{i}$ means nothing but $\VectCoord{\mathcal{X}}{n}\in \Gamma^i_{\bm{s},\Pi}$ and it follows that
\begin{align*}
    \P^*\Big(\bigcap_{i=1}^{\ell}\big\{\pi^{k,n}_{s_{i}-1}=\bm{\pi}_{i-1},\pi^{k,n}_{s_{i}}=\bm{\pi}_{i}\big\}\Big)&=\E^*\big[f^{\ell}_{\bm{s},\Pi}(\VectCoord{\mathcal{X}}{n})\big] \\ & =\frac{1}{\P^*(|\Delta^k(\mathcal{D}_n)|>0)}\E^*\Big[\frac{\mathcal{A}^k(\mathcal{D}_n,f^{\ell}_{\bm{s},\Pi})}{\mathcal{A}^k(\mathcal{D}_n,1)}\un_{\{|\Delta^k(\mathcal{D}_n)|\geq k\}}\Big],
\end{align*}
where we have used the definition of $\VectCoord{\mathcal{X}}{n}$ (see \eqref{LawVertices2} and \eqref{LawVertices1}) in the last equality. Since $f^{\ell}_{\bm{s},\Pi}$ satisfies the hereditary Assumption \ref{Assumption3}, we finally get \eqref{Th3Trace} from \eqref{ConvA} with $f=f^{\ell}_{\bm{s},\Pi}$ and by \eqref{ConvAQuotient} with $g=1$
\begin{align*}
    \lim_{n\to\infty}\P^*(\pi^{k,n}_{m_0}=\bm{\pi}_0,\ldots,\pi^{k,n}_{m_{\ell}}=\bm{\pi}_{\ell})=\sum_{s_1=m_0+1}^{m_1}\cdots\sum_{s_{\ell}=m_{\ell-1}+1}^{m_{\ell}}\Eb^*\Big[\frac{\mathcal{A}_{\infty}^k(f^{\ell}_{\bm{s},\Pi})}{(W_{\infty})^k}\Big].
\end{align*}
We now compute the conditional expectation of $\mathcal{A}_{\infty}^k(f^{\ell}_{\bm{s},\Pi})$ conditionally given the sigma-algebra $\mathcal{F}_{s_p-1}=\sigma((x,V(x));|x|<s_p)$. Start with $p=\ell$. Let $s_{i}\in\{m_{i-1}+1,\ldots,m_{i}\}$ for all $i\in\{1,\ldots,\ell\}$. Using the definition of $\mathcal{A}_{\infty}^k(f^{\ell}_{\bm{s},\Pi})$ and the fact that $\Vect{x}\in\Delta^k_l\cap\Gamma^{\ell}_{\bm{s},\Pi}$ for $l>s_{\ell}$ implies $\mathcal{S}^k(\Vect{x})\leq s_{\ell}$, we obtain, on the set of non-extinction
\begin{align*}
    \Eb\big[\mathcal{A}_{\infty}^k(f^{\ell}_{\bm{s},\Pi})|\mathcal{F}_{s_{\ell}}\big]=\lim_{l\to\infty}\Eb\Big[\sum_{\Vect{x}\in\Delta^k_l}f^{\ell}_{\bm{s},\Pi}(\Vect{x})e^{-\langle\Vect{1},V(\Vect{x})\rangle_{k}}|\mathcal{F}_{s_{\ell}}\Big] =\sum_{\Vect{x}\in\Delta^k_{s_{\ell}}}f^{\ell}_{\bm{s},\Pi}(\Vect{x})e^{-\langle\Vect{1},V(\Vect{x})\rangle_{k}},
\end{align*}
since $s_{\ell}-1$ corresponds to the last generation at which two or more vertices among $\VectCoord{x}{1},\ldots,\VectCoord{x}{l}$ share a common ancestor and we recall that $\langle\Vect{1},V(\Vect{x})\rangle_{k}=\sum_{i=1}^kV(\VectCoord{x}{i})$. In particular, these vertices don't share any common ancestor in generation $s_{\ell}$ and last equality comes from independence of the increments of the branching random walk $(\T,(V(x),x\in\T))$ together with the fact that $\psi(1)=0$. Before going any further, let us define a transformation of the increasing collection $\Pi=(\bm{\pi}_i)_{0\leq i\leq\ell}$ of partitions of $\{1,\ldots,k\}$. We build from $\Pi$ (which is by definition a collection of partitions of the set $\{1,\ldots,k\}$) a new collection $\Pi^{\ell-1}=(\tilde{\bm{\pi}}_i)_{0\leq i\leq\ell-1}$ of partitions of the set $\{1,\ldots,|\bm{\pi}_{\ell-1}|\}$ as follows:
\begin{itemize}
    \item[$\bullet$]  $\tilde{\bm{\pi}}_{\ell-1}=\{\{1\},\ldots,\{|\bm{\pi}_{\ell-1}|\}\}$; \\
    \item[$\bullet$]  for any $1\leq i\leq\ell-2$ and any $1\leq j\leq|\bm{\pi}_i|$, the $j$-th block $\bm{B}_j^i$ of the partition $\bm{\pi}_i$ is the union of $\mathfrak{b}_{\ell-1}(\bm{B}_j^i)\geq 1$ block(s) of the partition $\bm{\pi}_{\ell-1}$. We then denote by $\tilde{\bm{B}}^i_j$ the subset of $\{1,\ldots,|\bm{\pi}_{\ell-1}|\}$ composed of all indices of these $\mathfrak{b}_{\ell-1}(\bm{B}_j^i)$ block(s) and let $\tilde{\bm{\pi}}_{i}=\{\tilde{\bm{B}}^i_1,\ldots,\tilde{\bm{B}}^i_{|\bm{\pi}_i|}\}$. By definition, $\tilde{\bm{\pi}}_0$ remains a one-block partition: $\tilde{\bm{\pi}}_0=\{\{1,\ldots,|\bm{\pi}_{\ell-1}|\}\}$.
\end{itemize}
Note that for any $0\leq i\leq\ell-1$, $|\tilde{\bm{\pi}}_{i}|=|\bm{\pi}_{i}|$ and for any $0\leq i\leq\ell-2$, $1\leq j\leq |\bm{\pi}_i|$, $b_i(\bm{B}_j)=\tilde{b}_i(\tilde{\bm{B}}_j)$, where $\tilde{\bm{B}}_j\in\tilde{\bm{\pi}_i}$ is the union of $\tilde{b}_i(\tilde{\bm{B}}_j)\geq 1$ block(s) of $\tilde{\pi}_{i+1}$. \\

\vspace{0.1cm}

\begin{example}\label{ProcedurePartition} 
If $\Pi$ is defined by $\bm{\pi}_4=\{\{1\},\{2\},\{3\},\{4\},\{5\}\}$,$\bm{\pi}_3=\{\{1,3\},\{2\},\{4\},\{5\}\}$, $\bm{\pi}_2=\{\{1,3\},\{2,5\},\{4\}\}$, $\bm{\pi}_1=\{\{1,3,4\},\{2,5\}\}$ and $\bm{\pi}_0=\{\{1,2,3,4,5\}\}$ then we have: \\
$\tilde{\bm{\pi}}_3=\{\{1\},\{2\},\{3\},\{4\}\}$, $\tilde{\bm{\pi}}_2=\{\{1\},\{2,4\},\{3\}\}$,$\tilde{\bm{\pi}}_1=\{\{1,3\},\{2,4\}\}$, and $\tilde{\bm{\pi}}_0=\{\{1,2,3,4\}\}$.
\end{example}
\noindent If we set $\Pi^{\ell}:=\Pi$, then for any $i\in\{0,\ldots,\ell-1\}$, let $\Pi^{i}$ be the collection of partitions of $\{1,\ldots,|\bm{\pi}_i|\}$ resulting from the previous procedure applied to $\Pi^{i+1}$. Note that $\Pi^i$ is an increasing collection of partitions of $\{1,\ldots,|\bm{\pi}_i|\}$. This construction is a way of preserving the genealogical information through the generations. \\
Let $\bm{s}^{\ell-1}=(s_1,\ldots,s_{\ell-1})$ and recall the definitions regarding partitions in \eqref{betaDef}. One can now notice that, since that the number of vertices of the $k$-tuple $\Vect{x}\in\Delta^k_{s_{\ell}}$ sharing the same parent $\VectCoord{u}{j}$ is $b_{\ell-1}(\bm{B}_j)$ (where we recall that $b_{\ell-1}(\bm{B}_j)$ stands for $b_{\ell-1}(\bm{B}^{\ell-1}_j)$), we have 
\begin{align*}
    \sum_{\Vect{x}\in\Delta^k_{s_{\ell}}}f^{\ell}_{\bm{s},\Pi}(\Vect{x})e^{-\langle\Vect{1},V(\Vect{x})\rangle_{k}}=\sum_{\bm{u}\in\Delta^{|\bm{\pi}_{\ell-1}|}_{s_{\ell}-1}}f^{\ell-1}_{\bm{s}^{\ell-1},\Pi^{\ell-1}}(\Vect{u})\prod_{j=1}^{|\bm{\pi}_{\ell-1}|}\sum_{\VectCoord{\Vect{x}}{j}\in\Delta^{b_{\ell-1}(\bm{B}_j)}_{s_{\ell}}}\prod_{i=1}^{b_{\ell-1}(\bm{B}_j)}&\un_{\{(\VectCoord{x}{j,i})^*=\VectCoord{u}{j}\}} \\ & \times e^{-V(\VectCoord{x}{j,i})},
\end{align*}
where $\VectCoord{\Vect{x}}{j}=(\VectCoord{x}{j,1},\ldots,\VectCoord{x}{j,b_{\ell-1}(\bm{B}_j)})$ and $(\VectCoord{x}{j,i})^*$ is the parent of $\VectCoord{x}{j,i}$. Moreover, by definition, $b_{\ell-1}(\bm{B}_j)=|\bm{B}^{\ell-1}_j|$ (it comes from the fact that $\bm{\pi}_{\ell}=\{\{1\},\ldots,\{k\}\}$) so
\begin{align*}
    \prod_{j=1}^{|\bm{\pi}_{\ell-1}|}\sum_{\VectCoord{\Vect{x}}{j}\in\Delta^{b_{\ell-1}(\bm{B}_j)}_{s_{\ell}}}\prod_{i=1}^{b_{\ell-1}(\bm{B}_j)}\un_{\{(\VectCoord{x}{j,i})^*=\VectCoord{u}{j}\}}e^{-V(\VectCoord{x}{j,i})}&= e^{-\langle\bm{\beta}^{\ell-1},V(\Vect{u})\rangle_{|\bm{\pi}_{\ell-1}|}}\prod_{j=1}^{|\bm{\pi}_{\ell-1}|}\sum_{\VectCoord{\Vect{x}}{j}\in\Delta^{b_{\ell-1}(\bm{B}_j)}_{s_{\ell}}} \\ & \times \prod_{i=1}^{b_{\ell-1}(\bm{B}_j)}\un_{\{(\VectCoord{x}{j,i})^*=\VectCoord{u}{j}\}}e^{-V_{\VectCoord{u}{j}}(\VectCoord{x}{j,i})},
\end{align*}
where $\bm{\beta}^{\ell-1}=(|\bm{B}^{\ell-1}_1|,\ldots,|\bm{B}^{\ell-1}_{|\bm{\pi}_{\ell-1}|}|)$ and $V_{\VectCoord{u}{j}}(\VectCoord{x}{j,i})=V(\VectCoord{x}{j,i})-V(\VectCoord{u}{j})$. By independence of the increments of the branching random walk $(\T,(V(x),x\in\T))$, since $\psi(1)=0$
\begin{align*}
    \Eb\Big[\sum_{\Vect{x}\in\Delta^k_{s_{\ell}}}f^{\ell}_{\bm{s},\Pi}(\Vect{x})e^{-\langle\Vect{1},V(\Vect{x})\rangle_{k}}|\mathcal{F}_{s_{\ell}-1}\Big]&= \mathcal{A}^{|\bm{\pi}_{\ell-1}|}_{s_{\ell}-1}\big(f^{\ell-1}_{\bm{s}^{\ell-1},\Pi^{\ell-1}},\bm{\beta}^{\ell-1}\big)\prod_{j=1}^{|\bm{\pi}_{\ell-1}|}c_{b_{\ell-1}(\bm{B}_j)}(\bm{1})\underset{|\mathfrak{B}|\geq 2}{\prod_{\mathfrak{B}\in\bm{\pi}_{\ell}}}e^{\psi(|\mathfrak{B}|)} \\ & = \mathcal{A}^{|\bm{\pi}_{\ell-1}|}_{s_{\ell}-1}\big(f^{\ell-1}_{\bm{s}^{\ell-1},\Pi^{\ell-1}},\bm{\beta}^{\ell-1}\big)\prod_{j=1}^{|\bm{\pi}_{\ell-1}|}c_{b_{\ell-1}(\bm{B}_j)}(\bm{\beta}^{\ell-1}_j),   
\end{align*}
where $\bm{\beta}^{\ell-1}_j:=(\beta^{\ell-1}_{j,1},\ldots,\beta^{\ell-1}_{j,b_{\ell-1}(\bm{B}_j)})=(1,\ldots,1)$, see \eqref{betaDef}. We also recall that $\mathcal{A}^m_l(g,\bm{\beta})=\sum_{\Vect{x}\in\Delta^m_l}g(\Vect{x})e^{-\langle\bm{\beta},V(\bm{x})\rangle_m}$ and see Assumption \ref{Assumption2} for the definition of $c_l(\bm{\beta})$. Now recall that $\Pi^{\ell-2}$ is the collection of partitions of $\{1,\ldots,|\bm{\pi}_{\ell-2}|\}$ obtain from $\Pi^{\ell-1}$ with the same procedure as above (see Example \ref{ProcedurePartition}). Let $\bm{s}^{\ell-2}=(s_1,\ldots,s_{\ell-2})$. Again, exactly $b_{\ell-2}(\bm{B}_j)$ vertices in generation $s_{\ell-1}$ are sharing the same parent $\VectCoord{z}{j}$ so
\begin{align*}
    \mathcal{A}^{|\bm{\pi}_{\ell-1}|}_{s_{\ell}-1}\big(f^{\ell-1}_{\bm{s}^{\ell-1},\Pi^{\ell-1}},\bm{\beta}^{\ell-1}\big)=&\sum_{\bm{z}\in\Delta^{|\bm{\pi}_{\ell-2}|}_{s_{\ell-1}-1}}f^{\ell-2}_{\bm{s}^{\ell-2},\Pi^{\ell-2}}(\Vect{z})\prod_{j=1}^{|\bm{\pi}_{\ell-2}|}\sum_{\VectCoord{\Vect{u}}{j}\in\Delta^{b_{\ell-2}(\bm{B}_j)}_{s_{\ell-1}}}\prod_{i=1}^{b_{\ell-2}(\bm{B}_j)}\un_{\{(\VectCoord{u}{j,i})^*=\VectCoord{z}{j}\}} \\ & \times e^{-\beta^{\ell-2}_{j,i}V(\VectCoord{u}{j,i})}\sum_{\VectCoord{\Vect{x}}{j}\in\Delta^{b_{\ell-2}(\bm{B}_j)}_{s_{\ell}-1}}\un_{\{\VectCoord{x}{j,i}\geq\VectCoord{u}{j,i}\}}e^{-\beta^{\ell-2}_{j,i}V_{\VectCoord{u}{j,i}}(\VectCoord{x}{j,i})},
\end{align*}
where $\VectCoord{\Vect{u}}{j}=(\VectCoord{u}{j,1},\ldots,\VectCoord{u}{j,b_{\ell-2}(\bm{B}_j)})$, $\VectCoord{\Vect{x}}{j}=(\VectCoord{x}{j,1},\ldots,\VectCoord{x}{j,b_{\ell-2}(\bm{B}_j)})$ and $V_{\VectCoord{u}{j,i}}(\VectCoord{x}{j,i})$ is the increment $V(\VectCoord{x}{j,i})-V(\VectCoord{u}{j,i})$. Then, by independence of the increments of the branching random walk $(\T,(V(x),x\in\T))$, denoting $s_{\ell}^*=s_{\ell}-s_{\ell-1}-1$
\begin{align*}
    \Eb\Big[\mathcal{A}^{|\bm{\pi}_{\ell-1}|}_{s_{\ell}-1}\big(f^{\ell-1}_{\bm{s}^{\ell-1},\Pi^{\ell-1}},\bm{\beta}^{\ell-1}\big)|\mathcal{F}_{s_{\ell-1}}\Big]=\sum_{\bm{z}\in\Delta^{|\bm{\pi}_{\ell-2}|}_{s_{\ell-1}-1}}&f^{\ell-2}_{\bm{s}^{\ell-2},\Pi^{\ell-2}}(\Vect{z})\prod_{j=1}^{|\bm{\pi}_{\ell-2}|}\sum_{\VectCoord{\Vect{u}}{j}\in\Delta^{b_{\ell-2}(\bm{B}_j)}_{s_{\ell-1}}}\prod_{i=1}^{b_{\ell-2}(\bm{B}_j)} \\ & \times\un_{\{(\VectCoord{u}{j,i})^*=\VectCoord{z}{j}\}} e^{-\beta^{\ell-2}_{j,i}V(\VectCoord{u}{j,i})}e^{s_{\ell}^*\psi(\beta^{\ell-2}_{j,i})},
\end{align*}
which is also equal to
\begin{align*}
    \sum_{\bm{z}\in\Delta^{|\bm{\pi}_{\ell-2}|}_{s_{\ell-1}-1}}f^{\ell-2}_{\bm{s}^{\ell-2},\Pi^{\ell-2}}(\Vect{z})\prod_{j=1}^{|\bm{\pi}_{\ell-2}|}\sum_{\VectCoord{\Vect{u}}{j}\in\Delta^{b_{\ell-2}(\bm{B}_j)}_{s_{\ell-1}}}\prod_{i=1}^{b_{\ell-2}(\bm{B}_j)}&\un_{\{(\VectCoord{u}{j,i})^*=\VectCoord{z}{j}\}} e^{-\beta^{\ell-2}_{j,i}V(\VectCoord{u}{j,i})} \\ & \times\prod_{j=1}^{|\bm{\pi}_{\ell-2}|}\prod_{i=1}^{b_{\ell-2}(\bm{B}_j)}e^{s_{\ell}^*\psi(\beta^{\ell-2}_{j,i})}.
\end{align*}
Moreover, since $\sum_{i=1}^{b_{\ell-2}(\bm{B}_j)}\beta^{\ell-2}_{j,i}=|\bm{B}^{\ell-2}_j|$ (see \eqref{betaDef}), we have
\begin{align*}
    &\prod_{j=1}^{|\bm{\pi}_{\ell-2}|}\sum_{\VectCoord{\Vect{u}}{j}\in\Delta^{b_{\ell-2}(\bm{B}_j)}_{s_{\ell-1}}}\prod_{i=1}^{b_{\ell-2}(\bm{B}_j)}\un_{\{(\VectCoord{u}{j,i})^*=\VectCoord{z}{j}\}} e^{-\beta^{\ell-2}_{j,i}V(\VectCoord{u}{j,i})} \\ & =e^{-\langle\Vect{\bm{\beta}^{\ell-2}},V(\Vect{z})\rangle_{|\bm{\pi}_{\ell-2}|}}\prod_{j=1}^{|\bm{\pi}_{\ell-2}|}\sum_{\VectCoord{\Vect{u}}{j}\in\Delta^{b_{\ell-2}(\bm{B}_j)}_{s_{\ell-1}}}\prod_{i=1}^{b_{\ell-2}(\bm{B}_j)}\un_{\{(\VectCoord{u}{j,i})^*=\VectCoord{z}{j}\}} e^{-\beta^{\ell-2}_{j,i}V_{\VectCoord{z}{j}}(\VectCoord{u}{j,i})},
\end{align*}
with $\bm{\beta}^{\ell-2}=(|\bm{B}^{\ell-2}_1|,\ldots,|\bm{B}^{\ell-2}_{|\bm{\pi}_{\ell-2}|}|)$ and again, by independence of the increments of the branching random walk $(\T,(V(x),x\in\T))$, using that $\prod_{j=1}^{|\bm{\pi}_{\ell-2}|}\prod_{i=1}^{b_{\ell-2}(\bm{B}_j)}e^{s_{\ell}^*\psi(\beta^{\ell-2}_{j,i})}=\prod_{\mathfrak{B}\in\bm{\pi}_{\ell-1}}e^{s_{\ell}^*\psi(|\mathfrak{B}|)}=\prod_{\mathfrak{B}\in\bm{\pi}_{\ell-1}, |\mathfrak{B}|\geq 2}e^{s_{\ell}^*\psi(|\mathfrak{B}|)}$, we have
\begin{align*}
    \Eb\Big[\mathcal{A}^{|\bm{\pi}_{\ell-1}|}_{s_{\ell}-1}\big(f^{\ell-1}_{\bm{s}^{\ell-1},\Pi^{\ell-1}},\bm{\beta}^{\ell-1}\big)|\mathcal{F}_{s_{\ell-1}-1}\Big]= \mathcal{A}^{|\bm{\pi}_{\ell-2}|}_{s_{\ell-1}-1}\big(f^{\ell-2}_{\bm{s}^{\ell-2},\Pi^{\ell-2}},\bm{\beta}^{\ell-2}\big)&\prod_{j=1}^{|\bm{\pi}_{\ell-2}|}c_{b_{\ell-2}(\bm{B}_j)}(\bm{\beta}^{\ell-2}_j) \\ & \times\underset{|\mathfrak{B}|\geq 2}{\prod_{\mathfrak{B}\in\bm{\pi}_{\ell-1}}}e^{s_{\ell}^*\psi(|\mathfrak{B}|)},
\end{align*}
where $\bm{\beta}^{\ell-2}_j=(\beta^{\ell-2}_{j,1},\ldots,\beta^{\ell-2}_{j,b_{\ell-2}(\bm{B}_j)})$. Thus, we obtain
\begin{align*}
    \Eb\Big[\sum_{\Vect{x}\in\Delta^k_{s_{\ell}}}f^{\ell}_{\bm{s},\Pi}(\Vect{x})e^{-\langle\Vect{1},V(\Vect{x})\rangle_{k}}|\mathcal{F}_{s_{\ell-1}-1}\Big]=\mathcal{A}^{|\bm{\pi}_{\ell-2}|}_{s_{\ell-1}-1}\big(f^{\ell-2}_{\bm{s}^{\ell-2},\Pi^{\ell-2}},\bm{\beta}^{\ell-2}\big)\prod_{i=\ell-1}^{\ell}&\prod_{j=1}^{|\bm{\pi}_{i-1}|}c_{b_{i-1}(\bm{B}_j)}(\bm{\beta}^{i-1}_j) \\ & \times\underset{|\mathfrak{B}|\geq 2}{\prod_{\mathfrak{B}\in\bm{\pi}_{i}}}e^{s_{\ell}^*\psi(|\mathfrak{B}|)}. 
\end{align*}
By induction on $2\leq p\leq\ell$, we finally get, on the set of non-extinction
\begin{align*}
     \Eb\big[\mathcal{A}_{\infty}^k(f^{\ell}_{\bm{s},\Pi})|\mathcal{F}_{s_{p}-1}\big]=\mathcal{A}^{|\bm{\pi}_{p-1}|}_{s_{p}-1}\big(f^{p-1}_{\bm{s}^{p-1},\Pi^{p-1}},\bm{\beta}^{p-1}\big)\prod_{i=p}^{\ell}\prod_{j=1}^{|\bm{\pi}_{i-1}|}c_{b_{i-1}(\bm{B}_j)}(\bm{\beta}^{i-1}_j)\underset{|\mathfrak{B}|\geq 2}{\prod_{\mathfrak{B}\in\bm{\pi}_{i}}}e^{s_{i+1}^*\psi(|\mathfrak{B}|)}.
\end{align*}
Taking $p=2$ in the above formula, we have, on the set of non-extinction
\begin{align*}
    \Eb\big[\mathcal{A}_{\infty}^k(f^{\ell}_{\bm{s},\Pi})|\mathcal{F}_{s_{2}-1}\big]=\mathcal{A}^{|\bm{\pi}_{1}|}_{s_{2}-1}\big(f^{1}_{\bm{s}^{1},\Pi^{1}},\bm{\beta}^{1}\big)\prod_{i=2}^{\ell}\prod_{j=1}^{|\bm{\pi}_{i-1}|}c_{b_{i-1}(\bm{B}_j)}(\bm{\beta}^{i-1}_j)\underset{|\mathfrak{B}|\geq 2}{\prod_{\mathfrak{B}\in\bm{\pi}_{i}}}e^{s_{i+1}^*\psi(|\mathfrak{B}|)},
\end{align*}
where for any $i\in\{2,\ldots,\ell\}$, $s_i^*=s_i-s_{i-1}-1$ and $s_{\ell+1}^*=1$. Since $\sum_{j=1}^{b_0(\bm{B}_1)}|\bm{B}_j^1|=k$ (it comes from the fact that $\bm{\pi}_0=\{\{1,\ldots,k\}\}$), we have
\begin{align*}
     \Eb\big[\mathcal{A}^{|\bm{\pi}_{1}|}_{s_{2}-1}\big(f^{1}_{\bm{s}^{1},\Pi^{1}},\bm{\beta}^{1}\big)|\mathcal{F}_{s_{1}-1}\big]&=\sum_{|z|=s_1-1}e^{-kV(z)}c_{b_0(\bm{B}_1)}(\bm{\beta}^{1})\underset{|\mathfrak{B}|\geq 2}{\prod_{\mathfrak{B}\in\bm{\pi}_{1}}}e^{s_{2}^*\psi(|\mathfrak{B}|)} \\ & =\sum_{|z|=s_1-1}e^{-kV(z)}\prod_{j=1}^{|\bm{\pi}_0|}c_{b_0(\bm{B}_j)}(\bm{\beta}^0_j)\underset{|\mathfrak{B}|\geq 2}{\prod_{\mathfrak{B}\in\bm{\pi}_{1}}}e^{s_{2}^*\psi(|\mathfrak{B}|)},
\end{align*}
the last equality coming from the fact $\bm{\beta}^0_j=\bm{\beta}^1=(|\bm{B}^1_1|,\ldots,|\bm{B}^1_{|\bm{\pi}_1|}|)$. Finally, \begin{align*}
     \Eb\big[\mathcal{A}_{\infty}^k(f^{\ell}_{\bm{s},\Pi})\big]=e^{\psi(k)}\prod_{i=1}^{\ell}\prod_{j=1}^{|\bm{\pi}_{i-1}|}c_{b_{i-1}(\bm{B}_j)}(\bm{\beta}^{i-1}_j)\underset{|\mathfrak{B}|\geq 2}{\prod_{\mathfrak{B}\in\bm{\pi}_{i}}}e^{s_{i+1}^*\psi(|\mathfrak{B}|)},
\end{align*}
thus completing to proof.
\end{proof}
\noindent We end this subsection with the proof of Corollary \ref{GenTh4}.
\begin{proof}[Proof of Corollary \ref{GenTh4}]
First recall that for $1\leq\ell<k$, $\mathfrak{s}\in\N^*$ and $\bm{s}=(s_1,\ldots,s_{\ell})\in\N^{\times\ell}$ such that $s_1<\cdots<s_{\ell}\leq\mathfrak{s}$, for all $\Vect{x}\in\Delta^k$ such that $\min_{1\leq j\leq k}|\VectCoord{x}{j}|\geq\mathfrak{s}$,
$$F^{\ell}_{\bm{s}}(\Vect{x})=\sum_{\Xi\textrm{ increasing }}f^{\ell}_{\bm{s},\Xi}(\Vect{x}). $$
Recall that, by $\Xi\textrm{ increasing}$, we mean here that $\Xi=(\bm{\Xi}_i)_{0\leq i\leq\ell}$ is an increasing collection of partitions of $\{1,\ldots,k\}$. Since $f^{\ell}_{\bm{s},\Xi}$ satisfies the hereditary Assumption \ref{Assumption3}, the same goes for $F^{\ell}_{\bm{s}}$ by taking $\mathfrak{g}=\mathfrak{s}$. Using the linearity of $g\mapsto \mathcal{A}^k_l(g)$, we get \eqref{Th4Trace} thanks to Theorem \ref{GenTh5}. \\
Similarly as in the proof of Corollary \ref{GenTh3}, one can see that
\begin{align*}
    \big\{\mathcal{S}^{k,n}_1=s_1,\ldots,\mathcal{S}^{k,n}_{\ell}=s_{\ell},\mathcal{J}^{k,n}=\ell\big\}=\bigcup_{\Xi\textrm{ increasing }}\bigcap_{i=1}^{\ell}\big\{\pi^{k,n}_{s_i-1}=\bm{\Xi}_{i-1}, \pi^{k,n}_{s_i}=\bm{\Xi}_{i}\big\},
\end{align*}
and by definition of $\VectCoord{\mathcal{Y}}{n}$, we have
\begin{align*}
    \P^*\big(\mathcal{S}^{k,n}_1=s_1,\ldots,\mathcal{S}^{k,n}_{\ell}=s_{\ell},\mathcal{J}^{k,n}=\ell\big)&=\sum_{\Xi\textrm{ increasing }}\E^*\big[f^{\ell}_{\bm{s},\Xi}\big(\VectCoord{\mathcal{Y}}{n}\big)\big] \\ & =\frac{1}{\P^*(|\Delta^k(\mathcal{D}_n)|>0)}\E^*\Big[\frac{\mathcal{A}^k(\mathcal{D}_n,F^{\ell}_{\bm{s}})}{\mathcal{A}^k(\mathcal{D}_n,\un_{\mathcal{C}^k_{\mathfrak{s}}})}\un_{\{|\Delta^k(\mathcal{D}_n)|\geq k\}}\Big].
\end{align*}
We conclude using Theorem \ref{GenTh5}.
\end{proof}

\section{Proofs of Propositions \ref{GENPROPCONV2} and \ref{GENPROPCONV1}}\label{ProofsProp}

This section is devoted to the proofs of our two propositions. We show that relevant $k$-tuples of visited vertices are those in the set $\mathfrak{E}^{k,\cdot}$, see \eqref{DistinctExcur}.

\vspace{0.2cm}

\noindent Let us recall the well-known many-to-one lemma: 
\begin{lemm}[many-to-one]\label{many-to-one}
For any $p\in\N^*$ and any bounded function $\bm{h}:\R^k\to\R$
\begin{align*}
    \Eb[\bm{h}(S_1,\ldots,S_p)]=\Eb\Big[\sum_{|x|=p}e^{-V(x)}\bm{h}(V(x_1),\ldots,V(x_p))\Big],
\end{align*}
where $(S_i)_{i\in\N}$ is the real valued random walk defined in \eqref{RW}.
\end{lemm}

\noindent We now state and prove a lemma that will be useful all along this section. For any vertex $z\in\T$, recall that $T_z=\inf\{i\geq 1\; X_i=z\}$, the hitting time of $z$ and for any $\Vect{x}=(\VectCoord{x}{1},\ldots,\VectCoord{x}{q})\in\Delta^q$, $T_{\Vect{x}}=\max_{1\leq i\leq q}T_{\VectCoord{x}{i}}$. Recall that for any $j\in\N^*$, $T^j$ denotes the $j$-th return time to the parent $e^*$ of the root $e$. For $1\leq\ell<q$ two integers, $\bm{m}=(m_1,\ldots,m_{\ell})\in\ProdSet{\N}{\ell}$ such that $m_1<\cdots<m_{\ell}$ and $\Pi=(\bm{\pi}_i)_{0\leq i\leq\ell}$ an increasing collection of partitions of $\{1,\ldots,q\}$ that is to say $|\bm{\pi}_{i-1}|<|\bm{\pi}_i|$ with $\pi_{0}=\{\{1,\ldots,q\}\}$ and $\bm{\pi}_{\ell}=\{\{1\},\ldots,\{q\}\}$, recall the definition of $f^{\ell}_{\bm{m},\Pi}$ in \eqref{Def_f_partition}.

\begin{lemm}\label{EspProbaTAPartition}
Let $k\geq 2$ and $\mathfrak{a}\geq 1$ be two integers and assume $\kappa>2\mathfrak{a}k$. Let $q\in\{k,\ldots,2\mathfrak{a}k\}$ and $\bm{p}=(p_1,\ldots,p_q)\in\ProdSet{\N}{q}$. Under the Assumptions \ref{Assumption1} and \ref{Assumption2}, there exists a constant $\mathfrak{C}>0$ does not depending neither on $\bm{p}$, nor on $\bm{m}$ such that
\begin{align*}
    \Eb\Big[\underset{|\Vect{x}|=\bm{p}}{\sum_{\Vect{x}\in\Delta^q}}f^{\ell}_{\bm{m},\Pi}(\Vect{x})\P^{\mathcal{E}}(T_{\Vect{x}}<T^1)\Big]\leq\mathfrak{C},
\end{align*}
where $|\Vect{x}|=\Vect{p}$ means that $|\VectCoord{x}{i}|=p_i$ for any $i\in\{1,\ldots,q\}$. In particular, for any integer $m\in\N^*$, $q'\leq q$ and any distinct $i_1,\ldots,i_{q'}\in\{1,\ldots, k\}$, there exists a constant $C_{\ref{EspProbaTAPartition}}>0$ does not depending on $\Vect{p}$ such that
\begin{align}\label{EqLemmaProbaTA}
    \Eb\Big[\underset{|\Vect{x}|=\bm{p}}{\sum_{\Vect{x}\in\Delta^q}}\un_{\mathcal{C}^{q'}_m}(\Vect{x}_{q'})\un_{\mathcal{C}^{q-q'}_m}(\Bar{\Vect{x}}_{q'})\P^{\mathcal{E}}(T_{\Vect{x}}<T^1)\Big]\leq C_{\ref{EspProbaTAPartition}}\Big(\max_{1\leq i\leq q}p_i\Big)^{q'\land(q-q')}m^{q'\lor(q-q')-1}, 
\end{align}
where $\Vect{x}_{q'}:=(\VectCoord{x}{i_1},\ldots,\VectCoord{x}{i_{q'}})$ and $\Bar{\Vect{x}}_{q'}:=(\VectCoord{x}{i})_{i\in\{1,\ldots,k\}\setminus\{i_1,\ldots,i_{q'}\}}$.
\end{lemm}
\begin{proof}[Proof in the case $\cap_{j=1}^{\ell}\Gamma^j_{\bm{m},\Pi}\subset\{\Vect{x}\in\Delta^q;\; \mathcal{C}^q(\Vect{x})<\min_{1\leq i\leq q}p_i\}$] First recall that $\Pi^i$ is the partition of $\{1,\ldots,|\bm{\pi}_i|\}$ obtained via the procedure defined above Example \ref{ProcedurePartition} and for any $i\in\{1,\ldots,\ell\}$, any $j\in\{1,\ldots,|\bm{\pi}_{i-1}|\}$, the $j$-th block $\bm{B}^{i}_j$ of the partition $\bm{\pi}_{i-1}$ is the union of $b_{i-1}(\bm{B}_j)\geq 1$ block(s) of the partition $\bm{\pi}_i$. Note (see the proof of Corollary \ref{GenTh3}) that 
\begin{align*}
    \underset{|\Vect{x}|=\bm{p}}{\sum_{\Vect{x}\in\Delta^q}}f^{\ell}_{\bm{m},\Pi}(\Vect{x})\P^{\mathcal{E}}(T_{\Vect{x}}<T^1)=\sum_{\bm{z}\in\Delta^{|\bm{\pi}_{\ell-1}|}_{m_{\ell}-1}}f^{\ell-1}_{\bm{m}^{\ell-1},\Pi^{\ell-1}}(\Vect{z})\prod_{j=1}^{|\bm{\pi}_{\ell-1}|}&\sum_{\VectCoord{\Vect{u}}{j}\in\Delta^{b_{\ell-1}(\bm{B}_j)}_{m_{\ell}}}\prod_{i=1}^{b_{\ell-1}(\bm{B}_j)}\un_{\{(\VectCoord{u}{j,i})^*=\VectCoord{z}{j}\}} \\ &  \times\sum_{\Vect{x}\in\Delta^q}\un_{\{|\Vect{x}|=\Vect{p},\; \Vect{x}\geq\mathfrak{u}\}}\P^{\mathcal{E}}(T_{\Vect{x}}<T^1),
\end{align*}
where  $\bm{m}^{\ell-1}=(m_1,\ldots,m_{\ell-1})$, $\mathfrak{u}$ is the concatenation of $\VectCoord{\Vect{u}}{1},\ldots,\VectCoord{\Vect{u}}{|\bm{\pi}_{\ell-1}|}$ and $\Vect{x}\geq\mathfrak{u}$ means that $\VectCoord{x}{p}\geq\VectCoord{\mathfrak{u}}{p}$. Thanks to the strong Markov property at time $T_{\VectCoord{z}{i}}$, there exists a constant $C_q\geq 1$ such that 
\begin{align*}
    \underset{|\Vect{x}|=\bm{p}}{\sum_{\Vect{x}\in\Delta^q}}f^{\ell}_{\bm{m},\Pi}(\Vect{x})\P^{\mathcal{E}}(T_{\Vect{x}}<T^1)\leq C_q&\sum_{\bm{z}\in\Delta^{|\bm{\pi}_{\ell-1}|}_{m_{\ell}-1}}f^{\ell-1}_{\bm{m}^{\ell-1},\Pi^{\ell-1}}(\Vect{z})\P^{\mathcal{E}}(T_{\Vect{z}}<T^1)\prod_{j=1}^{|\bm{\pi}_{\ell-1}|}\sum_{\VectCoord{\Vect{u}}{j}\in\Delta^{b_{\ell-1}(\bm{B}_j)}_{m_{\ell}}}\prod_{i=1}^{b_{\ell-1}(\bm{B}_j)} \\ & \times\un_{\{(\VectCoord{u}{j,i})^*=\VectCoord{z}{j}\}}\underset{|\VectCoord{\Vect{x}}{j}|=\VectCoord{\bm{p}}{j}}{\sum_{\VectCoord{\Vect{x}}{j}\in\Delta^{b_{\ell-1}(\bm{B}_j)}}}\un_{\{\VectCoord{x}{j,i}\geq\VectCoord{u}{j,i}\}}\P^{\mathcal{E}}_{\VectCoord{z}{j}}(T_{\VectCoord{x}{j,i}}<T^1),
\end{align*}
where $\Vect{p}$ is now seen as the concatenation of $\VectCoord{\Vect{p}}{1},\ldots,\VectCoord{\Vect{p}}{|\bm{\pi}_{\ell-1}|}$. Moreover, it is known that for all $z\leq x$ in $\T$, 
\begin{align}\label{ProbaTA}
   \P^{\mathcal{E}}_z(T_x<T^1)=\frac{\sum_{e\leq w\leq z}e^{V(w)}}{\sum_{e\leq w\leq x}e^{V(w)}}\;\;\textrm{ if } z\not=e,\;\;\P^{\mathcal{E}}(T_x<T^1)=\frac{1}{\sum_{e\leq w\leq x}e^{V(w)}}\;\;\textrm{ otherwise},
\end{align}
so $\P^{\mathcal{E}}_z(T_x<T^1)\leq e^{-V(x)}\sum_{e\leq w\leq z}e^{V(w)}$. By independence of the increments of the branching random walk $(\T,(V(x),x\in\T))$, using that $b_{\ell-1}(\bm{B}_j)=|\bm{B}^{\ell-1}_j|$ and $\psi(1)=0$
\begin{align*}
    \Eb\Big[\underset{|\Vect{x}|=\bm{p}}{\sum_{\Vect{x}\in\Delta^q}}f^{\ell}_{\bm{m},\Pi}(\Vect{x})\P^{\mathcal{E}}(T_{\Vect{z}}<T^1)\Big]\leq\mathfrak{C}_{\ell-1}\Eb\Big[\sum_{\bm{z}\in\Delta^{|\bm{\pi}_{\ell-1}|}_{m_{\ell}-1}}f^{\ell-1}_{\bm{m}^{\ell-1},\Pi^{\ell-1}}(\Vect{z})\P^{\mathcal{E}}(T_{\Vect{z}}<T^1)\prod_{j=1}^{|\bm{\pi}_{\ell-1}|}(H_{\VectCoord{z}{j}})^{|\bm{B}^{\ell-1}_j|}\Big],
\end{align*}
with $H_z=\sum_{e\leq w\leq z}e^{V(w)-V(z)}$ and $\mathfrak{C}_{\ell-1}=C_q\prod_{\bm{B}\in\bm{\pi}_{\ell-1}}c_{|\bm{B}|}(\Vect{1})\in(0,\infty)$ thanks to Assumption \ref{Assumption2} since for any $\bm{B}\in\pi_{\ell-1}$, $|\bm{B}|<q\leq 4k<\kappa$. Again, thanks to the strong Markov property at time $T_{\VectCoord{w}{i}}$
\begin{align*}
    &\sum_{\bm{z}\in\Delta^{|\bm{\pi}_{\ell-1}|}_{m_{\ell}-1}}f^{\ell-1}_{\bm{m}^{\ell-1},\Pi^{\ell-1}}(\Vect{z})\P^{\mathcal{E}}(T_{\Vect{z}}<T^1)\prod_{j=1}^{|\bm{\pi}_{\ell-1}|}(H_{\VectCoord{z}{j}})^{|\bm{B}^{\ell-1}_j|} \\ & \leq C_{\ell-1}\sum_{\bm{w}\in\Delta^{|\bm{\pi}_{\ell-2}|}_{m_{\ell-1}-1}}f^{\ell-2}_{\bm{m}^{\ell-2},\Pi^{\ell-2}}(\Vect{w})\P^{\mathcal{E}}(T_{\Vect{w}}<T^1)\prod_{j=1}^{|\bm{\pi}_{\ell-2}|}\sum_{\VectCoord{\Vect{v}}{j}\in\Delta^{b_{\ell-2}(\bm{B}_{j})}_{m_{\ell-1}}}\prod_{i=1}^{b_{\ell-2}(\bm{B}_{j})}\un_{\{(\VectCoord{v}{j,i})^*=\VectCoord{w}{j}\}} \\ & \times\sum_{\VectCoord{\Vect{z}}{j}\in\Delta^{|\bm{\pi}_{\ell-1}|}_{s_{\ell}-1}}\un_{\{\VectCoord{z}{j,i}\geq\VectCoord{v}{j,i}\}}(H_{\VectCoord{z}{j,i}})^{\beta^{\ell-2}_{j,i}}\P^{\mathcal{E}}_{\VectCoord{w}{j}}(T_{\VectCoord{z}{j,i}}<T^1),
\end{align*}
for some constant $C_{\ell-1}\geq 1$, where $\VectCoord{\Vect{v}}{j}=(\VectCoord{u}{j,1},\ldots,\VectCoord{u}{j,b_{\ell-2}(\bm{B}_j)})$ and recall the definition of $\beta^{\ell-2}_{j,i}$ in \eqref{betaDef}. Thanks to \eqref{ProbaTA}
\begin{align*}
    (H_{\VectCoord{z}{j,i}})^{\beta^{\ell-2}_{j,i}}\P^{\mathcal{E}}_{\VectCoord{w}{j}}(T_{\VectCoord{z}{j,i}}<T^1)\leq H_{\VectCoord{w}{j}}e^{-V_{\VectCoord{w}{j}}(\VectCoord{z}{j,i})}(H_{\VectCoord{z}{j,i}})^{\beta^{\ell-2}_{j,i}-1},
\end{align*}
and $H_{\VectCoord{z}{j,i}}=H_{\VectCoord{v}{j,i}}e^{-V_{\VectCoord{v}{j,i}}(\VectCoord{z}{j,i})}+\Tilde{H}_{\VectCoord{v}{j,i},\VectCoord{z}{j,i}}$ where, for any $u<x$, $\Tilde{H}_{u,x}:=\sum_{u<w\leq x}e^{V(w)-V(x)}$. Since $H_u\geq 1$ for all $u\in\T$, we have
\begin{align*}
    H_{\VectCoord{z}{j,i}}\leq H_{\VectCoord{w}{j}}\big(e^{-V_{\VectCoord{w}{i}}(\VectCoord{v}{j,i})}+1\big)\big(e^{-V_{\VectCoord{v}{j,i}}(\VectCoord{z}{j,i})}+ \Tilde{H}_{\VectCoord{v}{j,i},\VectCoord{z}{j,i}}\big),
\end{align*} 
thus giving that $(H_{\VectCoord{z}{j,i}})^{\beta^{\ell-2}_{j,i}}\P^{\mathcal{E}}_{\VectCoord{w}{j}}(T_{\VectCoord{z}{j,i}}<T^1)$ is smaller than
\begin{align*}
    (H_{\VectCoord{w}{j}})^{\beta^{\ell-2}_{j,i}}e^{-V_{\VectCoord{w}{j}}(\VectCoord{v}{j,i})}\big(e^{-V_{\VectCoord{w}{i}}(\VectCoord{v}{j,i})}+&1\big)^{\beta^{\ell-2}_{j,i}-1}e^{-V_{\VectCoord{v}{j,i}}(\VectCoord{z}{j,i})} \\ & \times\big(e^{-V_{\VectCoord{v}{j,i}}(\VectCoord{z}{j,i})}+ \Tilde{H}_{\VectCoord{v}{j,i},\VectCoord{z}{j,i}}\big)^{\beta^{\ell-2}_{j,i}-1}.
\end{align*}
By independence of the increments of the branching random walk $(\T,(V(x),x\in\T))$, using that $\sum_{i=1}^{b_{\ell-2}(\bm{B}_j)}\beta^{\ell-2}_{j,i}=|\bm{B}^{\ell-2}_j|$
\begin{align*}
    &\Eb\Big[\sum_{\bm{z}\in\Delta^{|\bm{\pi}_{\ell-1}|}_{m_{\ell}-1}}f^{\ell-1}_{\bm{m}^{\ell-1},\Pi^{\ell-1}}(\Vect{z})\P^{\mathcal{E}}(T_{\Vect{z}}<T^1)\prod_{j=1}^{|\bm{\pi}_{\ell-1}|}(H_{\VectCoord{z}{j}})^{|\bm{B}^{\ell-1}_j|}\Big] \\ & \leq \mathfrak{C}_{\ell-2}\Eb\Big[\sum_{\bm{w}\in\Delta^{|\bm{\pi}_{\ell-2}|}_{m_{\ell-1}-1}}f^{\ell-2}_{\bm{m}^{\ell-2},\Pi^{\ell-2}}(\Vect{w})\P^{\mathcal{E}}(T_{\Vect{w}}<T^1)\prod_{j=1}^{|\bm{\pi}_{\ell-2}|}(H_{\VectCoord{w}{j}})^{|\bm{B}^{\ell-2}_j|}\Big],
\end{align*}
where, thanks to the many-to-one Lemma \ref{many-to-one} 
\begin{align*}
    \mathfrak{C}_{\ell-2}=\prod_{j=1}^{|\bm{\pi}_{\ell-2}|}\Eb\Big[\sum_{\Vect{v}\in\Delta^{b_{\ell-2}(\bm{B}_j)}_1}\prod_{i=1}^{b_{\ell-2}(\bm{B}_{j})}e^{-V(\VectCoord{v}{i})}(e^{-V(\VectCoord{v}{i})}+1)^{\beta^{\ell-2}_{j,i}}\Big]\prod_{\bm{B}\in\pi_{\ell-1}}\Eb\big[(e^{-S_{m_{\ell}^*}}+H^S_{m_{\ell}^*})^{|\bm{B}|-1}\big],
\end{align*}
$m_{\ell}^*=m_{\ell}-m_{\ell-1}-1$, $H^S_m:=\sum_{p=0}^m e^{S_p-S_m}$ (the random walk $(S_p)$ is defined in \eqref{RW}). Note that $\mathfrak{C}_{\ell-2}\in(0,\infty)$. Indeed, the first mean in the definition of $\mathfrak{C}_{\ell-2}$ belongs to $(0,\infty)$ thanks to Assumption \ref{Assumption2} since for any $1\leq j\leq |\bm{\pi}_{\ell-2}|$, $b_{\ell-2}(\bm{B}_j)<q\leq 2\mathfrak{a}k<\kappa$ and $\sum_{i=1}^{b_{\ell-2}(\bm{B}_j)}\beta^{\ell-2}_{j,i}=|\bm{B}^{\ell-2}_j|<q$. The second one also belongs to $(0,\infty)$ since for all $\bm{B}\in\pi_{\ell-1}$, $|\bm{B}|-1\leq q-2<\kappa-2$ and as it is proved in \cite{AndDiel3} that $\sup_{m\in\N^*}\Eb[(H^S_{m})^{\kappa-1-\varepsilon}]<\infty$
for any $\varepsilon>0$. We also deduce from this, together with the fact that $\psi'(1)<0$ and $m_{\ell}^*\geq 0$ that $\mathfrak{C}_{\ell-2}$ is bounded by a positive constant does not depending on $\bm{m}$. By induction, there exists a constant $\mathfrak{C}_2\in(0,\infty)$ (still not depending on $\bm{m}$) such that
\begin{align*}
    \Eb\Big[\underset{|\Vect{x}|=\bm{p}}{\sum_{\Vect{x}\in\Delta^q}}f^{\ell}_{\bm{m},\Pi}(\Vect{x})\P^{\mathcal{E}}(T_{\Vect{x}}<T^1)\Big]\leq\mathfrak{C}_2\Eb\Big[\sum_{|z|=m_1-1}\sum_{\Vect{u}\in\Delta^{|\bm{\pi}_1|}_{m_2-1}}\P^{\mathcal{E}}(T_{\Vect{u}}<T^1)\prod_{i=1}^{|\bm{\pi}_{1}|}(H_{\VectCoord{u}{i}})^{|\bm{B}^{1}_i|}\un_{\{\VectCoord{u}{i}>z\}}\Big].
\end{align*}
Thanks to the strong Markov property, $\P^{\mathcal{E}}(T_{\Vect{u}}<T^1)\leq C_{|\bm{\pi}_1|}\P^{\mathcal{E}}(T_z<T^1)\prod_{i=1}^{\bm{|\pi_1|}}\P^{\mathcal{E}}_z(T_{\VectCoord{u}{i}}<T^1)=C_{|\bm{\pi}_1|}e^{-V(z)}(H_z)^{|\bm{\pi}_1|-1}\prod_{i=1}^{\bm{|\pi_1|}}e^{-V_z(\VectCoord{u}{i})}/H_{\VectCoord{u}{i}}$ for some constant $C_{|\bm{\pi}_1|}\geq 1$ and the last equality comes from \eqref{ProbaTA}. Then, using the many-to-one Lemma \ref{many-to-one}
\begin{align*}
    \Eb\Big[\underset{|\Vect{x}|=\bm{p}}{\sum_{\Vect{x}\in\Delta^q}}f^{\ell}_{\bm{m},\Pi}(\Vect{x})\P^{\mathcal{E}}(T_{\Vect{x}}<T^1)\Big]\leq\mathfrak{C}_1\Eb\Big[\sum_{|z|=m_1-1}e^{-V(z)}(H_z)^{|\bm{\pi}_1|-1}\Big]=\mathfrak{C}_1\Eb\big[(H_{m_1-1}^S)^{|\bm{\pi}_1|-1}\big].
\end{align*}
Again, $|\bm{\pi}_1|-1\leq q-1\leq 2\mathfrak{a}k-1<\kappa-1$ so $\Eb[(H_{m_1-1}^S)^{|\bm{\pi}_1|-1}]\leq\sup_{m\in\N^*}\Eb[(H_{m-1}^S)^{|\bm{\pi}_1|-1}]\in(0,\infty)$ which ends the proof.
\end{proof}

\subsection{The range on $\mathfrak{E}^{k,\cdot}$}

This section is dedicated to the proof of Proposition \ref{GENPROPCONV2} in which the range is restricted to the $k$-tuples of vertices belonging to the set $\mathfrak{E}^{k,\cdot}$, that is such that the vertices are visited during $k$ distinct excursions, see \eqref{DistinctExcur} for the definition of $\mathfrak{E}^{k,\cdot}$.

\subsubsection{The relevant vertices: the set $\mathcal{C}^k_{a_n}$}

First recall that $\mathcal{C}^k_{m}=\{\Vect{x}\in\Delta^k;\; \mathcal{S}^k(\Vect{x})\leq m\}$ where, for any $\Vect{x}=(\VectCoord{x}{1},\ldots,\VectCoord{x}{k})\in\Delta^k$ and $\mathcal{S}^k(\Vect{x})-1$ is the last generation at which two or more vertices among $\VectCoord{x}{1},\ldots,\VectCoord{x}{k}$ share a common ancestor (see \eqref{MultiMRCA}). In this subsection, we focus on the range on $\mathfrak{E}^{k,\cdot}\cap\mathcal{C}^k_{a_n}$ with $a_n=(2\delta_0)^{-1}\log n$, which is the set of relevant $k$-tuples of vertices in the case of small generations. Before going any further, let us state and prove the following lemma. Recall that $H_u=\sum_{e\leq z\leq u}e^{V(z)-V(u)}$.
\begin{lemm}\label{HighDownfalls}
Let $k\geq 2$ and $\mathfrak{a}\geq 1$ be two integers and assume $\kappa>2\mathfrak{a}k$. Under the Assumptions \ref{Assumption1}, \ref{Assumption2} and \ref{AssumptionIncrements}
\begin{enumerate}[label=(\roman*)]
    \item\label{HighDownfalls1} for any integer $q\in\{k,\ldots,2\mathfrak{a}k\}$ and any $\bm{\beta}=(\beta_1,\ldots,\beta_q)\in\ProdSet{(\N^*)}{q}$ such that $\sum_{j=1}^q\beta_j\leq 2\mathfrak{a}k$, there exists a constant $\mathfrak{C}_{\ref{HighDownfalls},1}>0$ such that 
        \begin{align*}
            \sup_{\Vect{p}\in\ProdSet{(\N^*)}{q}}\Eb\Big[\underset{|\Vect{x}|=\Vect{p}}{\sum_{\Vect{x}\in\Delta^q}}e^{-\langle\bm{\beta},V(\Vect{x})\rangle_q}\Big]\leq\mathfrak{C}_{\ref{HighDownfalls},1};
        \end{align*}
        \item\label{HighDownfalls2} for any integer $q\in\{k,\ldots,2\mathfrak{a}k\}$ there exists a constant $\mathfrak{C}_{\ref{HighDownfalls},2}>0$ such that for $n$ large enough and any $h>0$
        \begin{align*}
            \Eb\Big[\sum_{\Vect{x}\in\Delta^q_{a_n}}\un_{\{\max_{1\leq i\leq q}H_{\VectCoord{x}{i}}>h\}}e^{-\langle\bm{1},V(\Vect{x})\rangle_q}\Big]\leq\frac{\mathfrak{C}_{\ref{HighDownfalls},2}}{h^{\kappa-1}}+o(1).
        \end{align*}
\end{enumerate}
\end{lemm}
\begin{proof}[Proof in the case $\cap_{j=1}^{\ell}\Gamma^j_{\bm{m},\Pi}\subset\{\Vect{x}\in\Delta^q;\; \mathcal{C}^q(\Vect{x})<\min_{1\leq i\leq q}p_i\}$]
Not that, since $H_u\geq 1$, we have $\Eb[\sum_{\Vect{x}\in\Delta^q_{a_n}}e^{-\langle\bm{1},V(\Vect{x})\rangle_q}]=\Eb[\sum_{\Vect{x}\in\Delta^q_{a_n}}\un_{\{\max_{1\leq i\leq q}H_{\VectCoord{x}{i}}>h\}}e^{-\langle\bm{1},V(\Vect{x})\rangle_q}]$ for all $h\leq 1$. The proof of \textit{\ref{HighDownfalls1}} is similar to the proof of Corollary \ref{GenTh3} and Lemma \ref{EspProbaTAPartition} so we focus on \textit{\ref{HighDownfalls2}}. In order to avoid unnecessary technical difficulties, we prove it for any $\mathfrak{a}\geq 2$. Recall the definition of $f^{\ell}_{\bm{s},\Pi}$ in \eqref{Def_f_partition} for $\ell\in\{1,\ldots,q-1\}$, $\bm{s}=(s_1,\ldots,s_{\ell})\in\ProdSet{\N}{\ell}$ such that $s_1<\cdots<s_{\ell}$ and $\Pi=(\bm{\pi}_i)_{0\leq i\leq\ell}$ an increasing collection of partitions of $\{1,\ldots,q\}$. Note that \begin{align*}
    \sum_{\Vect{x}\in\Delta^q_{a_n}}\sum_{j=1}^q\un_{\{\max_{1\leq i\leq q}H_{\VectCoord{x}{i}}>h\}}e^{-\langle\bm{1},V(\Vect{x})\rangle_q}=\sum_{\ell=1}^{q-1}\;\sum_{\bm{s};s_1<\ldots<s_{\ell}\leq a_n}\;\sum_{\Pi\textrm{ increasing}}\;&\sum_{\Vect{x}\in\Delta^q_{a_n}}f^{\ell}_{\bm{s},\Pi}(\Vect{x})e^{-\langle\bm{1},V(\Vect{x})\rangle_q} \\ & \times\un_{\{\max_{1\leq i\leq q}H_{\VectCoord{x}{i}}>h\}},
\end{align*}
and $\sum_{\Vect{x}\in\Delta^q_{a_n}}f^{\ell}_{\bm{s},\Pi}(\Vect{x})\un_{\{\max_{1\leq i\leq q}H_{\VectCoord{x}{j}}>h\}}e^{-\langle\bm{1},V(\Vect{x})\rangle_q}$ is equal to
\begin{align*}
   \sum_{\bm{z}\in\Delta^{|\bm{\pi}_{\ell-1}|}_{s_{\ell}-1}}f^{\ell-1}_{\bm{s}^{\ell-1},\Pi^{\ell-1}}(\Vect{z})\prod_{j=1}^{|\bm{\pi}_{\ell-1}|}\sum_{\VectCoord{\Vect{u}}{j}\in\Delta^{b_{\ell-1}(\bm{B}_j)}_{s_{\ell}}}&\prod_{i=1}^{b_{\ell-1}(\bm{B}_j)}\un_{\{(\VectCoord{u}{j,i})^*=\VectCoord{z}{j}\}}\sum_{\VectCoord{\Vect{x}}{j}\in\Delta^{b_{\ell-1}(\bm{B}_j)}_{a_n}}\un_{\{\VectCoord{x}{j,i}\geq\VectCoord{u}{j,i}\}} \\ &\times e^{-V(\VectCoord{x}{j,i})}\un_{\{\max_{1\leq j'\leq|\bm{\pi}_{\ell-1}|}\max_{1\leq i'\leq b_{\ell-1}(\bm{B}_{j'})}H_{\VectCoord{x}{j',i'}}>h\}}.
\end{align*}
For any $u\leq x$, introduce $H_{u,x}:=\sum_{u\leq z\leq x}e^{V(z)-V(x)}$. Thanks to Assumption \ref{AssumptionIncrements} together with the fact that $H_{\VectCoord{z}{j'}}\geq 1$
\begin{align*}
    H_{\VectCoord{x}{j',i'}}\leq H_{\VectCoord{z}{j'}}e^{\mathfrak{h}}e^{-V_{\VectCoord{u}{j',i'}}(\VectCoord{x}{j',i'})}+H_{\VectCoord{u}{j',i'},\VectCoord{x}{j',i'}},
\end{align*}
so $H_{\VectCoord{x}{j',i'}}>h$ implies that $H_{\VectCoord{z}{j'}}e^{\mathfrak{h}}e^{-V_{\VectCoord{u}{j',i'}}(\VectCoord{x}{j',i'})}>h/2$ or $H_{\VectCoord{u}{j',i'},\VectCoord{x}{j',i'}}>h/2$. We also decompose according to the values of $H_{\VectCoord{z}{j'}}$:
\begin{align*}
    &\un_{\big\{\max_{1\leq j'\leq|\bm{\pi}_{\ell-1}|}\max_{1\leq i'\leq b_{\ell-1}(\bm{B}_{j'})}H_{\VectCoord{z}{j'}}e^{\mathfrak{h}}e^{-V_{\VectCoord{u}{j',i'}}(\VectCoord{x}{j',i'})}>h/2\big\}} \\ & \leq \un_{\{\max_{1\leq j'\leq|\bm{\pi}_{\ell-1}|}H_{\VectCoord{z}{j'}}>h\}}+\un_{\big\{\max_{1\leq j'\leq|\bm{\pi}_{\ell-1}|}\max_{1\leq i'\leq b_{\ell-1}(\bm{B}_{j'})}2e^{\mathfrak{h}}e^{-V_{\VectCoord{u}{j',i'}}(\VectCoord{x}{j',i'})}>1\big\}}.
\end{align*}
We therefore deduce that $\un_{\{\max_{1\leq j'\leq|\bm{\pi}_{\ell-1}|}\max_{1\leq i'\leq b_{\ell-1}(\bm{B}_{j'})}H_{\VectCoord{x}{j',i'}}>h\}}$ is smaller than
\begin{align*}
     \un_{\{\max_{1\leq j'\leq|\bm{\pi}_{\ell-1}|}H_{\VectCoord{z}{j'}}>h\}}+\sum_{j'=1}^{|\bm{\pi}_{\ell-1}|}\sum_{i'=1}^{b_{\ell-1}(\bm{B}_{j'})}\Big(\un_{\{H_{\VectCoord{u}{j',i'},\VectCoord{x}{j',i'}}>h/2\}}+ \un_{\big\{2e^{\mathfrak{h}}e^{-V_{\VectCoord{u}{j',i'}}(\VectCoord{x}{j',i'})}>1\big\}}\Big).
\end{align*}
By independence of the increments of the branching random walk $(\T,(V(x),x\in\T))$, since $\psi(1)=0$
\begin{align*}
    &\Eb\Big[\sum_{\Vect{x}\in\Delta^q_{a_n}}f^{\ell}_{\bm{s},\Pi}(\Vect{x})\un_{\{\max_{1\leq i\leq q}H_{\VectCoord{x}{i}}>h\}}e^{-\langle\bm{1},V(\Vect{x})\rangle_q}\Big] \\ & \leq\Eb\Big[\sum_{\bm{z}\in\Delta^{|\bm{\pi}_{\ell-1}|}_{s_{\ell}-1}}f^{\ell-1}_{\bm{s}^{\ell-1},\Pi^{\ell-1}}(\Vect{z})\un_{\{\max_{1\leq j'\leq|\bm{\pi}_{\ell-1}|}H_{\VectCoord{z}{j'}}>h\}}e^{-\langle\bm{\beta}^{\ell-1},V(\bm{z})\rangle_{|\bm{\pi}_{\ell-1}|}}\Big]\prod_{j=1}^{|\bm{\pi}_{\ell-1}|}c_{b_{\ell-1}(\bm{B}_j)}(\bm{1}) \\ & +q(\mathfrak{z}_{1,n}+\mathfrak{z}_{2,n})\prod_{j=1}^{|\bm{\pi}_{\ell-1}|}c_{b_{\ell-1}(\bm{B}_j)}(\bm{1})\Eb\Big[\mathcal{A}^{|\bm{\pi}_{\ell-1}|}_{s_{\ell}-1}\big(f^{\ell-1}_{\bm{s}^{\ell-1},\Pi^{\ell-1}},\bm{\beta}^{\ell-1}\big)\Big],
\end{align*}
where we recall that $\mathcal{A}^q_l(g,\beta)=\sum_{\Vect{x}\in\Delta^q_l}e^{-\langle \beta,V(\Vect{x})\rangle_q}$,
\begin{align*}
    \mathfrak{z}_{1,n}=\Eb[\sum_{|x|=a_n-s_{\ell}}e^{-V(x)}\un_{\{H_x>h/2\}}]\;\;\textrm{ and }\;\;\mathfrak{z}_{2,n}=\Eb[\sum_{|x|=a_n-s_{\ell}}e^{-V(x)}\un_{\{2e^{\mathfrak{h}}e^{-V(x)}>1\}}].
\end{align*}
Thanks to the many-to-one Lemma \ref{many-to-one}
\begin{align*}
    \mathfrak{z}_{1,n}=\Pb(H^S_{a_n-s_{\ell}}>h/2)\leq\mathfrak{C}_{\ref{HighDownfalls},3}/h^{\kappa-1},
\end{align*}
for some constant $\mathfrak{C}_{\ref{HighDownfalls},3}>0$, the last inequality coming from (\cite{AndDiel3}, Lemma 2.2). We now turn to $\mathfrak{z}_{n,2}$. If $s_{\ell}\leq a_n/2$ then, for any $\rho\in(0,\kappa-1)$
\begin{align*}
    \mathfrak{z}_{2,n}\leq 2^{\rho}(1+e^{\mathfrak{h}})^{\rho}e^{a_n\psi(1+\rho)/2}.
\end{align*}
Otherwise $s_{\ell}-1\geq a_n/2$ and thanks to the Cauchy–Schwarz inequality
\begin{align*}
    \Eb\Big[\mathcal{A}^{|\bm{\pi}_{\ell-1}|}_{s_{\ell}-1}\big(f^{\ell-1}_{\bm{s}^{\ell-1},\Pi^{\ell-1}},\bm{\beta}^{\ell-1}\big)\Big]\leq&\Eb\Big[\mathcal{A}^{|\bm{\pi}_{\ell-1}|}_{s_{\ell}-1}\big(f^{\ell-1}_{\bm{s}^{\ell-1},\Pi^{\ell-1}},\bm{\beta}^{\ell-1}\big)\un_{\tilde{\mathcal{V}}_n}\Big]+ \\ & +\big(1-\Pb(\tilde{\mathcal{V}}_n)\big)^{1/2}\Eb\Big[\Big(\mathcal{A}^{|\bm{\pi}_{\ell-1}|}_{s_{\ell}-1}\big(f^{\ell-1}_{\bm{s}^{\ell-1},\Pi^{\ell-1}},\bm{\beta}^{\ell-1}\big)\Big)^2\Big]^{1/2},
\end{align*}
where $\tilde{\mathcal{V}}_n:=\{\min_{a_n/2\leq|z|\leq a_n}V(z)>3/2\log n\}$ (recall that $a_n=(2\delta_0)^{-1}\log n$). On the one hand, by definition, there exists $i_{\alpha}\in\{1,\ldots,|\bm{\pi}_{\ell-1}|\}$ such that $|\bm{B}^{\ell-1}_{i_{\alpha}}|\geq 2$. It follows that
\begin{align*}
    \Eb\Big[\mathcal{A}^{|\bm{\pi}_{\ell-1}|}_{s_{\ell}-1}\big(f^{\ell-1}_{\bm{s}^{\ell-1},\Pi^{\ell-1}},\bm{\beta}^{\ell-1}\big)\un_{\tilde{\mathcal{V}}_n}\Big]\leq n^{-3/2}\Eb\Big[\mathcal{A}^{|\bm{\pi}_{\ell-1}|}_{s_{\ell}-1}\big(f^{\ell-1}_{\bm{s}^{\ell-1},\Pi^{\ell-1}},\tilde{\bm{\beta}}^{\ell-1}\big)\Big],
\end{align*}
where $\tilde{\bm{\beta}}^{\ell-1}_j=|\bm{B}^{\ell-1}_j|$ for all $j\not=i_{\alpha}$ and $\tilde{\bm{\beta}}^{\ell-1}_{i_{\alpha}}=|\bm{B}^{\ell-1}_{i_{\alpha}}|-1\geq 1$. One the other hand, $1-\Pb(\tilde{\mathcal{V}}_n)\leq n^{-\rho_5}$ with $\rho_5>0$ thanks to Lemma \ref{MinPotential}. Moreover, both $\sum_{j=1}^{|\bm{\pi}_{\ell-1}|}\bm{\beta}^{\ell-1}_j$ and $\sum_{j=1}^{|\bm{\pi}_{\ell-1}|}\tilde{\bm{\beta}}^{\ell-1}_j$ are smaller than $2k$ since $q\leq 2k$. Hence, thanks to \textit{\ref{HighDownfalls1}}
\begin{align*}
    \Eb\Big[\mathcal{A}^{|\bm{\pi}_{\ell-1}|}_{s_{\ell}-1}\big(f^{\ell-1}_{\bm{s}^{\ell-1},\Pi^{\ell-1}},\tilde{\bm{\beta}}^{\ell-1}\big)\Big]+\Eb\Big[\Big(\mathcal{A}^{|\bm{\pi}_{\ell-1}|}_{s_{\ell}-1}\big(f^{\ell-1}_{\bm{s}^{\ell-1},\Pi^{\ell-1}},\bm{\beta}^{\ell-1}\big)\Big)^2\Big]^{1/2}\leq\mathfrak{C}_{\ref{HighDownfalls},4},
\end{align*}
for some constant $\mathfrak{C}_{\ref{HighDownfalls},4}>0$. We obtain
\begin{align*}
    &\Eb\Big[\sum_{\Vect{x}\in\Delta^q_{a_n}}f^{\ell}_{\bm{s},\Pi}(\Vect{x})\un_{\{\max_{1\leq i\leq q}H_{\VectCoord{x}{i}}>h\}}e^{-\langle\bm{1},V(\Vect{x})\rangle_q}\Big] \\ & \leq\Eb\Big[\sum_{\bm{z}\in\Delta^{|\bm{\pi}_{\ell-1}|}_{s_{\ell}-1}}f^{\ell-1}_{\bm{s}^{\ell-1},\Pi^{\ell-1}}(\Vect{z})\un_{\{\max_{1\leq j'\leq|\bm{\pi}_{\ell-1}|}H_{\VectCoord{z}{j'}}>h\}}e^{-\langle\bm{\beta}^{\ell-1},V(\bm{z})\rangle_{|\bm{\pi}_{\ell-1}|}}\Big]\prod_{j=1}^{|\bm{\pi}_{\ell-1}|}c_{b_{\ell-1}(\bm{B}_j)}(\bm{1}) \\ & +q\mathfrak{z}_{1,n}\prod_{j=1}^{|\bm{\pi}_{\ell-1}|}c_{b_{\ell-1}(\bm{B}_j)}(\bm{1})\Eb\Big[\mathcal{A}^{|\bm{\pi}_{\ell-1}|}_{s_{\ell}-1}\big(f^{\ell-1}_{\bm{s}^{\ell-1},\Pi^{\ell-1}},\bm{\beta}^{\ell-1}\big)\Big]+n^{-\rho_6},
\end{align*}
thanks to the Assumption \ref{Assumption2} and for $\rho_6>0$. Note (see the proof of Corollary \ref{GenTh3}) that 
\begin{align*}
    \Eb\Big[\mathcal{A}^{|\bm{\pi}_{\ell-1}|}_{s_{\ell}-1}\big(f^{\ell-1}_{\bm{s}^{\ell-1},\Pi^{\ell-1}},\bm{\beta}^{\ell-1}\big)\Big]=\mathcal{A}^{|\bm{\pi}_{\ell-2}|}_{s_{\ell-1}-1}\big(f^{\ell-2}_{\bm{s}^{\ell-2},\Pi^{\ell-2}},\bm{\beta}^{\ell-2}\big)\prod_{j=1}^{|\bm{\pi}_{\ell-2}|}c_{b_{\ell-2}(\bm{B}_j)}(\bm{\beta}^{\ell-2}_j)\underset{|\mathfrak{B}|\geq 2}{\prod_{\mathfrak{B}\in\bm{\pi}_{\ell-1}}}e^{s_{\ell}^*\psi(|\mathfrak{B}|)},
\end{align*}
with $\bm{\beta}^{\ell-2}_j=(\beta^{\ell-2}_{j,1},\ldots,\beta^{\ell-2}_{j,b_{\ell-2}(\bm{B}_j)})$ and $s_{\ell}^*=s_{\ell}-s_{\ell-1}-1$. Since for any $\mathfrak{B}\in\bm{\pi}_{\ell-1}$ such that $|\mathfrak{B}|\geq 2$, $\psi(|\mathfrak{B}|)<0$, we have
\begin{align*}
    \sum_{s_{\ell}=s_{\ell-1}+1}^{a_n}\Eb\Big[\mathcal{A}^{|\bm{\pi}_{\ell-1}|}_{s_{\ell}-1}\big(f^{\ell-1}_{\bm{s}^{\ell-1},\Pi^{\ell-1}},\bm{\beta}^{\ell-1}\big)\Big]\leq\Eb\Big[\mathcal{A}^{|\bm{\pi}_{\ell-2}|}_{s_{\ell-1}-1}\big(f^{\ell-2}_{\bm{s}^{\ell-2},\Pi^{\ell-2}},\bm{\beta}^{\ell-2}\big)\Big]&\prod_{\bm{B}\in\bm{\pi}_{\ell-2}}c_{b_{\ell-2}(\bm{B})}(\bm{\beta}^{\ell-1}) \\ & \Big(1-\underset{|\mathfrak{B}|\geq 2}{\prod_{\mathfrak{B}\in\bm{\pi}_{\ell-1}}}e^{\psi(|\mathfrak{B}|)}\Big)^{-1}.
\end{align*}
Doing the same for $\Eb[\sum_{\bm{z}\in\Delta^{|\bm{\pi}_{\ell-1}|}_{s_{\ell}-1}}f^{\ell-1}_{\bm{s}^{\ell-1},\Pi^{\ell-1}}(\Vect{z})\un_{\{\max_{1\leq j'\leq|\bm{\pi}_{\ell-1}|}H_{\VectCoord{z}{j'}}>h\}}e^{-\langle\bm{\beta}^{\ell-1},V(\bm{z})\rangle_{|\bm{\pi}_{\ell-1}|}}]$, we obtain, thanks to Assumption \ref{Assumption2} 
\begin{align*}
    &\sum_{\bm{s};s_1<\ldots<s_{\ell}\leq a_n}\Eb\Big[\sum_{\Vect{x}\in\Delta^q_{a_n}}f^{\ell}_{\bm{s},\Pi}(\Vect{x})\un_{\{\max_{1\leq i\leq q}H_{\VectCoord{x}{i}}>h\}}e^{-\langle\bm{1},V(\Vect{x})\rangle_q}\Big] \\ & \leq\sum_{s_1<\ldots<s_{\ell-1}\leq a_n}\Eb\Big[\sum_{\bm{z}\in\Delta^{|\bm{\pi}_{\ell-2}|}_{s_{\ell-1}-1}}f^{\ell-2}_{\bm{s}^{\ell-2},\Pi^{\ell-2}}(\Vect{z})\un_{\{\max_{1\leq j'\leq|\bm{\pi}_{\ell-2}|}H_{\VectCoord{z}{j'}}>h\}}e^{-\langle\bm{\beta}^{\ell-2},V(\bm{z})\rangle_{|\bm{\pi}_{\ell-2}|}}\Big] \\ & + \frac{\mathfrak{C}_{\ref{HighDownfalls},5}}{h^{\kappa-1}} + n^{-\rho_7},
\end{align*}
for some constant $\mathfrak{C}_{\ref{HighDownfalls},5}>0$ and $\rho_7>0$. We conclude by induction together with Assumption \ref{Assumption2}.
\end{proof}

\noindent We remind the definition of the range $\mathcal{A}^k(\mathcal{D}_{n,T^s},g)$
\begin{align*}
    \mathcal{A}^k(\mathcal{D}_{n,T^s},g)=\underset{\ell_n\leq|\Vect{x}|\leq\mathfrak{L}_n}{\sum_{\Vect{x}\in\Delta^k}}g(\Vect{x})\un_{\{T_{\Vect{x}}<T^s\}},
\end{align*}
where $T_{\Vect{x}}=\max_{1\leq i\leq k}T_{\VectCoord{x}{i}}$ and $\ell_n\leq|\Vect{x}|\leq\mathfrak{L}_n$ means that $\ell_n\leq|\VectCoord{x}{i}|\leq\mathfrak{L}_n$ for all $i\in\{1,\ldots,k\}$. \\
\noindent Vertices with high potential have a major contribution to the range. One can note that under the Assumption \ref{Assumption1}, the potential $V(u)$ of the vertex $u\in\T$ behaves like $|u|$ when $|u|$ is large (see \cite{Biggins} and \cite{HuShi24} for instance). It allows to say that 

\vspace{0.1cm}

\noindent\textbf{Fact 1.} For all $\varepsilon\in(0,1)$, there exists $a_{\varepsilon}>0$ such that 
\begin{align}\label{Fact1}
    \Pb^*\big(\inf_{z\in\T}V(z)\geq -a_{\varepsilon}\big)\geq 1-\varepsilon.
\end{align}
Moreover,
\begin{lemm}\label{MinPotential}
Under the Assumption \ref{Assumption1}, there exists $\delta_0>0$ and $\rho_1>1/2$ such that for any positive integer $\zeta$
\begin{align*}
    \Pb\big(\min_{|z|=\delta_0^{-1}\zeta}V(z)\geq 3\zeta\big)\geq 1-e^{-\rho_1\zeta},
\end{align*}
\end{lemm}
\noindent Using Lemma \ref{MinPotential}, we are able to prove that any vertex $x\in\T$ in a generation between $\delta_0^{-1}\log n$ and $\sqrtBis{n}$ is visited during a single excursion above the parent $e^*$ of the root $e$. For that, let us define the edge local time $N_u^T:=\sum_{j=1}^T\un_{\{X_{j-1}=u^*,X_j=u\}}$ of the vertex $u\in\T$ and introduce 
\begin{align*}
    E^{s}_u:=\sum_{j=1}^{s}\un_{\{N_u^{T^j}-N_u^{T^{j-1}}\geq 1\}},
\end{align*}
the number of excursions during which the vertex $x$ is visited by the random walk $\X$.
\begin{lemm}\label{LemmeExc1}
Under the Assumption \ref{Assumption1}, for all $\varepsilon_1\in(0,1)$, there exists $\rho_2:=\rho_2(\varepsilon_1)>0$ such that for $n$ large enough
\begin{align*}
    \P^*\Big(\bigcup_{s=\varepsilon_1\sqrtBis{n}}^{\sqrtBis{n}/\varepsilon_1}\;\bigcup_{|z|=\delta_0^{-1}\log n}^{\sqrtBis{n}}\big\{E_z^{s}\geq 2\big\}\Big)\leq n^{-\rho_2}.
\end{align*}
\end{lemm}
\noindent The proof of Lemma \ref{LemmeExc1} is similar to the one of Lemma 3.5 in \cite{AndDiel3}. \\
Introduce the set $\mathfrak{S}^{k,s}$ of $k$-tuples of vertices visited during a single excursion:
\begin{align}\label{SingleExcur}
    \mathfrak{S}^{k,s}:=\{\Vect{x}=(\VectCoord{x}{1},\ldots,\VectCoord{x}{k})\in\Delta^k;\; \forall\; 1\leq i\leq k,\; E^{s}_{\VectCoord{x}{i}}=1\}.
\end{align}
In other words, Lemma \ref{LemmeExc1} says that we can restrict the study of the range $\mathcal{A}^k(\mathcal{D}_{n,T^s},f\un_{\mathfrak{E}^{k,s}\cap\mathcal{C}_{a_n}^k})$ to the set $\mathfrak{S}^{k,s}$. This restriction allows to get quasi-independence in the trajectory of the random walk $\X$ and the resulting quasi-independent version of the range $\mathcal{A}^k(\mathcal{D}_{n,T^s},f\un_{\mathfrak{E}^{k,s}\cap\mathcal{C}_{a_n}^k})$ is easier to deal with. A similar idea is developed in \cite{AndDiel3} and \cite{AndAKHightPotential}. Let $\Vect{j}\in\llbracket 1,s\rrbracket_k$, $\Vect{p}\in\ProdSet{\{\ell_n,\ldots,\mathfrak{L}_n\}}{k}$ and define  
\begin{align}\label{DefIndepRange}
    \mathcal{A}^{k,n}_{\Vect{p}}(\Vect{j},g):=\underset{|\Vect{x}|=\Vect{p}}{\sum_{\Vect{x}\in\Delta^k}}g(\Vect{x})\prod_{i=1}^k\un_{\{N_{\VectCoord{x}{i}}^{T^{j_i}}-N_{\VectCoord{x}{i}}^{T^{j_i-1}}\geq 1\}}\;\;\textrm{ and }\;\; \mathcal{A}^{k,n}(\Vect{j},g):=\sum_{\Vect{p}\in\ProdSet{\{\ell_n,\ldots,\mathfrak{L}_n\}}{k}}\mathcal{A}^{k,n}_{\Vect{p}}(\Vect{j},g),
\end{align}
where for any $\Vect{x}=(\VectCoord{x}{1},\ldots,\VectCoord{x}{k})$, $|\Vect{x}|=\Vect{p}$ means nothing but $|\VectCoord{x}{i}|=p_i$ for all $i\in\{1,\ldots,k\}$.
In the next lemma, we show that $\mathcal{A}^k(\mathcal{D}_{n,T^s},f\un_{\mathfrak{E}^{k,s}\cap\mathcal{C}_{\cdot}^k})$ and $\sum_{\bm{j}\in\llbracket 1,s\rrbracket_k}\mathcal{A}^{k,n}(\Vect{j},f\un_{\mathcal{C}_{\cdot}^k})$ have the same behavior

\begin{lemm}\label{LemmeSommeExc}
Let $k\geq 2$ be an integer and assume $\kappa>2k$. Under the Assumptions \ref{Assumption1} and \ref{Assumption2}, for all bounded and non-negative function $g$, any $\varepsilon,\varepsilon_1\in(0,1)$, there exists $\rho_4:=\rho_4(\varepsilon,\varepsilon_1)>0$ such that for $n$ large enough
\begin{align*}
    \P^*\Big(\bigcup_{s=\varepsilon_1\sqrtBis{n}}^{\sqrtBis{n}/\varepsilon_1}\Big\{\Big|\mathcal{A}^k(\mathcal{D}_{n,T^s},g\un_{\mathfrak{E}^{k,s}\cap\mathcal{C}^k_{a_n}})-\sum_{\bm{j}\in\llbracket 1,s\rrbracket_k}\mathcal{A}^{k,n}(\Vect{j},g\un_{\mathcal{C}^k_{a_n}})\Big|>\varepsilon (s\bm{L}_n)^k\Big\}\Big)\leq n^{-\rho_4}.
    \end{align*}
\end{lemm}
\begin{proof}
We first decompose as follows
\begin{align*}
    \mathcal{A}^k(\mathcal{D}_{n,T^s},g\un_{\mathfrak{E}^{k,s}\cap\mathcal{C}^k_{a_n}})=\mathcal{A}^k(\mathcal{D}_{n,T^s},\un_{\mathfrak{E}^{k,s}\cap\mathcal{C}^k_{a_n}\cap\mathfrak{S}^{k,s}}) + \mathcal{A}^k(\mathcal{D}_{n,T^s},g\un_{\mathfrak{E}^{k,s}\cap\mathcal{C}^k_{a_n}\cap\Delta^k\setminus\mathfrak{S}^{k,s}}).
\end{align*}
By Lemma \ref{LemmeExc1}, we have that for $n$ large enough
\begin{align*}
    \P^*\Big(\bigcup_{s=\varepsilon_1\sqrtBis{n}}^{\sqrtBis{n}/\varepsilon_1}\mathcal{A}^k(\mathcal{D}_{n,T^s},g\un_{\mathfrak{E}^{k,s}\cap\mathcal{C}^k_{a_n}\cap\Delta^k\setminus\mathfrak{S}^{k,s}})>\varepsilon(s\bm{L}_n)^k/2\Big)\leq n^{-\rho_2},
\end{align*}
so we can focus on $\mathcal{A}^k(\mathcal{D}_{n,T^s},g\un_{\mathfrak{E}^{k,s}\cap\mathfrak{S}^{k,s}\cap\mathcal{C}^k_{a_n}})$. Note that $\Vect{x}\in\mathfrak{E}^{k,s}\cap\mathfrak{S}^{k,s}$ means nothing but there exists $\Vect{j}\in\llbracket 1,s\rrbracket_k$ such that for any $i\in\{1,\ldots,k\}$, $N^{T^{j_i}}_{\VectCoord{x}{i}}-N^{T^{j_i-1}}_{\VectCoord{x}{i}}\geq 1$ and for all $\mathfrak{j}\not=j_i$, $N^{T^{\mathfrak{j}}}_{\VectCoord{x}{i}}-N^{T^{\mathfrak{j}-1}}_{\VectCoord{x}{i}}=0$, thus giving that $\mathcal{A}^k(\mathcal{D}_{n,T^s},g\un_{\mathfrak{E}^{k,s}\cap\mathfrak{S}^{k,s}\cap\mathcal{C}^k_{a_n}})$ is equal to
\begin{align*}
    \sum_{\Vect{j}\in\llbracket 1,s\rrbracket_k}\;\sum_{\Vect{p}\in\ProdSet{\{\ell_n,\ldots,\mathfrak{L}_n\}}{k}}\underset{|\Vect{x}|=\Vect{p}}{\sum_{\Vect{x}\in\Delta^k}}g\un_{\mathcal{C}^k_{a_n}}(\Vect{x})\prod_{i=1}^k\un_{\{N_{\VectCoord{x}{i}}^{T^{j_i}}-N_{\VectCoord{x}{i}}^{T^{j_i-1}}\geq 1;\; \forall \mathfrak{j}\not=j_i, N_{\VectCoord{x}{i}}^{T^{\mathfrak{j}}}-N_{\VectCoord{x}{i}}^{T^{\mathfrak{j}-1}}=0\}}.
\end{align*}
Hence, for any $s\in\{\varepsilon_1\sqrtBis{n},\ldots,\sqrtBis{n}/\varepsilon_1\}$
\begin{align*}
    \sum_{\bm{j}\in\llbracket 1,s\rrbracket_k}\mathcal{A}^{k,n}(\Vect{j},g\un_{\mathcal{C}^k_{a_n}})-\mathcal{A}^k(\mathcal{D}_{n,T^s},g\un_{\mathfrak{E}^{k,s}\cap\mathfrak{S}^{k,s}\cap\mathcal{C}^k_{a_n}})\geq 0,
\end{align*}
and thanks to Markov inequality
\begin{align}\label{UnifBound2}
    &\P^{\mathcal{E}}\Big(\bigcup_{s=\varepsilon_1\sqrtBis{n}}^{\sqrtBis{n}/\varepsilon_1}\Big\{\Big|\mathcal{A}^k(\mathcal{D}_{n,T^s},g\un_{\mathfrak{E}^{k,s}\cap\mathfrak{S}^{k,s}\cap\mathcal{C}^k_{a_n}})-\sum_{\bm{j}\in\llbracket 1,s\rrbracket_k}\mathcal{A}^{k,n}(\Vect{j},g\un_{\mathcal{C}^k_{a_n}})\Big|>\varepsilon (s\bm{L}_n)^k/2\Big\}\Big)\nonumber \\ & \leq\frac{2}{\varepsilon(\varepsilon_1\sqrtBis{n}\bm{L}_n)^k}\E^{\mathcal{E}}\Big[\sum_{\bm{j}\in\llbracket 1,s\rrbracket_k}\mathcal{A}^{k,n}(\Vect{j},g\un_{\mathcal{C}^k_{a_n}})-\mathcal{A}^k(\mathcal{D}_{n,T^{s}},g\un_{\mathfrak{E}^{k,s}\cap\mathfrak{S}^{k,s}\cap\mathcal{C}^k_{a_n}})\Big].
\end{align}
One can see that we can restrict ourselves to the $k$-tuples $\Vect{x}=(\VectCoord{x}{1},\ldots,\VectCoord{x}{k})\in\mathcal{C}^k_{a_n}$, $\Vect{x}=\Vect{p}\in\{\ell_n,\ldots,\mathfrak{L}_n\}$, such that for all $i,\mathfrak{j}\in\{1,\ldots,k\}$ with $i\not=\mathfrak{j}$, $H_{\VectCoord{x}{i}\land\VectCoord{x}{\mathfrak{j}}}\leq e^{\w_0a_n/2}$ and $V_{\VectCoord{x}{i}\land\VectCoord{x}{\mathfrak{j}}}(\VectCoord{x}{\mathfrak{j}})\geq\w_0a_n$ for some $\w_0>0$. Indeed, if this subset of $\Delta^k$ is denoted by $\mathcal{H}^k_n$, then, using similar arguments as the ones we have used several times, it can be proved that for a given $\w_0>0$ and $n$ large enough
\begin{align}\label{EspBound3}
    \E^*\Big[\frac{1}{(\sqrtBis{n}\bm{L}_n)^k}\sup_{s\leq\sqrtBis{n}/\varepsilon_1}\;\sum_{\bm{j}\in\llbracket 1,s\rrbracket_k}\mathcal{A}^{k,n}(\Vect{j},g\un_{\mathcal{C}^k_{a_n}}\un_{\Delta^k\setminus\mathcal{H}^k_n})\Big]\leq n^{\rho_2'},
\end{align}
for some $\rho_2'>0$. \\
We now aim to provide a lower bound for $\E^{\mathcal{E}}[\mathcal{A}^k(\mathcal{D}_{n,T^{s}},g\un_{\mathfrak{E}^{k,s}\cap\mathfrak{S}^{k,s}\cap\mathcal{C}^k_{a_n}\cap\mathcal{H}^k_n})]$. Thanks to the strong Markov property, the random variables $N_z^{T^{l}}-N_z^{T^{l-1}}$, $l\in\N^*$, are i.i.d under $\P^{\mathcal{E}}$ and distributed as $N_z^{T^1}$. It follows that 
\begin{align*}
    &\E^{\mathcal{E}}\Big[\mathcal{A}^k(\mathcal{D}_{n,T^{s}},g\un_{\mathfrak{E}^{k,s}\cap\mathfrak{S}^{k,s}\cap\mathcal{C}^k_{a_n}})\Big] \\ & = \underset{|\Vect{x}|=\bm{p}}{\sum_{\Vect{x}\in\mathcal{H}^k_n}}g\un_{\mathcal{C}^k_{a_n}}(\Vect{x})\prod_{i=1}^k\P^{\mathcal{E}}(\forall\;\mathfrak{j}\not=i,\;T_{\VectCoord{x}{\mathfrak{j}}}>T^1>T_{\VectCoord{x}{i}})\P^{\mathcal{E}}(T_{\VectCoord{x}{i}}>T^1)^{s-k} \\ & \geq \underset{|\Vect{x}|=\bm{p}}{\sum_{\Vect{x}\in\mathcal{H}^k_n}}g\un_{\mathcal{C}^k_{a_n}}(\Vect{x})\prod_{i=1}^k\Big(\P^{\mathcal{E}}\big(T_{\VectCoord{x}{i}}<T^1\big)-\sum_{\mathfrak{j}=1;\; \mathfrak{j}\not=i}^k\P^{\mathcal{E}}\big(T_{\VectCoord{x}{i}}<T^1,\;T_{\VectCoord{x}{\mathfrak{j}}}<T^1 \big)\Big)\P^{\mathcal{E}}(T_{\VectCoord{x}{i}}>T^1)^{s-k}.
\end{align*}
One can see that for any $\Vect{x}\in\mathcal{H}^k_n$, $\sum_{\mathfrak{j}=1;\; \mathfrak{j}\not=i}^k\P^{\mathcal{E}}(T_{\VectCoord{x}{i}}<T^1,\;T_{\VectCoord{x}{\mathfrak{j}}}<T^1)$ is very small with respect to $\P^{\mathcal{E}}(T_{\VectCoord{x}{i}}<T^1)$. Indeed, by the strong Markov property and \eqref{ProbaTA}, we have, for any $\mathfrak{j}\not=i$
\begin{align*}
    \P^{\mathcal{E}}(T_{\VectCoord{x}{i}}<T^1,\;T_{\VectCoord{x}{\mathfrak{j}}}<T^1)&\leq \P^{\mathcal{E}}(T_{\VectCoord{x}{i}}<T^1)\P^{\mathcal{E}}_{\VectCoord{x}{i}\land\VectCoord{x}{\mathfrak{j}}}(T_{\VectCoord{x}{\mathfrak{j}}}<T^1) \\ & +\P^{\mathcal{E}}(T_{\VectCoord{x}{\mathfrak{j}}}<T^1)\P^{\mathcal{E}}_{\VectCoord{x}{i}\land\VectCoord{x}{\mathfrak{j}}}(T_{\VectCoord{x}{i}}<T^1) \\ & \leq 2H_{\VectCoord{x}{i}\land\VectCoord{x}{\mathfrak{j}}}e^{-V_{\VectCoord{x}{i}\land\VectCoord{x}{\mathfrak{j}}}(\VectCoord{x}{\mathfrak{j}})}\P^{\mathcal{E}}(T_{\VectCoord{x}{i}}<T^1) \\ & \leq 2n^{-\w_0(4\delta_0)^{-1}}\P^{\mathcal{E}}(T_{\VectCoord{x}{i}}<T^1),
\end{align*}
recalling that $a_n=(2\delta_0)^{-1}\log n$. Using \eqref{ProbaTA} again, we have, on $\mathcal{V}_n=\{\min_{\delta_0^{-1}\log n\leq|x|\leq\sqrtBis{n}}V(z)\geq 3\log n\}$, that 
\begin{align*}
    \P^{\mathcal{E}}(T_{\VectCoord{x}{i}}>T^1)^{s-k}\geq(1-\P^{\mathcal{E}}(T_{\VectCoord{x}{i}}<T^1))^{s}\geq(1-e^{-V(\VectCoord{x}{i})})^{s}\geq (1-n^{-3})^s\geq (1-n^{-3})^{\sqrtBis{n}/\varepsilon_1}.
\end{align*}
Hence, $\E^{\mathcal{E}}[\mathcal{A}^k(\mathcal{D}_{n,T^{s}},g\un_{\mathfrak{E}^{k,s}\cap\mathfrak{S}^{k,s}\cap\mathcal{C}^k_{a_n}})]$ is larger than
\begin{align*}
    (1-n^{-3})^{\sqrtBis{n}/\varepsilon_1}\big(1-2kn^{-\w_0(4\delta_0)^{-1}}\big)^k\underset{|\Vect{x}|=\bm{p}}{\sum_{\Vect{x}\in\mathcal{H}^k_n}}g\un_{\mathcal{C}^k_{a_n}}(\Vect{x})\prod_{i=1}^k\P^{\mathcal{E}}\big(T_{\VectCoord{x}{i}}<T^1\big).
\end{align*}
It follows that 
\begin{align*}
    &\E\Big[\un_{\mathcal{V}_n}\sum_{\bm{j}\in\llbracket 1,s\rrbracket_k}\mathcal{A}^{k,n}(\Vect{j},g\un_{\mathcal{C}^k_{a_n}\cap\mathcal{H}^k_n})-\mathcal{A}^k(\mathcal{D}_{n,T^{s}},g\un_{\mathfrak{E}^{k,s}\cap\mathfrak{S}^{k,s}\cap\mathcal{C}^k_{a_n}\cap\mathcal{H}^k_n})\Big] \\ & \leq\|g\|_{\infty}(s\bm{L}_n)^k\Big(1-(1-n^{-3})^{\sqrtBis{n}/\varepsilon_1}\big(1-2kn^{-\w_0(4\delta_0)^{-1}}\big)^k\Big)\sup_{\Vect{p}\in\ProdSet{(\N^*)}{k}}\Eb\Big[\underset{|\Vect{x}|=\Vect{p}}{\sum_{\Vect{x}\in\Delta^k_{p}}}\prod_{i=1}^k\P^{\mathcal{E}}\big(T_{\VectCoord{x}{i}}<T^1\big)\Big],
\end{align*}
and by \eqref{UnifBound2} and \eqref{EspBound3}, $\P^*(\bigcup_{s=\varepsilon_1\sqrtBis{n}}^{\sqrtBis{n}/\varepsilon_1}\{|\mathcal{A}^k(\mathcal{D}_{n,T^s},g\un_{\mathfrak{E}^{k,s}\cap\mathfrak{S}^{k,s}\cap\mathcal{C}^k_{a_n}})-\sum_{\bm{j}\in\llbracket 1,s\rrbracket_k}\mathcal{A}^{k,n}(\Vect{j},g\un_{\mathcal{C}^k_{a_n}})|>\varepsilon (s\bm{L}_n)^k/2\})$ is smaller, for $n$ large enough, than
\begin{align*}
    1-\Pb^*(\mathcal{V}_n) + \mathfrak{C}_{\ref{LemmeSommeExc}}n^{-\rho_2'}+\mathfrak{C}_{\ref{LemmeSommeExc},1}\Big(1-(1-n^{-3})^{\sqrtBis{n}/\varepsilon_1}&\big(1-2kn^{-\w_0(4\delta_0)^{-1}}\big)^k\Big) \\ & \times\sup_{\Vect{p}\in\ProdSet{(\N^*)}{k}}\Eb\Big[\underset{|\Vect{x}|=\Vect{p}}{\sum_{\Vect{x}\in\Delta^k_{p}}}e^{-\langle\bm{\beta},V(\Vect{x})\rangle_k}\Big],
\end{align*}
for some constant $\mathfrak{C}_{\ref{LemmeSommeExc}}>0$ and $\mathfrak{C}_{\ref{LemmeSommeExc},1}>0$. Finally, by Lemma \ref{MinPotential}, for $n$ large enough $1-\Pb^*(\mathcal{V}_n)\leq n^{-\rho'_1}$ for some $\rho'_1>0$, $(1-(1-n^{-3})^{\sqrtBis{n}/\varepsilon_1}(1-2kn^{-\w_0(4\delta_0)^{-1}})^k)\leq n^{-\rho_{2,k}}$ for $n$ large enough and some $\rho_{2,k}>0$ and thanks to Lemma \ref{HighDownfalls} \ref{HighDownfalls1} with $\Vect{\beta}=\Vect{1}$, $\sup_{\Vect{p}\in\ProdSet{(\N^*)}{k}}\Eb[\sum_{\Vect{x}\in\Delta^k_{p};|\Vect{x}|=\Vect{p}}e^{-\langle\bm{\beta},V(\Vect{x})\rangle_k}]$ is finite which completes the proof.
\end{proof}
\noindent The next lemma relates $\sum_{\bm{j}\in\llbracket 1,s\rrbracket_k}\mathcal{A}^{k,n}(\Vect{j},f\un_{\mathcal{C}_{a_n}^k})$ with its quenched mean and illustrates why this quasi-independent version of the range is easier to deal with.
\begin{lemm}\label{GenVariance}
Let $k\geq 2$ and $\mathfrak{a}\geq 1$ be two integers and assume $\kappa>2\mathfrak{a}k$. Under the Assumptions \ref{Assumption1}, \ref{AssumptionSmallGenerations} and \ref{Assumption2}, there exits a constant $\mathfrak{C}_{\ref{GenVariance}}>0$ and a non-decreasing sequence of positive numbers $(\Tilde{\mathfrak{q}}(j))_{j\geq 2}$ satisfying $\Tilde{\mathfrak{q}}_2=1$ and $\Tilde{\mathfrak{q}}(j)\to\infty$ when $j\to\infty$ such that for $n$ large enough and any $\varepsilon_1\sqrtBis{n}\leq s\leq\sqrtBis{n}/\varepsilon_1$
\begin{align*}
    \E\Big[\Big(\sum_{\bm{j}\in\llbracket 1,s\rrbracket_k}\mathcal{A}^{k,n}(\Vect{j},f\un_{\mathcal{C}_{a_n}^k})-\E^{\mathcal{E}}\big[\sum_{\bm{j}\in\llbracket 1,s\rrbracket_k}\mathcal{A}^{k,n}(\Vect{j},f\un_{\mathcal{C}_{a_n}^k})\big]\Big)^{2\mathfrak{a}}\Big]\leq \mathfrak{C}_{\ref{GenVariance}}(\bm{L}_n)^{2\mathfrak{a}k}(\mathfrak{L}_n)^{\Tilde{\mathfrak{q}}_{\mathfrak{a}}}s^{2\mathfrak{a}k-\Tilde{\mathfrak{q}}_{\mathfrak{a}}}.
\end{align*}
\end{lemm}

\begin{proof}
Recall the definition of $\mathcal{A}^{k,n}(\Vect{j},f\un_{\mathcal{C}_{a_n}^k})$ in \eqref{DefIndepRange}. For $\mathfrak{a}=1$, note that 
\begin{align*}
    &\E^{\mathcal{E}}\Big[\Big(\sum_{\bm{j}\in\llbracket 1,s\rrbracket_k}\mathcal{A}^{k,n}(\Vect{j},f\un_{\mathcal{C}_{a_n}^k})\Big)^2\Big] \\ & =\sum_{\bm{j},\bm{j}'\in\llbracket 1,s\rrbracket_k}\;\sum_{\Vect{p},\Vect{p}'\in\ProdSet{\{\ell_n,\ldots,\mathfrak{L}_n\}}{k}}\;\underset{|\Vect{x}|=\Vect{p},|\Vect{y}|=\Vect{p}'}{\sum_{\Vect{x},\Vect{y}\in\mathcal{C}_{a_n}^k}}f(\Vect{x})f(\Vect{y})\E^{\mathcal{E}}\Big[\prod_{i=1}^k\un_{\{N^{T^{j_i}}_{\VectCoord{x}{i}}-N^{T^{j_i}-1}_{\VectCoord{x}{i}}\geq 1,\; N^{T^{j'_i}}_{\VectCoord{y}{i}}-N^{T^{j'_i}-1}_{\VectCoord{y}{i}}\geq 1\}}\Big],
\end{align*}
with the notation $\bm{j}=(j_1,\ldots,j_k)$ and $\bm{j}'=(j'_1,\ldots,j'_k)$. Thanks to the strong Markov property, the random variables $N_z^{T^{i}}-N_z^{T^{i-1}}$ are i.i.d under $\P^{\mathcal{E}}$ and distributed as $N_z^{T^1}$ for any $z\in\T$. In particular, the term $s^{2k}$ in $\E^{\mathcal{E}}[(\sum_{\bm{j}\in\llbracket 1,s\rrbracket_k}\mathcal{A}^{k,n}(\Vect{j},f\un_{\mathcal{C}_{a_n}^k})-\E^{\mathcal{E}}[\sum_{\bm{j}\in\llbracket 1,s\rrbracket_k}\mathcal{A}^{k,n}(\Vect{j},f\un_{\mathcal{C}_{a_n}^k})\big])^{2}]$ is equal to zero and we actually have
\begin{align*}
    &\E\Big[\Big(\sum_{\bm{j}\in\llbracket 1,s\rrbracket_k}\mathcal{A}^{k,n}(\Vect{j},f\un_{\mathcal{C}_{a_n}^k})-\E^{\mathcal{E}}\big[\sum_{\bm{j}\in\llbracket 1,s\rrbracket_k}\mathcal{A}^{k,n}(\Vect{j},f\un_{\mathcal{C}_{a_n}^k})\big]\Big)^{2}\Big] \\ & \leq\mathfrak{C}_{\ref{GenVariance},1}(\bm{L}_n)^{2k}\big((\mathfrak{L}_n)^2s^{2k-2}+\mathfrak{L}_ns^{2k-1}\big)\leq 2\mathfrak{C}_{\ref{GenVariance},1}(\bm{L}_n)^{2k}\mathfrak{L}_ns^{2k-1}, 
\end{align*}
where the constant $\mathfrak{C}_{\ref{GenVariance},1}>0$ comes from Lemma \ref{EspProbaTAPartition} and the last inequality comes the fact that $\mathfrak{L}_n\leq s$ for $n$ large enough. \\
When $\mathfrak{a}\geq 2$, using similar arguments we have
\begin{align*}
    &\E\Big[\Big(\sum_{\bm{j}\in\llbracket 1,s\rrbracket_k}\mathcal{A}^{k,n}(\Vect{j},f\un_{\mathcal{C}_{a_n}^k})-\E^{\mathcal{E}}\big[\sum_{\bm{j}\in\llbracket 1,s\rrbracket_k}\mathcal{A}^{k,n}(\Vect{j},f\un_{\mathcal{C}_{a_n}^k})\big]\Big)^{2\mathfrak{a}}\Big] \\ & \leq \mathfrak{C}_{\ref{GenVariance},2}(\bm{L}_n)^{2\mathfrak{a}k}(\mathfrak{L}_n)^{2\lfloor\mathfrak{a}/2\rfloor}s^{2\mathfrak{a}k-2\lfloor\mathfrak{a}/2\rfloor}.
\end{align*}
We finally obtain the result by taking $\mathfrak{q}_{\mathfrak{a}}:=\mathfrak{a}\un_{\{\mathfrak{a}=1\}}+2\lfloor\mathfrak{a}/2\rfloor\un_{\{\mathfrak{a}\geq 2\}}$.
\end{proof}

\subsubsection{Convergence of the quenched mean of the range on $\mathcal{C}^k_{a_n}$}

We prove that the quenched mean of the quasi-independent version $\sum_{\bm{j}\in\llbracket 1,s\rrbracket_k}\mathcal{A}^k(\bm{j},f\mathcal{C}^k_{a_n})$ of the range on the set $\mathcal{C}^k_{a_n}$ converges in $\Pb^*$-probability by using the hereditary Assumption \ref{Assumption3}. 
\begin{lemm}\label{ConvL1}
Let $k\geq 2$ be an integer and assume $\kappa>2k$. Under the Assumptions \ref{Assumption1}, \ref{Assumption2} and \ref{AssumptionIncrements}, if $f$ satisfies the hereditary Assumption \ref{Assumption3} then 
\begin{align*}
    \lim_{n\to\infty}\Eb^*\Big[\Big|\frac{1}{(\bm{L}_n)^k }\underset{\ell_n\leq|\Vect{x}|\leq\mathfrak{L}_n}{\sum_{\Vect{x}\in\Delta^k}}f\un_{\mathcal{C}_{a_n}^k}(\Vect{x})\prod_{i=1}^k\frac{e^{-V(\VectCoord{x}{i})}}{H_{\VectCoord{x}{i}}}-(c_{\infty})^k\mathcal{A}^k_{a_n}(f)\Big|\Big]=0
\end{align*}
\end{lemm}
\begin{proof}
Let us first prove that
\begin{align}\label{limUnifCarre'}
    \lim_{n\to\infty}\Eb\Big[\Big(\frac{1}{(\bm{L}_n)^{k}}\underset{\ell_n\leq|\Vect{x}|\leq\mathfrak{L}_n}{\sum_{\Vect{x}\in\mathcal{C}_{a_n}^k}}f(\Vect{x})\prod_{i=1}^k\frac{e^{-V(\VectCoord{x}{i})}}{H_{\VectCoord{x}{i}}}-\;\sum_{\Vect{x}\in\Delta^k_{a_n}}f(\Vect{x})\prod_{i=1}^ke^{-V(\VectCoord{x}{i})}\Tilde{\varphi}_{n}(H_{\VectCoord{x}{i}})\Big)^2\Big]=0,
\end{align}
where, for any $r\geq 1$, $\Tilde{\varphi}_n(r):=\sum_{p=\ell_n}^{\mathfrak{L}n}\varphi_{n,p}(r)/\bm{L}_n$ and 
\begin{align*}
    \varphi_{n,p}(r)=\Eb\Big[\sum_{|x|=p-a_n}e^{-V(x)}\big((r-1)e^{-V(x)}+H_x\big)^{-1}\Big].
\end{align*}
For that, the first step is to decompose $\sum_{\Vect{x}\in\Delta^k;\ell_n\leq|\Vect{x}|\leq\mathfrak{L}_n}f\un_{\mathcal{C}_{a_n}^k}(\Vect{x})\prod_{i=1}^ke^{-V(\VectCoord{x}{i})}/H_{\VectCoord{x}{i}}$:
\begin{align*}
    \underset{\ell_n\leq|\Vect{x}|\leq\mathfrak{L}_n}{\sum_{\Vect{x}\in\Delta^k}}f\un_{\mathcal{C}_{a_n}^k}(\Vect{x})\prod_{i=1}^k\frac{e^{-V(\VectCoord{x}{i})}}{H_{\VectCoord{x}{i}}}&=\sum_{\Vect{p}\in\ProdSet{\{\ell_n,\ldots,\mathfrak{L}_n\}}{k}}\sum_{\bm{z}\in\Delta^k_{a_n}}\underset{|\Vect{x}|=\Vect{p};\;\VectCoord{x}{i}>\VectCoord{z}{i}}{\sum_{\Vect{x}\in\Delta^k}}f\un_{\mathcal{C}_{a_n}^k}(\Vect{x})\prod_{i=1}^k\frac{e^{-V(\VectCoord{x}{i})}}{H_{\VectCoord{x}{i}}} \\ & =\sum_{\Vect{p}\in\ProdSet{\{\ell_n,\ldots,\mathfrak{L}_n\}}{k}}\sum_{\bm{z}\in\Delta^k_{a_n}}f(\Vect{z})\underset{|\Vect{x}|=\Vect{p};\;\VectCoord{x}{i}>\VectCoord{z}{i}}{\sum_{\Vect{x}\in\Delta^k}}\prod_{i=1}^k\frac{e^{-V(\VectCoord{x}{i})}}{H_{\VectCoord{x}{i}}},
\end{align*}
where the last equality comes from the hereditary Assumption \ref{Assumption3}. As we did above, we decompose $H_{\VectCoord{x}{i}}$: $H_{\VectCoord{x}{i}}=(H_{\VectCoord{z}{i}}-1)e^{-V_{\VectCoord{z}{i}}(\VectCoord{x}{i})}+H_{\VectCoord{z}{i},\VectCoord{x}{i}}$. By independence of the increments of the branching random walk $(\T,(V(x),x\in\T))$
\begin{align}\label{GenEspCondi'}
    \Eb\Big[\underset{|\Vect{x}|=\Vect{p}}{\sum_{\Vect{x}\in\Delta^k}}f\un_{\mathcal{C}_{a_n}^k}(\Vect{x})\prod_{i=1}^k\frac{e^{-V(\VectCoord{x}{i})}}{H_{\VectCoord{x}{i}}}\Big|\mathcal{F}_{a_n}\Big]=\sum_{\Vect{z}\in\Delta^k_{a_n}}f(\Vect{z})\prod_{i=1}^ke^{-V(\VectCoord{z}{i})}\varphi_{n,p_i}(H_{\VectCoord{z}{i}}),
\end{align}
where  $\mathcal{F}_{a_n}=\sigma((x,V(x));|x|\leq a_n)$. Thanks to \eqref{GenEspCondi'}, we have that the expectation in equation \eqref{limUnifCarre'} is equal to 
\begin{align*}
    \Eb\Big[\Big(\frac{1}{(\bm{L}_n)^{k}}\underset{\ell_n\leq|\Vect{x}|\leq\mathfrak{L}_n}{\sum_{\Vect{x}\in\mathcal{C}_{a_n}^k}}f(\Vect{x})\prod_{i=1}^k\frac{e^{-V(\VectCoord{x}{i})}}{H_{\VectCoord{x}{i}}}\Big)^2\Big]-\Eb\Big[\Big(\sum_{\Vect{x}\in\Delta^k_{a_n}}f(\Vect{x})\prod_{i=1}^ke^{-V(\VectCoord{x}{i})}\Tilde{\varphi}_{n}(H_{\VectCoord{x}{i}})\Big)^2\Big].
\end{align*}
For $\Vect{x},\Vect{y}\in\Delta^k$, denote by $\Vect{x}\Vect{y}=(\VectCoord{x}{1},\ldots,\VectCoord{x}{k},\VectCoord{y}{1},\ldots,\VectCoord{y}{k})$ the concatenation of $\Vect{x}$ and $\Vect{y}$. Note that
\begin{align*}
    \Big(\underset{\ell_n\leq|\Vect{x}|\leq\mathfrak{L}_n}{\sum_{\Vect{x}\in\mathcal{C}_{a_n}^k}}f(\Vect{x})\prod_{i=1}^k\frac{e^{-V(\VectCoord{x}{i})}}{H_{\VectCoord{x}{i}}}\Big)^2= &\underset{\ell_n\leq|\Vect{x}|,|\Vect{y}|\leq\mathfrak{L}_n}{\sum_{\Vect{x},\Vect{y}\in\mathcal{C}_{a_n}^k;\; \Vect{xy}\not\in\Delta^{2k}}}f(\Vect{x})f(\Vect{y})\prod_{i=1}^k\frac{e^{-V(\VectCoord{x}{i})}}{H_{\VectCoord{x}{i}}}\frac{e^{-V(\VectCoord{y}{i})}}{H_{\VectCoord{y}{i}}} \\ & +\underset{\ell_n\leq|\Vect{x}|,|\Vect{y}|\leq\mathfrak{L}_n}{\sum_{\Vect{x},\Vect{y}\in\mathcal{C}_{a_n}^k;\; \Vect{xy}\in\Delta^{2k}\setminus\mathcal{C}^{2k}_{a_n}}}f(\Vect{x})f(\Vect{y})\prod_{i=1}^k\frac{e^{-V(\VectCoord{x}{i})}}{H_{\VectCoord{x}{i}}}\frac{e^{-V(\VectCoord{y}{i})}}{H_{\VectCoord{y}{i}}} \\ & + \underset{\ell_n\leq|\Vect{x}|,|\Vect{y}|\leq\mathfrak{L}_n}{\sum_{\Vect{x},\Vect{y}\in\Delta^k;\; \Vect{xy}\in\mathcal{C}_{a_n}^{2k}}}f(\Vect{x})f(\Vect{y})\prod_{i=1}^k\frac{e^{-V(\VectCoord{x}{i})}}{H_{\VectCoord{x}{i}}}\frac{e^{-V(\VectCoord{y}{i})}}{H_{\VectCoord{y}{i}}},
\end{align*}
where for any $\Vect{x},\Vect{y}\in\Delta^k$, $\Vect{x}\Vect{y}\not\in\Delta^{2k}$ means that there exists $\alpha\in\{1,\ldots,k\}$ and $i_1,\ldots,i_{\alpha}\in\{1,\ldots,k\}$ distinct such that $\VectCoord{x}{i_j}=\VectCoord{y}{i_j}$ for all $j\in\{1,\ldots,\alpha\}$ and $\ell_n\leq |\Vect{x}|,|\Vect{y}|\leq\mathfrak{L}_n$ means nothing but $\ell_n\leq |\Vect{x}|\leq\mathfrak{L}_n$ and $\ell_n\leq|\Vect{y}|\leq\mathfrak{L}_n$. It follows
\begin{align*}
    \lim_{n\to\infty}\frac{1}{(\bm{L}_n)^{2k}}\Eb\Big[\underset{\ell_n\leq|\Vect{x}|,|\Vect{y}|\leq\mathfrak{L}_n}{\sum_{\Vect{x},\Vect{y}\in\mathcal{C}_{a_n}^k;\; \Vect{xy}\not\in\Delta^{2k}}}f(\Vect{x})f(\Vect{y})\prod_{i=1}^k\frac{e^{-V(\VectCoord{x}{i})}}{H_{\VectCoord{x}{i}}}\frac{e^{-V(\VectCoord{y}{i})}}{H_{\VectCoord{y}{i}}}\Big]=0.
\end{align*}
Indeed, using that $H_z\geq 1$ for any $z\in\T$, we have
\begin{align*}
    &\Eb\Big[\underset{\ell_n\leq|\Vect{x}|,|\Vect{y}|\leq\mathfrak{L}_n}{\sum_{\Vect{x},\Vect{y}\in\mathcal{C}_{a_n}^k;\; \Vect{xy}\not\in\Delta^{2k}}}f(\Vect{x})f(\Vect{y})\prod_{i=1}^k\frac{e^{-V(\VectCoord{x}{i})}}{H_{\VectCoord{x}{i}}}\frac{e^{-V(\VectCoord{y}{i})}}{H_{\VectCoord{y}{i}}}\Big] \\ & \leq\|f\|_{\infty}^2(\bm{L}_n)^{2k}\sum_{\alpha=1}^k\;\sum_{i_1\not=i_2\ldots\not=i_{\alpha}=1}^k\sup_{\ell_n\leq\bm{q}\leq\mathfrak{L}_n}\Eb\Big[\underset{|\Vect{u}|=\bm{q}}{\sum_{\Vect{u}\in\Delta^{2k-\alpha}}}\prod_{j=1}^{\alpha}e^{-2V(\VectCoord{u}{i_j})}\underset{i'\not\in\{i_1,\ldots,i_{\alpha}\}}{\prod_{i'=1}^k}e^{-V(\VectCoord{u}{i'})}\Big]. 
\end{align*}
One can decompose according to the value of $\mathcal{S}^{2k-\alpha}(\Vect{u})$. We have (see the proof of Lemma \ref{CauchyCarre} for example)
\begin{align*}
    \lim_{n\to\infty}\sup_{\ell_n\leq\bm{q}\leq\mathfrak{L}_n}\Eb\Big[\underset{|\Vect{u}|=\bm{q}}{\sum_{\Vect{u}\in\Delta^{2k-\alpha}}}\un_{\{\mathcal{S}^{2k-\alpha}(\Vect{u})>a_n\}}\prod_{j=1}^{\alpha}e^{-2V(\VectCoord{u}{i_j})}\underset{i'\not\in\{i_1,\ldots,i_{\alpha}\}}{\prod_{i'=1}^k}e^{-V(\VectCoord{u}{i'})}\Big]=0,
\end{align*}
and by independence of the increments of the branching random walk $(\T,(V(x),x\in\T))$ and the facts that $\psi(1)=0$ and $\psi(2)<0$
\begin{align*}
    &\sup_{\ell_n\leq\bm{q}\leq\mathfrak{L}_n}\Eb\Big[\underset{|\Vect{u}|=\bm{q}}{\sum_{\Vect{u}\in\mathcal{C}^{2k-\alpha}_{a_n}}}\prod_{j=1}^{\alpha}e^{-2V(\VectCoord{u}{i_j})}\underset{i'\not\in\{i_1,\ldots,i_{\alpha}\}}{\prod_{i'=1}^k}e^{-V(\VectCoord{u}{i'})}\Big] \\ & =\sup_{\ell_n\leq\bm{q}\leq\mathfrak{L}_n}\prod_{j=1}^{\alpha}e^{(q_{i_j}-a_n)\psi(2)}\Eb\Big[\sum_{\Vect{z}\in\Delta^{2k-\alpha}_{a_n}}\prod_{j=1}^{\alpha}e^{-2V(\VectCoord{z}{i_j})}\underset{i'\not\in\{i_1,\ldots,i_{\alpha}\}}{\prod_{i'=1}^k}e^{-V(\VectCoord{z}{i'})}\Big] \\ & \leq \mathfrak{C}_{\ref{ConvL1}}e^{\alpha(\ell_n-a_n)\psi(2)},
\end{align*}
where $\mathfrak{C}_{\ref{ConvL1}}>0$ is a constant coming from Lemma \ref{HighDownfalls}, thus giving the convergence we wanted, recalling that $a_n\leq\ell_n/2$. Similarly, we have
\begin{align*}
    \lim_{n\to\infty}\frac{1}{(\bm{L}_n)^{2k}}\Eb\Big[\underset{\ell_n\leq|\Vect{x}|,|\Vect{y}|\leq\mathfrak{L}_n}{\sum_{\Vect{x},\Vect{y}\in\mathcal{C}_{a_n}^k;\; \Vect{xy}\in\Delta^{2k}\setminus\mathcal{C}^{2k}_{a_n}}}f(\Vect{x})f(\Vect{y})\prod_{i=1}^k\frac{e^{-V(\VectCoord{x}{i})}}{H_{\VectCoord{x}{i}}}\frac{e^{-V(\VectCoord{y}{i})}}{H_{\VectCoord{y}{i}}}\Big]=0,
\end{align*}
thus giving
\begin{align}\label{EspDiffCarre1}
    \lim_{n\to\infty}\frac{1}{(\bm{L}_n)^{2k}}\Eb\Big[\Big(\underset{\ell_n\leq|\Vect{x}|\leq\mathfrak{L}_n}{\sum_{\Vect{x}\in\mathcal{C}_{a_n}^k}}f(\Vect{x})\prod_{i=1}^k&\frac{e^{-V(\VectCoord{x}{i})}}{H_{\VectCoord{x}{i}}}\Big)^2\nonumber \\ & -\underset{\ell_n\leq|\Vect{x}|,|\Vect{y}|\leq\mathfrak{L}_n}{\sum_{\Vect{x},\Vect{y}\in\Delta^k;\; \Vect{xy}\in\mathcal{C}_{a_n}^{2k}}}f(\Vect{x})f(\Vect{y})\prod_{i=1}^k\frac{e^{-V(\VectCoord{x}{i})}}{H_{\VectCoord{x}{i}}}\frac{e^{-V(\VectCoord{y}{i})}}{H_{\VectCoord{y}{i}}}\Big]=0.
\end{align}
Exact same arguments yield
\begin{align}\label{EspDiffCarre2}
    \lim_{n\to\infty}\Eb\Big[\Big(\sum_{\Vect{x}\in\Delta^k_{a_n}}&f(\Vect{x})\prod_{i=1}^ke^{-V(\VectCoord{x}{i})}\Tilde{\varphi}_{n}(H_{\VectCoord{x}{i}})\Big)^2\nonumber \\ & -\sum_{\Vect{x},\Vect{y}\in\Delta^k_{a_n};\; \Vect{xy}\in\Delta^{2k}_{a_n}}f(\Vect{x})f(\Vect{y})\prod_{i=1}^ke^{-V(\VectCoord{x}{i})}\Tilde{\varphi}_{n}(H_{\VectCoord{x}{i}})e^{-V(\VectCoord{y}{i})}\Tilde{\varphi}_{n}(H_{\VectCoord{y}{i}})\Big]=0.
\end{align}
Finally, similarly as equation \eqref{GenEspCondi'}, using again the hereditary Assumption \ref{Assumption3}, we have
\begin{align*}
    &\Eb\Big[\frac{1}{(\bm{L}_n)^{2k}}\underset{\ell_n\leq|\Vect{x}|,|\Vect{y}|\leq\mathfrak{L}_n}{\sum_{\Vect{x},\Vect{y}\in\Delta^k;\; \Vect{xy}\in\mathcal{C}_{a_n}^{2k}}}f(\Vect{x})f(\Vect{y})\prod_{i=1}^k\frac{e^{-V(\VectCoord{x}{i})}}{H_{\VectCoord{x}{i}}}\frac{e^{-V(\VectCoord{y}{i})}}{H_{\VectCoord{y}{i}}}\Big] \\ & =\Eb\Big[\sum_{\Vect{x},\Vect{y}\in\Delta^k_{a_n};\; \Vect{xy}\in\Delta^{2k}_{a_n}}f(\Vect{x})f(\Vect{y})\prod_{i=1}^ke^{-V(\VectCoord{x}{i})}\Tilde{\varphi}_{n}(H_{\VectCoord{x}{i}})e^{-V(\VectCoord{y}{i})}\Tilde{\varphi}_{n}(H_{\VectCoord{y}{i}})\Big],
\end{align*}
so \eqref{EspDiffCarre1} and \eqref{EspDiffCarre2} yield \eqref{limUnifCarre'}. We now prove that 
\begin{align}\label{limEspcInf}
    \lim_{n\to\infty}\Eb\Big[\Big|(c_{\infty})^k\sum_{\Vect{z}\in\Delta^k_{a_n}}f(\Vect{z})\prod_{i=1}^ke^{-V(\VectCoord{z}{i})}-\sum_{\Vect{z}\in\Delta^k_{a_n}}f(\Vect{z})\prod_{i=1}^ke^{-V(\VectCoord{z}{i})}\Tilde{\varphi}_{n}(H_{\VectCoord{z}{i}})\Big|\Big]=0.
\end{align}
Let $h_n=\log n$ (the choice of $h_n$ is almost arbitrary, $h_n\to\infty$ with $h_n=o(n^{\theta})$ for all $\theta>0$ should be enough). Note that $|(c_{\infty})^k-\prod_{i=1}^k\Tilde{\varphi}_{n}(H_{\VectCoord{z}{i}})|\leq 2$ so
\begin{align*}
    &\Eb\Big[\Big|(c_{\infty})^k\sum_{\Vect{z}\in\Delta^k_{a_n}}f(\Vect{z})\prod_{i=1}^ke^{-V(\VectCoord{z}{i})}-\sum_{\Vect{z}\in\Delta^k_{a_n}}f(\Vect{z})\prod_{i=1}^ke^{-V(\VectCoord{z}{i})}\Tilde{\varphi}_{n}(H_{\VectCoord{z}{i}})\Big|\Big] \\ & \leq\fInf\Eb\Big[\sum_{\Vect{z}\in\Delta^k_{a_n}}\Big(\prod_{i=1}^ke^{-V(\VectCoord{z}{i})}\un_{\{H_{\VectCoord{z}{i}}\leq h_n\}}\Big)\Big|(c_{\infty})^k-\prod_{i=1}^k\Tilde{\varphi}_{n}(H_{\VectCoord{z}{i}})\Big|\Big] \\ & +2\fInf\Eb\Big[\sum_{\Vect{z}\in\Delta^k_{a_n}}\un_{\{\max_{1\leq i\leq k}H_{\VectCoord{z}{i}}>h_n\}}e^{-\langle\Vect{1},V(\Vect{z})\rangle_k}\Big].
\end{align*}
We show that $\lim_{n\to\infty}\sup_{1\leq r_1,\ldots,r_k\leq h_n}|(c_{\infty})^k-\prod_{i=1}^k\Tilde{\varphi}_{n}(r_i)|=0$. For that, on the first hand, one can see that $\varphi_{n,p}(r)\leq\Eb[1/H_{\ell_n-a_n}^S]$ where we recall that $H^S_{m}=\sum_{j=0}^me^{S_j-S_m}$ (see \eqref{RW} for the definition of the random walk $S$). On the other, for any $\ell_n\leq p\leq\mathfrak{L}_n$ and $1\leq r\leq h_n$, $ \varphi_{n,p}(r)$ is larger, for any $\Tilde{r}>0$, than 
\begin{align*}
   \Eb\Big[\sum_{|x|=p-a_n}\frac{e^{-V(x)}}{h_ne^{-V(x)}+H_x}\un_{\{V(x)\geq \Tilde{r}\log n\}}\Big]\geq\Eb\Big[\frac{1}{h_nn^{-\Tilde{r}}+H_{\mathfrak{L}_n-a_n}^S}\Big]-\Pb(S_{p-a_n}<\Tilde{r}\log n).
\end{align*}
where we have used the many-to-one Lemma \ref{many-to-one}. \\
Note that $\Pb(S_{p-a_n}<\Tilde{r}\log n)\leq\Pb(\min_{(2\delta_0)^{-1}\log n\leq j\leq\mathfrak{L}_n}S_j<\Tilde{r}\log n)\to 0$ when $n\to\infty$ for some $\Tilde{r}>0$ since $a_n=(2\delta_0)^{-1}\log n$ and $\psi'(1)<0$. Moreover, by definition, both $(\Eb[1/H_{\ell_n-a_n}^S])$ and $(\Eb[1/(h_nn^{-\Tilde{r}}+H_{\mathfrak{L}_n-a_n}^S)])$ goes to $c_{\infty}$ when $n$ goes to $\infty$ and we obtain the convergence. Then
\begin{align*}
    &\Eb\Big[\Big|(c_{\infty})^k\sum_{\Vect{z}\in\Delta^k_{a_n}}f(\Vect{z})\prod_{i=1}^ke^{-V(\VectCoord{z}{i})}-\sum_{\Vect{z}\in\Delta^k_{a_n}}f(\Vect{z})\prod_{i=1}^ke^{-V(\VectCoord{z}{i})}\Tilde{\varphi}_{n}(H_{\VectCoord{z}{i}})\Big|\Big] \\ & \leq\fInf\Eb\Big[\sum_{\Vect{z}\in\Delta^k_{a_n}}e^{-\langle\Vect{1},V(\Vect{z})\rangle_k}\Big]\sup_{1\leq r_1,\ldots,r_k\leq h_n}\Big|(c_{\infty})^k-\prod_{i=1}^k\Tilde{\varphi}_{n}(r_i)\Big| \\ & +2\fInf\Eb\Big[\sum_{\Vect{z}\in\Delta^k_{a_n}}\un_{\{\max_{1\leq i\leq k}H_{\VectCoord{z}{i}}>h_n\}}e^{-\langle\Vect{1},V(\Vect{z})\rangle_k}\Big].
\end{align*}
Using Lemma \ref{HighDownfalls}, first \textit{\ref{HighDownfalls1}}, then \textit{\ref{HighDownfalls2}} with $h=h_n$, $\sup_{n\in\N}\Eb[\sum_{\Vect{z}\in\Delta^k_{a_n}}e^{-\langle\Vect{1},V(\Vect{x})\rangle_k}]<\infty$ and $\lim_{n\to\infty}\Eb[\sum_{\Vect{z}\in\Delta^k_{a_n}}\un_{\{\max_{1\leq i\leq k}H_{\VectCoord{z}{i}}>h_n\}}e^{-\langle\Vect{1},V(\Vect{z})\rangle_k}]=0$ thus giving \eqref{limEspcInf}. \\
Finally, putting together \eqref{limUnifCarre'} and \eqref{limEspcInf} yields the result.
\end{proof}

\subsubsection{Convergence of the quasi-martingale \texorpdfstring{$\mathcal{A}^k_l(f)$}{1}}\label{QuasiMartingale}

Recall that 
$$\mathcal{A}^k_{l}(f,\Vect{\beta})=\sum_{\Vect{x}\in\Delta^k_l}f(\Vect{x})e^{-\langle\Vect{\beta},V(\Vect{x})\rangle_{k}}=\sum_{\Vect{x}\in\Delta^k_l}f(\Vect{x})\prod_{i=1}^ke^{-\beta_iV(\VectCoord{x}{i})}\;\;\textrm{ and }\;\;\mathcal{A}^k_l(f)=\mathcal{A}^k_l(f,\Vect{1}).$$
The aim of this subsection is to prove that $\mathcal{A}^k_{\infty}:=\lim_{l\to\infty}\mathcal{A}^k_l(f)$ exists when $f$ satisfies our hereditary Assumption \ref{Assumption3}. For that, let us define for any $\Vect{p}\in\ProdSet{(\N^*)}{k}$
\begin{align*}
    \mathcal{A}^{k}_{\Vect{p}}(f):=\underset{|\Vect{x}|=\Vect{p}}{\sum_{\Vect{x}\in\Delta^k}}f(\Vect{x})e^{-\langle\Vect{1},V(\Vect{x})\rangle_{k}},
\end{align*}
where we recall that for any $\Vect{x}=(\VectCoord{x}{1},\ldots,\VectCoord{x}{k})\in\Delta^k$, $|\Vect{x}|=\Vect{p}$ if and only if $|\VectCoord{x}{i}|=p_i$ for all $i\in\{1,\ldots,k\}$. One can notice that when $\Vect{p}=(l,\ldots,l)\in\ProdSet{(\N^*)}{k}$, we have $\mathcal{A}^k_l(f)=\mathcal{A}^k_{\Vect{p}}(f)$.
\begin{lemm}\label{CauchyCarre}
Let $k\geq 2$ be an integer and assume $\kappa>2k$. Under the Assumptions \ref{Assumption1}, \ref{Assumption2} and \ref{AssumptionIncrements}, for any bounded function $f:\Delta^k\to\R^+$, there exists two constants $\mathfrak{C}_{\ref{CauchyCarre}}>0$ and $\mathfrak{b}\in(0,1)$ such that for any $\Vect{p}\in\ProdSet{(\N^*)}{k}$ and any integer $m\geq 1$ such that $m\leq\max\Vect{p}:=\max_{1\leq i\leq }p_i$
\begin{align*}
    \Eb^*\big[\big|\mathcal{A}^k_{\Vect{p}}(f\un_{\mathcal{C}^k_m})-\mathcal{A}_{\Vect{p}}^k(f)\big|^2\big]\leq\mathfrak{C}_{\ref{CauchyCarre}}e^{-\mathfrak{b}m}.
\end{align*}
\end{lemm}
\begin{proof}
In order to avoid unnecessary technical difficulties, we prove it for any $\kappa>4$. First note that $\mathcal{A}_{\Vect{p}}^k(f)-\mathcal{A}^k_{\Vect{p}}(f\un_{\mathcal{C}^k_m})=\sum_{\Vect{x}\in\Delta^k;\; |\Vect{x}|=\Vect{p}}f(\Vect{x})\un_{\{\mathcal{S}^k(\Vect{x})>m\}}e^{-\langle\Vect{1},V(\Vect{x})\rangle_{k}}$ which is smaller than $\fInf\sum_{\Vect{x}\in\Delta^k;\; |\Vect{x}|=\Vect{p}}\un_{\{\mathcal{S}^k(\Vect{x})>m\}}e^{-\langle\Vect{1},V(\Vect{x})\rangle_{k}}$. Using a similar argument as we developed in the proof of Lemma \ref{ConvL1}, it is enough to show the following estimation:
\begin{align}\label{LargeGenerationCA}
    \Eb^*\Big[\underset{|\Vect{x}|=\Vect{p}}{\sum_{\Vect{x}\in\Delta^q}}\un_{\{\mathcal{S}^q(\Vect{x})>m\}}e^{-\langle\Vect{1},V(\Vect{x})\rangle_{q}}\Big]\leq\mathfrak{C}_{\ref{CauchyCarre},1}e^{-\mathfrak{b}m},
\end{align}
for any $q\in\{k,\ldots,2k\}$ and some constant $\mathfrak{C}_{\ref{CauchyCarre},1}>0$. Assume that $\min\Vect{p}<\max\Vect{p}$ (the proof is similar when $\min\Vect{p}=\max\Vect{p}$). Note that if $m<\min\Vect{p}$, then
\begin{align*}
    \Eb\Big[\underset{|\Vect{x}|=\Vect{p}}{\sum_{\Vect{x}\in\Delta^q}}\un_{\{\mathcal{S}^q(\Vect{x})>m\}}e^{-\langle\Vect{1},V(\Vect{x})\rangle_{q}}\Big]=&\Eb\Big[\underset{|\Vect{x}|=\Vect{p}}{\sum_{\Vect{x}\in\Delta^q}}\un_{\{m<\mathcal{S}^q(\Vect{x})\leq\min\Vect{p}\}}e^{-\langle\Vect{1},V(\Vect{x})\rangle_{q}}\Big] \\ & +\Eb\Big[\underset{|\Vect{x}|=\Vect{p}}{\sum_{\Vect{x}\in\Delta^q}}\un_{\{\mathcal{S}^q(\Vect{x})>\min\Vect{p}\}}e^{-\langle\Vect{1},V(\Vect{x})\rangle_{q}}\Big].
\end{align*}
One can notice that, if $|\Vect{x}|=\Vect{p}$ and $\mathcal{S}^q(x)\leq\min\Vect{p}$, then $\mathcal{S}^q(x)=\mathcal{S}^q(u)$ for any $\Vect{u}\in\Delta^q$ such that $\max|\Vect{u}|=\min|\Vect{u}|=\min\Vect{p}$. Hence, as usual
\begin{align*}
    \Eb\Big[\underset{|\Vect{x}|=\Vect{p}}{\sum_{\Vect{x}\in\Delta^q}}\un_{\{m<\mathcal{S}^q(\Vect{x})\leq\min\Vect{p}\}}e^{-\langle\Vect{1},V(\Vect{x})\rangle_{q}}\Big]&=\Eb\Big[\sum_{\Vect{u}\in\Delta^q_{\min\Vect{p}}}\un_{\{\mathcal{S}^q(\Vect{u})>m\}}\sum_{\Vect{x}\in\Delta^q;\; \Vect{x}\geq\Vect{u}}e^{-\langle\Vect{1},V(\Vect{x})\rangle_{q}}\Big] \\ & =\Eb\Big[\sum_{\Vect{u}\in\Delta^q_{\min\Vect{p}}}\un_{\{\mathcal{S}^q(\Vect{u})>m\}}e^{-\langle\Vect{1},V(\Vect{u})\rangle_{q}}\Big],
\end{align*}
thus giving
\begin{align*}
    \Eb\Big[\underset{|\Vect{x}|=\Vect{p}}{\sum_{\Vect{x}\in\Delta^q}}\un_{\{\mathcal{S}^q(\Vect{x})>m\}}e^{-\langle\Vect{1},V(\Vect{x})\rangle_{q}}\Big]=&\Eb\Big[\sum_{\Vect{u}\in\Delta^q_{\min\Vect{p}}}\un_{\{\mathcal{S}^q(\Vect{u})>m\}}e^{-\langle\Vect{1},V(\Vect{u})\rangle_{q}}\Big] \\ & +\Eb\Big[\underset{|\Vect{x}|=\Vect{p}}{\sum_{\Vect{x}\in\Delta^q}}\un_{\{\mathcal{S}^q(\Vect{x})>\min\Vect{p}\}}e^{-\langle\Vect{1},V(\Vect{x})\rangle_{q}}\Big].
\end{align*}
We deduce from this equality that it is enough to prove \eqref{LargeGenerationCA} for any $m\leq\min\Vect{p}$ with $q\geq 3$. Again, we focus on the case $\min\Vect{p}<\max\Vect{p}$. \\
Assume $m\leq\min\Vect{p}$. Let $\Vect{x}\in\Delta^q$ such that $|\Vect{x}|=\Vect{p}$ and $\mathcal{C}^q(\Vect{x})>m$. There exists an integer $\mathfrak{f}\in\{m+1,\ldots,\max\Vect{p}\}$ such that, seen backwards in time, at least two vertices among $\VectCoord{x}{1},\ldots,\VectCoord{x}{q}$ share a common ancestor for the first times in the generation $\mathfrak{f}-1$ and there exits at least one vertex among these vertices in a generation smaller or equal to $\mathfrak{f}-1$. Then, one can notice that 
\begin{align*}
   \underset{|\Vect{x}|=\Vect{p}}{\sum_{\Vect{x}\in\Delta^q}}\un_{\{\mathcal{S}^q(\Vect{x})>m\}}e^{-\langle\Vect{1},V(\Vect{x})\rangle_{q}}=\sum_{\mathfrak{f}=m+1}^{\max\Vect{p}}\;\underset{\{1,\ldots,q\},\;|\bm{\pi}|<q}{\sum_{\bm{\pi}\textrm{ partition of}}}\;\underset{|\Vect{x}|=\Vect{p}}{\sum_{\Vect{x}\in\Delta^q}}\un_{\Upsilon_{\mathfrak{f}-1,\bm{\pi}}\cap\Upsilon_{\mathfrak{f},\bm{\eta}}}(\Vect{x})e^{-\langle\Vect{1},V(\Vect{x})\rangle_{q}},
\end{align*}
where $\bm{\eta}=\{\{1\},\ldots,\{q\}\}$ (recall the definition of $\Upsilon_{p-1,\bm{\pi}}\cap\Upsilon_{p,\bm{\eta}}$ in \eqref{Def_f_partition}). \\
By definition, there exists $\mathfrak{y}\in\{1,\ldots,q-2\}$ and $(i_1,\ldots,i_{\mathfrak{y}},i_{\mathfrak{y}+1},\ldots,i_q)\in\llbracket 1,q\rrbracket_q$ such that $\max_{1\leq\mathfrak{l}\leq\mathfrak{y}}p_{i_{\mathfrak{l}}}\leq\mathfrak{f}-1$ and $\min_{\mathfrak{y}+1\leq\mathfrak{l}\leq q}p_{i_{\mathfrak{l}}}\geq\mathfrak{f}-1$. By definition of the set $\Upsilon_{\cdot,\cdot}$, for all $\mathfrak{l}\in\{1,\ldots,\mathfrak{y}\}$, if $i_{\mathfrak{l}}$ belongs to the block $\bm{B}$ of the partition $\bm{\pi}$, then $\bm{B}=\{i_{\mathfrak{l}}\}$. Let $\Bar{\bm{\pi}}:=\bm{\pi}\setminus\{\{i_1\},\ldots,\{i_{\mathfrak{y}}\}\}$ and for all $j\in\{1,\ldots,|\bm{\pi}|-\mathfrak{y}\}$, denote by $\Bar{\bm{B}}_j$ the $j$-th block (ordered by their least element) of the partition $\Bar{\bm{\pi}}$ of the set $\{i_{\mathfrak{y}+1},\ldots,i_q\}=\{1,\ldots,q\}\setminus\{i_1,\ldots,i_{\mathfrak{y}}\}$. We have
\begin{align*}
    \Eb\Big[\underset{|\Vect{x}|=\Vect{p}}{\sum_{\Vect{x}\in\Delta^q}}\un_{\Upsilon_{\mathfrak{f}-1,\bm{\pi}}\cap\Upsilon_{\mathfrak{f},\bm{\eta}}}(\Vect{x})e^{-\langle\Vect{1},V(\Vect{x})\rangle_{q}}\big|\mathcal{F}_{\mathfrak{f}}\Big]= \underset{|\Vect{u}|=\Vect{p}_{\cdot}}{\sum_{\Vect{u}\in\Delta^{\mathfrak{y}}}}e^{-\langle\Vect{1},V(\Vect{u})\rangle_{\mathfrak{y}}}&\sum_{\Vect{z}\in\Delta_{\mathfrak{f}-1}^{|\bm{\pi}|-\mathfrak{y}}}\prod_{j=1}^{|\bm{\pi}|-\mathfrak{y}}\sum_{\VectCoord{\Vect{v}}{j}\in\Delta_{\mathfrak{f}}^{|\Bar{\bm{B}}_j|}}\prod_{i=1}^{|\Bar{\bm{B}}_j|} \\ & \times\un_{\{(\VectCoord{v}{j,i})^*=\VectCoord{z}{j}\}}e^{-V(\VectCoord{v}{j,i})},
\end{align*}
where $|\Vect{u}|=\Vect{p}_{\cdot}$ means that $\VectCoord{u}{\mathfrak{l}}=p_{i_{\mathfrak{l}}}$ for all $\mathfrak{l}\in\{1,\ldots,\mathfrak{y}\}$, $\VectCoord{\Vect{v}}{j}=(\VectCoord{v}{j,1},\ldots,\VectCoord{v}{j,|\Bar{\bm{B}}_j|})$. Thus
\begin{align*}
    \Eb\Big[\underset{|\Vect{x}|=\Vect{p}}{\sum_{\Vect{x}\in\Delta^q}}\un_{\Upsilon_{\mathfrak{f}-1,\bm{\pi}}\cap\Upsilon_{\mathfrak{f},\bm{\eta}}}(\Vect{x})e^{-\langle\Vect{1},V(\Vect{x})\rangle_{q}}\big|\mathcal{F}_{\mathfrak{f}-1}\Big]&=\underset{|\Vect{u}|=\Vect{p}_{\cdot}}{\sum_{\Vect{u}\in\Delta^{\mathfrak{y}}}}e^{-\langle\Vect{1},V(\Vect{u})\rangle_{\mathfrak{y}}}\sum_{\Vect{z}\in\Delta_{\mathfrak{f}-1}^{|\bm{\pi}|-\mathfrak{y}}}e^{-\langle\Bar{\Vect{\beta}},V(\Vect{z})\rangle_{|\bm{\pi}|-\mathfrak{y}}}\prod_{\Bar{\bm{B}}\in\Bar{\bm{\pi}}}c_{\Bar{\bm{B}}}(\Vect{1}) \\ & = \prod_{\Bar{\bm{B}}\in\Bar{\bm{\pi}}}c_{\Bar{\bm{B}}}(\Vect{1})\underset{|\Vect{u}|=\Tilde{\Vect{p}}}{\sum_{\Vect{u}\in\Delta^{|\bm{\pi}|}}}e^{-\langle\Tilde{\Vect{\beta}},V(\Vect{u})\rangle_{|\bm{\pi}|}},
\end{align*}
where $\Tilde{\Vect{p}}=(p_1,\ldots,p_{\mathfrak{y}},\mathfrak{f}-1,\ldots,\mathfrak{f}-1)\in\ProdSet{(\N^*)}{|\bm{\pi}|}$ and $\Tilde{\bm{\beta}}=(1,\ldots,1,\Bar{\bm{B}}_1,\ldots,\Bar{\bm{B}}_{|\bm{\pi}|-\mathfrak{y}})\in\ProdSet{(\N^*)}{|\bm{\pi}|}$. One can notice that there exists $r_0>0$ such that 
\begin{align}\label{EspSmallPot}
    \Eb\Big[\underset{|\Vect{u}|=\Tilde{\Vect{p}}}{\sum_{\Vect{u}\in\Delta^{|\bm{\pi}|}}}e^{-\langle\Tilde{\Vect{\beta}},V(\Vect{u})\rangle_{|\bm{\pi}|}}\un_{\{\min_{|w|=\mathfrak{f}-1}V(w)<r_0(\mathfrak{f}-1)\}}\Big]\leq\mathfrak{C}_{\ref{CauchyCarre},2}e^{-(\mathfrak{f}-1)},
\end{align}
for some constant $\mathfrak{C}_{\ref{CauchyCarre},2}>0$. Indeed, By the Cauchy–Schwarz inequality,
\begin{align*}
    \Eb\Big[\Big(\underset{|\Vect{u}|=\Tilde{\Vect{p}}}{\sum_{\Vect{u}\in\Delta^{|\bm{\pi}|}}}e^{-\langle\Vect{1},V(\Vect{u})\rangle_{|\bm{\pi}|}}\un_{\{\min_{|w|=\mathfrak{f}-1}V(w)<r_0(\mathfrak{f}-1)\}}\Big]\leq&\Eb\Big[\underset{|\Vect{u}|=\Bar{\Vect{p}}}{\sum_{\Vect{u}\in\Delta^{|\bm{\pi}|}}}e^{-\langle\Vect{1},V(\Vect{u})\rangle_{|\bm{\pi}|}}\Big)^2\Big]^{1/2} \\ & \times\Pb\big(\min_{|w|=\mathfrak{f}-1}V(w)<r_0(\mathfrak{f}-1)\big)^{1/2},
\end{align*}
and thanks to Lemma \ref{HighDownfalls} \textit{\ref{HighDownfalls1}}, $\Eb[(\sum_{\Vect{u}\in\Delta^{|\bm{\pi}|;\; |\Vect{u}|=\Bar{\Vect{p}}}}e^{-\langle\Vect{1},V(\Vect{u})\rangle_{|\bm{\pi}|}})^2]\leq\mathfrak{C}_{\ref{HighDownfalls},1}$, where we recall that $\mathfrak{C}_{\ref{HighDownfalls},1}>0$ is a constant does not depending on $\Vect{p}$ (or $\Tilde{\Vect{p}}$) since $|\bm{\pi}|<q\leq 2k$. Moreover, since $\psi'(1)<0$, we can find $r_0>0$ and a constant $\mathfrak{C}_{\ref{CauchyCarre},3}>0$ such that $\Pb(\min_{|w|=\mathfrak{f}-1}V(w)<r_0(\mathfrak{f}-1))\leq\mathfrak{C}_{\ref{CauchyCarre},3}e^{2(\mathfrak{f}-1)}$. This yields \eqref{EspSmallPot}. \\
Now, note that, since $|\bm{\pi}|<q$, there is at least one block of the partition $\bm{\pi}$ with cardinal larger or equal to 2 so $\langle\Tilde{\Vect{\beta}},V(\Vect{z})\rangle_{|\bm{\pi}|}\geq\langle\Vect{1},V(\Vect{z})\rangle_{|\bm{\pi}|}+\min_{|w|=\mathfrak{f}-1}V(w)$ thus giving that the mean $\Eb[\sum_{\Vect{x}\in\Delta^q;\; |\Vect{x}|=\Vect{p}}\un_{\Upsilon_{\mathfrak{f}-1,\bm{\pi}}\cap\Upsilon_{\mathfrak{f},\bm{\eta}}}(\Vect{x})e^{-\langle\Vect{1},V(\Vect{x})\rangle_{q}}]$ is smaller than
\begin{align*}
    &\prod_{\Bar{\bm{B}}\in\Bar{\bm{\pi}}}c_{\Bar{\bm{B}}}(\Vect{1})\Big(\Eb\Big[\Big(\underset{|\Vect{u}|=\Tilde{\Vect{p}}}{\sum_{\Vect{u}\in\Delta^{|\bm{\pi}|}}}e^{-\langle\Tilde{\Vect{\beta}},V(\Vect{u})\rangle_{|\bm{\pi}|}}\un_{\{\min_{|w|=\mathfrak{f}-1}V(w)<r_0(\mathfrak{f}-1)\}}\Big] \\ & +\Eb\Big[e^{-\min_{|w|=\mathfrak{f}-1}V(w)}\underset{|\Vect{u}|=\Tilde{\Vect{p}}}{\sum_{\Vect{u}\in\Delta^{|\bm{\pi}|}}}e^{-\langle\Vect{1},V(\Vect{u})\rangle_{|\bm{\pi}|}}\un_{\{\min_{|w|=\mathfrak{f}-1}V(w)\geq r_0(\mathfrak{f}-1)\}}\Big]\Big),
\end{align*}
which, thanks to Lemma \ref{HighDownfalls} \textit{\ref{HighDownfalls1}} and \eqref{EspSmallPot}, is smaller than $\mathfrak{C}_{\ref{CauchyCarre},4}e^{-(1\land r_0)(\mathfrak{f}-1)}$ for some constant $\mathfrak{C}_{\ref{CauchyCarre},4}>0$. Finally
\begin{align*}
    \Eb\Big[\underset{|\Vect{x}|=\Vect{p}}{\sum_{\Vect{x}\in\Delta^q}}\un_{\{\mathcal{S}^q(\Vect{x})>m\}}e^{-\langle\Vect{1},V(\Vect{x})\rangle_{q}}\Big]\leq\mathfrak{C}_{\ref{CauchyCarre},5}\sum_{\mathfrak{f}=m+1}^{\max\Vect{p}}e^{-(1\land r_0)(\mathfrak{f}-1)}\leq\mathfrak{C}_{\ref{CauchyCarre},1}e^{-(1\land r_0)m},
\end{align*}
for some constant $\mathfrak{C}_{\ref{CauchyCarre},5}>0$ and \eqref{LargeGenerationCA} is proved.
\end{proof}

\noindent The convergence of the sequence of random variables $(\mathcal{A}^k_l(f))_l$ directly follows from Lemma \ref{CauchyCarre}. Indeed, let $f$ be a bounded function satisfying the hereditary Assumption \ref{Assumption3}. For any $l>l'>m>\mathfrak{g}$
\begin{align*}
    \mathcal{A}^k_l(f\un_{\mathcal{C}^k_m})=\sum_{x\in\Delta^k_l}f\un_{\mathcal{C}^k_m}(\Vect{x})e^{-\langle\Vect{1},V(\Vect{x})\rangle_{k}}=\sum_{u\in\Delta^k_{l'}}f\un_{\mathcal{C}^k_m}(\Vect{u})\underset{\Vect{x}>\Vect{u}}{\sum_{x\in\Delta^k_l}}e^{-\langle\Vect{1},V(\Vect{x})\rangle_{k}}, 
\end{align*}
so $\Eb[\mathcal{A}^k_l(f\un_{\mathcal{C}^k_m})|\mathcal{F}_{l'}]=\mathcal{A}^k_{l'}(f\un_{\mathcal{C}^k_m})$ where $\mathcal{F}_m=\sigma((x,V(x));|x|\leq m)$ and $(\mathcal{A}^k_{l}(f\un_{\mathcal{C}^k_m}))_{l>m}$ is a martingale bounded in $L^2(\Pb)$. In particular, for any integer $m>\mathfrak{g}$, $(\mathcal{A}^k_{l}(f\un_{\mathcal{C}^k_m}))_{l>m}$ converges in $L^2(\Pb^*)$ and $\Pb^*$-almost surely. Hence, thanks to Lemma \ref{CauchyCarre}, $(\mathcal{A}_l(f))_l$ is a Cauchy sequence in $L^2(\Pb^*)$ and therefore, $\mathcal{A}_{\infty}^k(f)$ exists.

\subsubsection{$k$-tuples in the set $\Delta^k\setminus\mathcal{C}^k_{a_n}$}
Before proving Proposition \ref{GENPROPCONV2}, let us show that the contribution of the $k$-tuples in the set $\mathfrak{E}^{k,\cdot}\cap\Delta^k\setminus\mathcal{C}^k_{a_n}$ is not significant. To do that, the following lemma provides an estimation for the quasi-independent version \eqref{DefIndepRange} of the range on the set $\mathcal{C}^k_{a_n}$:
\begin{lemm}\label{CoaRecente}
Let $\varepsilon_1\in(0,1)$, $k\geq 2$ and assume $\kappa>2k$. Under the Assumptions \ref{Assumption1}, \ref{Assumption2} and \ref{AssumptionIncrements}, there exist two constants $C_{\ref{CoaRecente}}>0$ and $\rho_8>0$ such that
\begin{align}
       \E^*\Big[\frac{1}{(\sqrtBis{n}\bm{L}_n)^k}\sum_{\bm{j}\in\llbracket 1,\mathfrak{s}_n\rrbracket_k}\mathcal{A}^{k,n}(\Vect{j},\un_{\Delta^k\setminus\mathcal{C}^k_{a_n}})\Big]\leq C_{\ref{CoaRecente}}n^{-\rho_8},
\end{align}
with $\mathfrak{s}_n=\sqrtBis{n}/\varepsilon_1$.
\end{lemm}

\begin{proof}
Recall that, thanks to the strong Markov property together with \eqref{ProbaTA}
\begin{align*}
    \E^{\mathcal{E}}\Big[\mathcal{A}^{k,n}(\Vect{j},\un_{\Delta^k\setminus\mathcal{C}^k_{a_n}})\Big]=\underset{\ell_n\leq|\Vect{x}|\leq\mathfrak{L}_n}{\sum_{\Vect{x}\in\Delta^k}}\un_{\Delta^k\setminus\mathcal{C}^k_{a_n}}(\Vect{x})\prod_{i=1}^k\frac{e^{-V(\VectCoord{x}{i})}}{H_{\VectCoord{x}{i}}},
\end{align*}
and since $H_{\VectCoord{x}{i}}\geq 1$
\begin{align*}
    \E^*\Big[\sum_{\bm{j}\in\llbracket 1,\mathfrak{s}_n\rrbracket_k}\mathcal{A}^{k,n}(\Vect{j},\un_{\Delta^k\setminus\mathcal{C}^k_{a_n}})\Big]&\leq(\mathfrak{s}_n)^k\sum_{\Vect{p}\in\ProdSet{\{\ell_n,\ldots,\mathfrak{L}_n\}}{k}}\Eb^*\Big[\underset{|\Vect{x}|=\Vect{p}}{\sum_{\Vect{x}\in\Delta^k}}\un_{\{\mathcal{S}^k(\Vect{x})>a_n\}}e^{-\langle\Vect{1},V(\Vect{x})\rangle_k}\Big] \\ & \leq(\mathfrak{s}_n\bm{L}_n)^k\mathfrak{C}_{\ref{CauchyCarre}}e^{-\mathfrak{b}a_n}=(\mathfrak{s}_n\bm{L}_n)^k\mathfrak{C}_{\ref{CauchyCarre}}n^{-\mathfrak{b}(2\delta_0)^{-1}},
\end{align*}
which ends the proof.
\end{proof}

\noindent We are now ready to prove Proposition \ref{GENPROPCONV2}.
\begin{proof}[Proof of Proposition \ref{GENPROPCONV2}]
We have to prove that for any $\varepsilon_1\sqrtBis{n}\leq s\leq\sqrtBis{n}/\varepsilon_1$, $\varepsilon_1\in(0,1)$
\begin{align}\label{ProbaCoaLoint}
    \P^*\Big(\Big|\frac{1}{(s\bm{L}_n)^k}\mathcal{A}^k(\mathcal{D}_{n,T^s},f\un_{\mathfrak{E}^{k,s}\cap\mathcal{C}^k_{a_n}})-(c_{\infty})^k\mathcal{A}^k_{\infty}(f)\Big|>\varepsilon/2\Big)\underset{n\to\infty}{\longrightarrow} 0.
\end{align}
We deduce from Lemma \ref{GenVariance} with $\mathfrak{a}=1$ that the range $\sum_{\bm{j}\in\llbracket 1,s\rrbracket_k}\mathcal{A}^{k,n}(\Vect{j},f\un_{\mathcal{C}_{a_n}^k})$ concentrates around its quenched mean. Indeed, for any $\varepsilon\in(0,1)$, by Markov inequality
\begin{align*}
    &\P\Big(\Big|\sum_{\bm{j}\in\llbracket 1,s\rrbracket_k}\mathcal{A}^{k,n}(\Vect{j},f\un_{\mathcal{C}_{a_n}^k})-\E^{\mathcal{E}}\big[\sum_{\bm{j}\in\llbracket 1,s\rrbracket_k}\mathcal{A}^{k,n}(\Vect{j},f\un_{\mathcal{C}_{a_n}^k})\big]\Big|>\varepsilon(s\bm{L}_n)^k/16\Big) \\ & \leq\frac{16^2}{\varepsilon^1(s\bm{L}_n)^{2k}}\E\Big[\Big(\sum_{\bm{j}\in\llbracket 1,s\rrbracket_k}\mathcal{A}^{k,n}(\Vect{j},f\un_{\mathcal{C}_{a_n}^k})-\E^{\mathcal{E}}\big[\sum_{\bm{j}\in\llbracket 1,s\rrbracket_k}\mathcal{A}^{k,n}(\Vect{j},f\un_{\mathcal{C}_{a_n}^k})\big]\Big)^2\Big] \\ & \leq 16^2\mathfrak{C}_{\ref{GenVariance}}\frac{\mathfrak{L}_n}{s}\leq\mathfrak{C}_{\ref{GENPROPCONV2}}\frac{\mathfrak{L}_n}{\sqrtBis{n}}\underset{n\to\infty}{\longrightarrow}0,
\end{align*}
where the last inequality comes from the fact that $\mathfrak{L}_n=o(\sqrtBis{n})$. Then, we know, thanks to Lemma \ref{LemmeSommeExc} with $g=f$, that $\mathcal{A}^k(\mathcal{D}_{n,T^s},f\un_{\mathfrak{E}^{k,s}\cap\mathcal{C}^k_{a_n}})$ behaves like its quasi-independent version $\sum_{\bm{j}\in\llbracket 1,s\rrbracket_k}\mathcal{A}^{k,n}(\Vect{j},f\un_{\mathcal{C}^k_{a_n}})$: for $n$ large enough
\begin{align*}
    \P^*\Big(\Big|\mathcal{A}^k(\mathcal{D}_{n,T^s},f\un_{\mathfrak{E}^{k,s}\cap\mathcal{C}^k_{a_n}})-\sum_{\bm{j}\in\llbracket 1,s\rrbracket_k}\mathcal{A}^{k,n}(\Vect{j},f\un_{\mathcal{C}^k_{a_n}})\Big|>\varepsilon (s\bm{L}_n)^k/16\Big)\leq n^{-\rho_4}, 
\end{align*}
hence
\begin{align*}
    \lim_{n\to\infty} \P^*\Big(\Big|\mathcal{A}^k(\mathcal{D}_{n,T^s},f\un_{\mathfrak{E}^{k,s}\cap\mathcal{C}^k_{a_n}})-\E^{\mathcal{E}}\big[\sum_{\bm{j}\in\llbracket 1,s\rrbracket_k}\mathcal{A}^{k,n}(\Vect{j},f\un_{\mathcal{C}^k_{a_n}})\big]\Big|>\varepsilon (s\bm{L}_n)^k/8\Big)=0.
\end{align*}
One can notice that $$\E^{\mathcal{E}}\Big[\sum_{\bm{j}\in\llbracket 1,s\rrbracket_k}\mathcal{A}^{k,n}(\Vect{j},f\un_{\mathcal{C}_{a_n}^k})\Big]=s(s-1)\cdots(s-k+1)\underset{\ell_n\leq|\Vect{x}|\leq\mathfrak{L}_n}{\sum_{\Vect{x}\in\Delta^k}}f\un_{\mathcal{C}_{a_n}^k}(\Vect{x})\prod_{i=1}^k\frac{e^{-V(\VectCoord{x}{i})}}{H_{\VectCoord{x}{i}}}.$$
Finally, Lemma \ref{ConvL1} yields
\begin{align*}
    \P^*\Big(\Big|\frac{1}{(s\bm{L}_n)^k}\mathcal{A}^k(\mathcal{D}_{n,T^s},f\un_{\mathfrak{E}^{k,s}\cap\mathcal{C}^k_{a_n}})-(c_{\infty})^k\mathcal{A}^k_{a_n}(f)\Big|>\varepsilon/4\Big)\underset{n\to\infty}{\longrightarrow} 0,
\end{align*}
and the result of the subsection \ref{QuasiMartingale} leads to the convergence in \eqref{ProbaCoaLoint}.

\noindent Now using Lemma \ref{CoaRecente}, we show that 
\begin{align}\label{ProbaCoaRecente}
    \P^*\Big(\frac{1}{(s\bm{L}_n)^k}\mathcal{A}^k(\mathcal{D}_{n,T^s},f\un_{\mathfrak{E}^{k,s}\cap\Delta^k\setminus\mathcal{C}^k_{a_n}})>\varepsilon/2\Big)\underset{n\to\infty}{\longrightarrow} 0.
\end{align}
Indeed
\begin{align*}
    \frac{1}{(s\bm{L}_n)^k}\mathcal{A}^k(\mathcal{D}_{n,T^s},f\un_{\mathfrak{E}^{k,s}\cap\Delta^k\setminus\mathcal{C}^k_{a_n}})&\leq\frac{1}{(s\bm{L}_n)^k}\sum_{\bm{j}\in\llbracket 1,s\rrbracket_k}\mathcal{A}^{k,n}(\Vect{j},f\un_{\Delta^k\setminus\mathcal{C}^k_{a_n}}) \\ & \leq\frac{1}{(\varepsilon_1\sqrtBis{n}\bm{L}_n)^k}\sum_{\bm{j}\in\llbracket 1,\sqrtBis{n}/\varepsilon_1\rrbracket_k}\mathcal{A}^{k,n}(\Vect{j},f\un_{\Delta^k\setminus\mathcal{C}^k_{a_n}}),
\end{align*}
so Markov inequality together with Lemma \ref{CoaRecente} leads to \eqref{ProbaCoaRecente}. We end the proof putting together \eqref{ProbaCoaLoint} and \eqref{ProbaCoaRecente}.
\end{proof}

\subsection{The range on $\Delta^k\setminus\mathfrak{E}^{k,\cdot}$}

Recall 
\begin{align*}
    \mathcal{A}^k(\mathcal{D}_{n,T^s},g)=\underset{\ell_n\leq|\Vect{x}|\leq\mathfrak{L}_n}{\sum_{\Vect{x}\in\Delta^k}}g(\Vect{x})\un_{\{T_{\Vect{x}}<T^s\}},
\end{align*}
where $T_{\Vect{x}}=\max_{1\leq i\leq k}T_{\VectCoord{x}{i}}$, $T_z=\min\{j\geq 0;\; X_j=z\}$, $T^0=0$ and $T^s=\min\{j>T^{s-1};\; X_j=e\}$ for $s\in\N^*$. Also recall that $(\ell_n)$ and $(\mathfrak{L}_n)$ are two sequences of positive integers such that $\delta_0^{-1}\log n\leq\ell_n\leq\mathfrak{L}_n\leq\sqrtBis{n}$. \\
The last step of our study is to show that the contribution of the $k$-tuples of vertices in small generations (see \eqref{SmallGenerations}) and such that at least two of these vertices are visited during the same excursion is not significant. This section is thus devoted to the proof of Proposition \ref{GENPROPCONV1}, claiming that
\begin{align*}
    \P^*\big(\sup_{s\leq\sqrtBis{n}/\varepsilon_1}\mathcal{A}^k(\mathcal{D}_{n,T^{s}},\un_{\Delta^k\setminus\mathfrak{E}^{k,s}})>\varepsilon(\sqrtBis{n}\bm{L}_n)^k\big)\underset{n\to\infty}{\longrightarrow} 0
\end{align*}
\begin{lemm}\label{LemmMultiExcur}
Let $\varepsilon_1\in(0,1)$, $k\geq 2$, let $\mathfrak{s}_n=\sqrtBis{n}/\varepsilon_1$ and assume $\kappa>2k$. Assume that the Assumptions \ref{Assumption1}, \ref{Assumption2}, \ref{AssumptionIncrements} hold and that $\mathfrak{L}_n=o(\sqrtBis{n})$.
\begin{enumerate}[label=(\roman*)]
    \item \label{SameExcu} If 
        \begin{align*}
            \mathfrak{E}^{k,s}_1:=\bigcup_{j=1}^s\bigcap_{i=1}^k\{\Vect{x}=(\VectCoord{x}{1},\ldots,\VectCoord{x}{k})\in\Delta^k;\; \mathcal{L}^{T^{j}}_{\VectCoord{x}{i}}-\mathcal{L}^{T^{j-1}}_{\VectCoord{x}{i}}\geq 1\}
        \end{align*}
        denotes the set of $k$-tuples of vertices visited during the same excursion before the instant $T^s$, then
        \begin{align*}
            \lim_{n\to\infty}\E^*\Big[\frac{1}{(\sqrtBis{n}\bm{L}_n)^k}\sup_{s\leq\mathfrak{s}_n}\mathcal{A}^k(\mathcal{D}_{n,T^{s}},\un_{\mathfrak{S}^{k,s}\cap\mathfrak{E}^{k,s}_1})\Big]=0.
        \end{align*}
    \item \label{MultiExcu} Let $\mathfrak{E}^{k,s}_2:=\Delta^k\setminus(\mathfrak{E}^{k,s}\cup\mathfrak{E}^{k,s}_1)$. If $k\geq 3$ and the Assumption \ref{AssumptionSmallGenerations} hold, then, for all $B>0$
    \begin{align*}
            \lim_{n\to\infty}\E^*\Big[\frac{1}{(\sqrtBis{n}\bm{L}_n)^k}\sup_{s\leq\mathfrak{s}_n}\mathcal{A}^k(\mathcal{D}_{n,T^{s}},\un_{\mathfrak{S}^{k,s}\cap\mathfrak{E}^{k,s}_2}\un_{\{\underline{V}(\cdot)\geq-B\}})\Big]=0,
        \end{align*}
        with $\underline{V}(\Vect{x})\geq-B$ if and only if $\underline{V}(\VectCoord{x}{i})\geq-B$ for all $i\in\{1,\ldots,k\}$. 
\end{enumerate}
\end{lemm}
\begin{proof}
In order to avoid unnecessary technical difficulties, we prove it for any $\kappa>4$. Let us start with the proof of \textit{\ref{SameExcu}}. By definition, $\Vect{x}\in\mathfrak{S}^{k,s}\cap\mathfrak{E}^{k,s}_1$ if and only if there exists $j\in\{1,\ldots,s\}$ such that for all $1\leq i\leq k$, $N^{T^j}_{\VectCoord{x}{i}}-N^{T^{j-1}}_{\VectCoord{x}{i}}\geq 1$ and for all $p\not= j$, $N^{T^p}_{\VectCoord{x}{i}}-N^{T^{p-1}}_{\VectCoord{x}{i}}=0$. Thus, using again the strong Markov property
\begin{align*}
    \E\Big[\sup_{s\leq\mathfrak{s}_n}\mathcal{A}^k(\mathcal{D}_{n,T^{s}},\un_{\mathfrak{S}^{k,s}\cap\mathfrak{E}^{k,s}_1})\Big]&=\Eb\Big[\sup_{s\leq\mathfrak{s}_n}\sum_{j=1}^{s}\underset{\ell_n\leq|\Vect{x}|\leq\mathfrak{L}_n}{\sum_{\Vect{x}\in\Delta^k}}\un_{\cap_{i=1}^k\cap_{p\not=j}\{N^{T^{j}}_{\VectCoord{x}{i}}-N^{T^{j-1}}_{\VectCoord{x}{i}}\geq 1,N^{T^{p}}_{\VectCoord{x}{i}}-N^{T^{p-1}}_{\VectCoord{x}{i}}=0\}}\Big] \\ & \leq\sum_{j=1}^{\mathfrak{s}_n}\E\Big[\underset{\ell_n\leq|\Vect{x}|\leq\mathfrak{L}_n}{\sum_{\Vect{x}\in\Delta^k}}\prod_{i=1}^k\un_{\{N^{T^{j}}_{\VectCoord{x}{i}}-N^{T^{j-1}}_{\VectCoord{x}{i}}\geq 1\}}\Big] \\ & \leq\mathfrak{s_n}\sum_{\Vect{p}\in\ProdSet{\{\ell_n,\ldots,\mathfrak{L}_n\}}{k}}\Eb\Big[\underset{|\Vect{x}|=\Vect{p}}{\sum_{\Vect{x}\in\Delta^k}}\P^{\mathcal{E}}(T_{\Vect{x}}<T^1)\Big] \\ & \leq\mathfrak{C}_{\ref{GENPROPCONV1},1}\mathfrak{s}_n(\bm{L}_n)^k(\mathfrak{L}_n)^{k-1},
\end{align*}
where we have used Lemma \ref{EspProbaTAPartition} \eqref{EqLemmaProbaTA} with $m=\mathfrak{L}_n$ for the last inequality, recalling that the constant $\mathfrak{C}'_{\ref{GENPROPCONV1},1}>0$ does not depend on $\Vect{p}$. By definition of $\mathfrak{s}_n$
\begin{align*}
    \E\Big[\frac{1}{(\sqrtBis{n}\bm{L}_n)^k}\sup_{s\leq\mathfrak{s}_n}\mathcal{A}^k(\mathcal{D}_{n,T^{s}},\un_{\mathfrak{S}^{k,s}\cap\mathfrak{E}^{k,s}_1})\Big]\leq\frac{\mathfrak{C}_{\ref{GENPROPCONV1},1}}{\varepsilon_1}\Big(\frac{\mathfrak{L}_n}{\sqrtBis{n}}\Big)^{k-1},
\end{align*}
which goes to $0$ when $n$ goes to $\infty$ since $\mathfrak{L}_n=o(\sqrtBis{n})$ and this yields \textit{\ref{SameExcu}}.

\vspace{0.2cm}

\noindent We now focus on \textit{\ref{MultiExcu}}. Since $k\geq 3$, $\mathfrak{E}^{k,s}_2$ is nothing but the set of $k$-tuples in $\Delta^k$ of vertices neither visited during $k$ distinct excursions, nor during the same excursion. Therefore, there exists $\mathfrak{e}\in\{2,\ldots,k-1\}$ and $\mathfrak{e}$ disjoint subsets $I_1,\ldots,I_{\mathfrak{e}}$ of $\{1,\ldots,k\}$ such that $\{1,\ldots,k\}=I_1\cup\cdots\cup I_{\mathfrak{e}}$ and for any $j\in\{1,\ldots,\mathfrak{e}\}$, $i,i'\in I_j$ if and only if $\VectCoord{x}{i}$ and $\VectCoord{x}{i'}$ are visited during the same excursion before the instant $T^s$:
\begin{align*}
    \exists\; \mathfrak{j}\in\{1,\ldots,s\}:\; \big(\mathcal{L}^{T^{\mathfrak{j}}}_{\VectCoord{x}{i}}-\mathcal{L}^{T^{\mathfrak{j}-1}}_{\VectCoord{x}{i}}\big)\land\big(\mathcal{L}^{T^{\mathfrak{j}}}_{\VectCoord{x}{i'}}-\mathcal{L}^{T^{\mathfrak{j}-1}}_{\VectCoord{x}{i'}}\big)\geq 1.
\end{align*}
Let $m\in\N^*$ and introduce the following subset of $\Delta^k$
\begin{align*}
    \Upsilon^{k,s}_m:=\{\Vect{x}=(\VectCoord{x}{1},\ldots,\VectCoord{x}{k})\in\Delta^k;\;\forall j\not=j'\in\{1,\ldots,\mathfrak{e}\}, \forall\; i\in I_j,\forall\; i'\in I_{j'}:\; |\VectCoord{x}{i}\land\VectCoord{x}{i'}|<m\},
\end{align*}
where we recall that $u\land v$ is the most recent common ancestor (MRCA) of $u$ and $v$. $\Upsilon^k_m$ is the set of $k$-tuples of vertices such that the MRCA of two vertices visited during two distinct excursions before the instant $T^s$ has to be in a generation smaller than $m$. Note that the MRCA of two vertices visited during the same excursion can be in a generation larger or equal to $m$. \\
Recall that $(\Lambda_l)_{l\in\N}$ is the sequence of functions such that for all $t>0$, $\Lambda_0(t)=t$ and for any $l\in\{1,\ldots,l_0\}$, $\Lambda_{l-1}(t)=e^{\Lambda_l(t)}$ (see the Assumption \ref{AssumptionSmallGenerations}). Introduce $\mathfrak{g}_{l,n}:=4k\delta_0^{-1}\Lambda_l(\mathfrak{L}_n)$. Note that $\mathfrak{g}_{0,n}>\mathfrak{L}_n$ so
\begin{align*}
    \E\Big[\sup_{s\leq\mathfrak{s}_n}\mathcal{A}^k(\mathcal{D}_{n,T^{s}},\un_{\mathfrak{S}^{k,s}\cap\mathfrak{E}^{k,s}_2}\un_{\{\underline{V}(\cdot)\geq-B\}})\Big]=\E\Big[\sup_{s\leq\mathfrak{s}_n}\mathcal{A}^k(\mathcal{D}_{n,T^{s}},\un_{\mathfrak{S}^{k,s}\cap\mathfrak{E}^{k,s}_2}\un_{\{\underline{V}(\cdot)\geq-B\}\cap\mathcal{C}^k_{\mathfrak{g}_{0,n}}})\Big].
\end{align*}
Recall that for any $\Vect{x}=(\VectCoord{x}{1},\ldots,\VectCoord{x}{k})\in\Delta^k$, it belongs to $\mathfrak{S}^{k,s}$ if and only if $\VectCoord{x}{i}$ is visited during a single excursion before the instant $T^{s}$ for all $i\in\{1,\ldots,k\}$. Using what we previously said, we have, for any $s\leq\mathfrak{s}_n$
\begin{align*}
    \un_{\mathfrak{S}^{k,s}\cap\mathfrak{E}_2^{k,s}}(\Vect{x})&\leq\sum_{\mathfrak{e}=2}^{k-1}\;\sum_{\Vect{j}\in\llbracket 1,s\rrbracket_{\mathfrak{e}}}\underset{\cup_{\mathfrak{l}=1}^{\mathfrak{e}}I_{\mathfrak{l}}=\{1,\ldots,k\}}{\sum_{I_1,\ldots,I_{\mathfrak{e}}\textrm{ sets}}}\prod_{p=1}^{\mathfrak{e}}Y_p \\ & \leq\sum_{\mathfrak{e}=2}^{k-1}\;\sum_{\Vect{j}\in\llbracket 1,\mathfrak{s}_n\rrbracket_{\mathfrak{e}}}\;\underset{\cup_{\mathfrak{l}=1}^{\mathfrak{e}}I_{\mathfrak{l}}=\{1,\ldots,k\}}{\sum_{I_1,\ldots,I_{\mathfrak{e}}\textrm{ sets}}}\prod_{p=1}^{\mathfrak{e}}Y_p,
\end{align*}
where, for any $p\in\{1,\ldots,\mathfrak{e}\}$, $Y_p:=\un_{\cap_{i\in I_p}\{\mathcal{L}^{T^{j_p}}_{\VectCoord{x}{i}}-\mathcal{L}^{T^{j_p-1}}_{\VectCoord{x}{i}}\geq 1\}}$. It follows that 
\begin{align}\label{EspSup}
    &\E\big[\sup_{s\leq\mathfrak{s}_n}\mathcal{A}^k(\mathcal{D}_{n,T^{s}},\un_{\mathfrak{S}^{k,s}\cap\mathfrak{E}^{k,s}_2}\un_{\{\underline{V}(\cdot)\geq-B\}\cap\mathcal{C}^k_{\mathfrak{g}_{0,n}}})\big] \nonumber \\ & \leq\sum_{\Vect{p}\in\ProdSet{\{\ell,\ldots,\mathfrak{L}_n\}}{k}}\sum_{l=1}^{l_0}\sum_{\mathfrak{e}=2}^{k-1}\;\sum_{\Vect{j}\in\llbracket 1,\mathfrak{s}_n\rrbracket_{\mathfrak{e}}}\;\underset{\cup_{\mathfrak{l}=1}^{\mathfrak{e}}I_{\mathfrak{l}}=\{1,\ldots,k\}}{\sum_{I_1,\ldots,I_{\mathfrak{e}}\textrm{ sets}}}\E\Big[\underset{|\Vect{x}|=\Vect{p}}{\sum_{\Vect{x}\in\Delta^k}}\un_{\{\underline{V}(\Vect{x})\geq-B\}}\un_{\Upsilon^{k,\mathfrak{s}_n}_{\mathfrak{g}_{l-1,n}}\setminus\Upsilon^{k,\mathfrak{s}_n}_{\mathfrak{g}_{l,n}}}(\Vect{x})\prod_{p=1}^{\mathfrak{e}}Y_p\Big] \nonumber \\ & + \sum_{\Vect{p}\in\ProdSet{\{\ell,\ldots,\mathfrak{L}_n\}}{k}}\sum_{\mathfrak{e}=2}^{k-1}\;\sum_{\Vect{j}\in\llbracket 1,\mathfrak{s}_n\rrbracket_{\mathfrak{e}}}\;\underset{\cup_{\mathfrak{l}=1}^{\mathfrak{e}}I_{\mathfrak{l}}=\{1,\ldots,k\}}{\sum_{I_1,\ldots,I_{\mathfrak{e}}\textrm{ sets}}}\E\Big[\underset{|\Vect{x}|=\Vect{p}}{\sum_{\Vect{x}\in\Delta^k}}\un_{\{\underline{V}(\Vect{x})\geq-B\}}\un_{\Upsilon^{k,\mathfrak{s}_n}_{\mathfrak{g}_{l_0,n}}}(\Vect{x})\prod_{p=1}^{\mathfrak{e}}Y_p\Big].
\end{align}
First, let us prove that for any $\Vect{p}\in\ProdSet{\{\ell_n,\ldots,\mathfrak{L}_n\}}{k}$,
\begin{align}\label{MultiExcuControl}
    \E\Big[\underset{|\Vect{x}|=\Vect{p}}{\sum_{\Vect{x}\in\Delta^k}}\un_{\{\underline{V}(\Vect{x})\geq-B\}}\un_{\Upsilon^{k,\mathfrak{s}_n}_{\mathfrak{g}_{l-1,n}}\setminus\Upsilon^{k,\mathfrak{s}_n}_{\mathfrak{g}_{l,n}}}(\Vect{x})\prod_{p=1}^{\mathfrak{e}}Y_p\Big]\leq\mathfrak{C}_{\ref{LemmMultiExcur},1}(\mathfrak{L}_n)^{k-\mathfrak{e}}.
\end{align}
\noindent The proof of \eqref{MultiExcuControl} is quite technical so in order to keep it as clear as possible, as one can notice in the proof of Lemmas \ref{HighDownfalls} \textit{\ref{HighDownfalls1}} and \ref{EspProbaTAPartition} \eqref{EqLemmaProbaTA} with $m=\mathfrak{L}_n$, we can and shall restrict to the case $\Vect{p}=(m,\ldots,m)\in\ProdSet{\{\ell_n,\ldots,\mathfrak{L}_n\}}{k}$.

\vspace{0.2cm}

\noindent Thanks to the strong Markov property, the random variables $Y_1,\ldots,Y_{\mathfrak{e}}$ are i.i.d under $\P^{\mathcal{E}}$ and
\begin{align*}
    &\E\Big[\underset{|\Vect{x}|=\Vect{p}}{\sum_{\Vect{x}\in\Delta^k}}\un_{\{\underline{V}(\Vect{x})\geq-B\}}\un_{\Upsilon^{k,\mathfrak{s}_n}_{\mathfrak{g}_{l-1,n}}\setminus\Upsilon^{k,\mathfrak{s}_n}_{\mathfrak{g}_{l,n}}}(\Vect{x})\prod_{p=1}^{\mathfrak{e}}Y_p\Big] \\ & =\sum_{\Vect{x}\in\Delta^k_m}\un_{\{\underline{V}(\Vect{x})\geq-B\}}\un_{\Upsilon^{k,\mathfrak{s}_n}_{\mathfrak{g}_{l-1,n}}\setminus\Upsilon^{k,\mathfrak{s}_n}_{\mathfrak{g}_{l,n}}}(\Vect{x})\prod_{p=1}^{\mathfrak{e}}\P^{\mathcal{E}}\big(\max_{i\in I_p}T_{\VectCoord{x}{i}}<T^1\big).
\end{align*}
As usual, $\sum_{\Vect{x}\in\Delta^k_m}\un_{\{\underline{V}(\Vect{x})\geq-B\}}\un_{\Upsilon^{k,\mathfrak{s}_n}_{\mathfrak{g}_{l-1,n}}\setminus\Upsilon^{k,\mathfrak{s}_n}_{\mathfrak{g}_{l,n}}}(\Vect{x})\prod_{p=1}^{\mathfrak{e}}\P^{\mathcal{E}}(\max_{i\in I_p}T_{\VectCoord{x}{i}}<T^1)$ is equal to
\begin{align}\label{DecompoMultiExcu}
    \sum_{\ell=1}^{k-1}\;\sum_{\Pi\textrm{ increasing}}\;\sum_{\bm{t};t_1<\ldots<t_{\ell}<m}\;\sum_{\Vect{x}\in\Delta^k_m}\un_{\{\underline{V}(\Vect{x})\geq-B\}}f^{\ell}_{\bm{t},\Pi}\un_{\Upsilon^{k,\mathfrak{s}_n}_{\mathfrak{g}_{l-1,n}}\setminus\Upsilon^{k,\mathfrak{s}_n}_{\mathfrak{g}_{l,n}}}(\Vect{x})\prod_{p=1}^{\mathfrak{e}}\P^{\mathcal{E}}\big(\max_{i\in I_p}T_{\VectCoord{x}{i}}<T^1\big),
\end{align}
where the genealogical tree function $f^{\ell}_{\Vect{t},\Pi}$ is defined in \eqref{Def_f_partition}. Recall that $t_1-1,\ldots,t_{\ell}-1$ correspond to the consecutive coalescent/split times. We then define
\begin{align*}
    \tau^{\ell}:=\max\{\mathfrak{j}\in\{1,\ldots,\ell\};\; \exists\;p\not=p'\in\{2,\ldots,\mathfrak{e}\},\; \exists\;\bm{B}\in\pi_{\mathfrak{j}-1}:\bm{B}\cap I_p\not=\varnothing\textrm{ and }\bm{B}\cap I_{p'}\not=\varnothing\},
\end{align*}
and the $\Vect{x}$-version $\tau^{\ell}(\Vect{x})$ of $\tau^{\ell}$:

\begin{align*}
    \tau^{\ell}(\Vect{x}):=\max\{\mathfrak{j}\in\{1,\ldots,\ell\};\; \exists\;p\not=p'\in\{2,\ldots,\mathfrak{e}\},\; \exists\; i\in I_p,\; i'\in I_{p'}:\; |\VectCoord{x}{i}\land\VectCoord{x}{i'}|=t_{\mathfrak{j}}-1\}.
\end{align*}
In other words, if the genealogical tree of $\Vect{x}\in\Delta^k$ is given by $f^{\ell}_{\Vect{t},\Pi}$, then $\tau^{\ell}=\tau^{\ell}(\Vect{x})$ and $t_{\tau^{\ell}}-1$ is the last generation at which two or more vertices visited during two distinct excursions share a common ancestor.

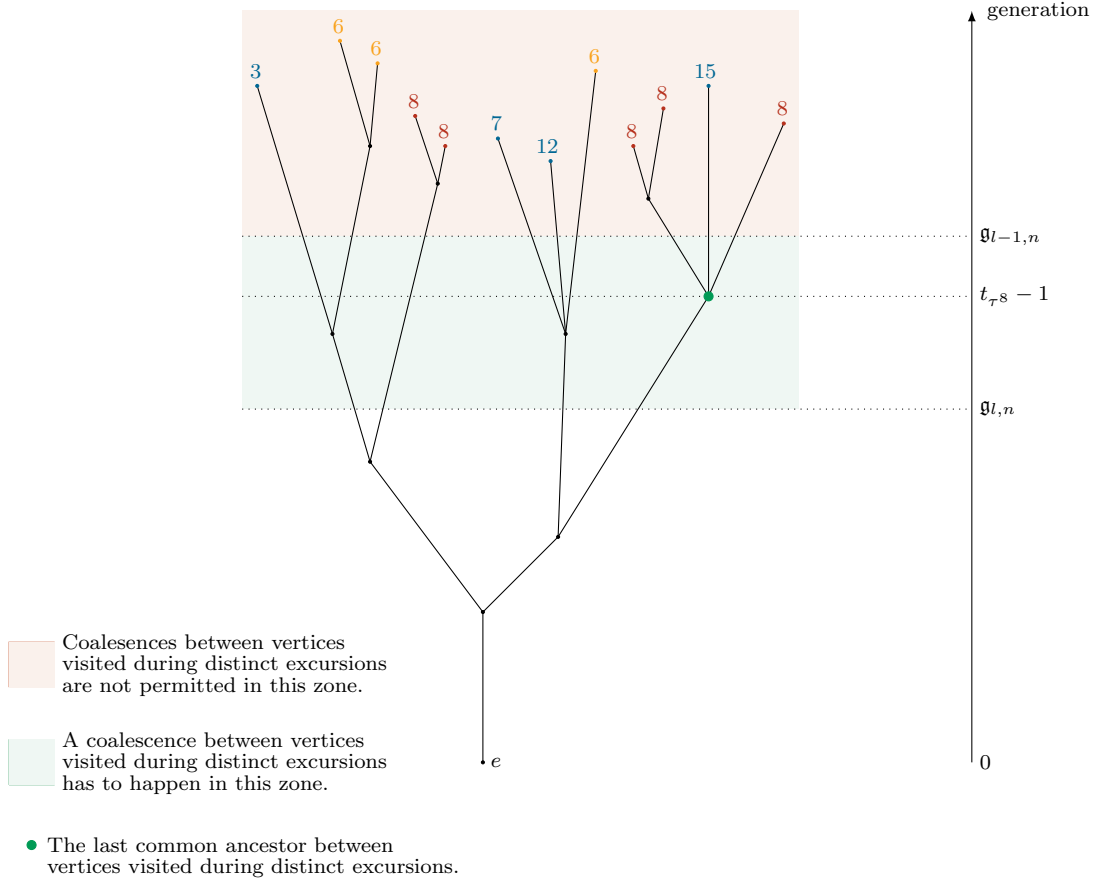
\vspace{0.5cm}

\begin{figure*}[h!]
\centering
\begin{tikzpicture}

\fill [BrickRed!5] (-4.2,7)--(3.2,7)--(3.2,10)--(-4.2,10)--(-4.2,7);
\fill [ForestGreen!5] (-4.2,7)--(3.2,7)--(3.2,4.7)--(-4.2,4.7)--(-4.2,7);

%%%%%%%%%% Gen 0 %%%%%%%%%%

\draw[black,fill] (-1,0) circle [radius=0.02];
\filldraw (-1,0) node[anchor=west] {$e$};

%%%%%%%%%% Gen 1 %%%%%%%%%%

\draw (-1,0) -- (-1,2);
\draw[black,fill] (-1,2) circle [radius=0.02];

%%%%%%%%%% Gen 2 %%%%%%%%%%

\draw (-1,2) -- (0,3);
\draw[black,fill] (0,3) circle [radius=0.02];

%%%%%%%%%% Gen 3 %%%%%%%%%%

\draw (-1,2) -- (-2.5,4);
\draw[black,fill] (-2.5,4) circle [radius=0.02];

%%%%%%%%%% Gen 4 %%%%%%%%%%

\draw (-2.5,4) -- (-3,5.7);
\draw[black,fill] (-3,5.7) circle [radius=0.02];

%%%%%%%%%% Gen 5 %%%%%%%%%%

\draw (0,3) -- (2,6.2);
%\draw[black,fill] (2,6.2) circle [radius=0.02];

%%%%%%%%%% Gen 6 %%%%%%%%%%

\draw (2,6.2) -- (1.2,7.5);
\draw[black,fill] (1.2,7.5) circle [radius=0.02];

%%%%%%%%%% Gen 7 %%%%%%%%%%

\draw (-3,5.7) -- (-2.5,8.2);
\draw[black,fill] (-2.5,8.2) circle [radius=0.02];

%%%%%%%%%% Gen + %%%%%%%%%%

\draw (-2.5,4) -- (-1.6,7.7);
\draw[black,fill] (-1.6,7.7) circle [radius=0.02];

\draw (0,3) -- (0.1,5.7);
\draw[black,fill] (0.1,5.7) circle [radius=0.02];

%%%%%%%%%% sommets x (de gauche à droite) %%%%%%%%%%

\draw (-3,5.7) -- (-4,9);
\draw[MidnightBlue,fill] (-4,9) circle [radius=0.02];
\filldraw (-4.2,9.2) node[anchor=west] {\textcolor{MidnightBlue}{$3$}};

\draw (-2.5,8.2) -- (-2.9,9.6);
\draw[YellowOrange,fill] (-2.9,9.6) circle [radius=0.02];
\filldraw (-3.1,9.8) node[anchor=west] {\textcolor{YellowOrange}{$6$}};

\draw (-2.5,8.2) -- (-2.4,9.3);
\draw[YellowOrange,fill] (-2.4,9.3) circle [radius=0.02];
\filldraw (-2.6,9.5) node[anchor=west] {\textcolor{YellowOrange}{$6$}};

\draw (-1.6,7.7) -- (-1.9,8.6);
\draw[BrickRed,fill] (-1.9,8.6) circle [radius=0.02];
\filldraw (-2.1,8.8) node[anchor=west] {\textcolor{BrickRed}{$8$}};

\draw (-1.6,7.7) -- (-1.5,8.2);
\draw[BrickRed,fill] (-1.5,8.2) circle [radius=0.02];
\filldraw (-1.7,8.4) node[anchor=west] {\textcolor{BrickRed}{$8$}};

\draw (0.1,5.7) -- (-0.8,8.3);
\draw[MidnightBlue,fill] (-0.8,8.3) circle [radius=0.02];
\filldraw (-1,8.5) node[anchor=west] {\textcolor{MidnightBlue}{$7$}};

\draw (0.1,5.7) -- (-0.1,8);
\draw[MidnightBlue,fill] (-0.1,8) circle [radius=0.02];
\filldraw (-0.4,8.2) node[anchor=west] {\textcolor{MidnightBlue}{$12$}};

\draw (0.1,5.7) -- (0.5,9.2);
\draw[YellowOrange,fill] (0.5,9.2) circle [radius=0.02];
\filldraw (0.3,9.4) node[anchor=west] {\textcolor{YellowOrange}{$6$}};

\draw (1.2,7.5) -- (1,8.2);
\draw[BrickRed,fill] (1,8.2) circle [radius=0.02];
\filldraw (0.8,8.4) node[anchor=west] {\textcolor{BrickRed}{$8$}};

\draw (1.2,7.5) -- (1.4,8.7);
\draw[BrickRed,fill] (1.4,8.7) circle [radius=0.02];
\filldraw (1.2,8.9) node[anchor=west] {\textcolor{BrickRed}{$8$}};

\draw (2,6.2) -- (2,9);
\draw[MidnightBlue,fill] (2,9) circle [radius=0.02];
\filldraw (1.7,9.2) node[anchor=west] {\textcolor{MidnightBlue}{$15$}};

\draw (2,6.2) -- (3,8.5);
\draw[BrickRed,fill] (3,8.5) circle [radius=0.02];
\filldraw (2.8,8.7) node[anchor=west] {\textcolor{BrickRed}{$8$}};

%%%%%%%%%% Générations %%%%%%%%%%

\draw[latex-] (5.5,10) -- (5.5,0);
\filldraw (5.5,10) node[anchor=west] {\textrm{ generation }};

\filldraw (5.5,0) node[anchor=west] {$0$};

\draw[dotted] (-4.2,4.7) -- (5.5,4.7);
\filldraw (5.5,4.7) node[anchor=west] {$\mathfrak{g}_{l,n}$};

\draw[dotted] (-4.2,7) -- (5.5,7);
\filldraw (5.5,7) node[anchor=west] {$\mathfrak{g}_{l-1,n}$};

\draw[dotted] (-4.2,6.2) -- (5.5,6.2);
\filldraw (5.5,6.2) node[anchor=west] {$t_{\tau^8}-1$};

%Last ancestor
\draw[ForestGreen,fill] (2,6.2) circle [radius=0.06];

%Legend

\draw[BrickRed!50] (-7.3,1)--(-6.7,1)--(-6.7,1.6)--(-7.3,1.6)--(-7.3,1);
\fill [BrickRed!5] (-7.3,1)--(-6.7,1)--(-6.7,1.6)--(-7.3,1.6)--(-7.3,1);
\filldraw (-6.7,1.6) node[anchor=west] {coalescences  between vertices};
\filldraw (-6.7,1.3) node[anchor=west] {visited during distinct excursions};
\filldraw (-6.7,1) node[anchor=west] {are not permitted in this zone.};

\draw[ForestGreen!50] (-7.3,-0.3)--(-6.7,-0.3)--(-6.7,0.3)--(-7.3,0.3)--(-7.3,-0.3);
\fill [ForestGreen!5] (-7.3,-0.3)--(-6.7,-0.3)--(-6.7,0.3)--(-7.3,0.3)--(-7.3,-0.3);
\filldraw (-6.7,0.3) node[anchor=west] {A coalescence between vertices};
\filldraw (-6.7,0) node[anchor=west] {visited during distinct excursions};
\filldraw (-6.7,-0.3) node[anchor=west] {has to happen in this zone.};

\draw[ForestGreen,fill] (-7,-1.1) circle [radius=0.06];
\filldraw (-6.9,-1.1) node[anchor=west] {The last common ancestor between};
\filldraw (-6.9,-1.4) node[anchor=west] {vertices visited during distinct excursions.};
\end{tikzpicture}
\end{figure*}

\begin{figure}
    \centering
    \caption{An example of a $12$-tuple belonging to $\Upsilon^{12,\cdot}_{\mathfrak{g}_{l-1,n}}\setminus\Upsilon^{12,\cdot}_{\mathfrak{g}_{l,n}}$ whose genealogical tree is given by $f^{\ell}_{\Vect{t},\Pi}$. \textcolor{YellowOrange}{$6$} means that the corresponding \textcolor{YellowOrange}{vertex} is visited during the \textcolor{YellowOrange}{$6$}-th excursion above $e^*$. In the present example, $\ell=8$ and $\tau^{8}=4$.}
\end{figure}

\newpage

\noindent By definition of $\tau^{\ell}$, for all $\mathfrak{j}\geq\tau^{\ell}$, if $\bm{B}\in\pi_{\mathfrak{j}}$, then $\bm{B}$ is necessarily a subset of $I_{p'}$ for some $p'\in\{1,\ldots,\mathfrak{e}\}$. In other words, each coalescence that occurs between $t_{\tau^{\ell}+1}$ and $t_{\ell}$ involves exclusively two or more vertices visited during the same excursion. As a consequence, for any $i\in\{\tau^{\ell},\ldots,\ell\}$ and $p\in\{1,\ldots,\mathfrak{e}\}$, we can defined the set $I_p^{i}$ as follows: we first set $I_p^{\ell}:=I_p$ so $I_1^{\ell},\ldots,I_{\mathfrak{e}}^{\ell}$ form a partition of $\{1,\ldots,k\}$. As we said before, by definition of $\tau^{\ell}$, coalescences can only happen between two or more vertices which indexes belong to the same $I_p^{\ell}$. Thus, for any $p\in\{1,\ldots,\mathfrak{e}\}$, there exists an integer $\mathfrak{e}_p^{\ell-1}\geq 1$ and $\mathfrak{e}_p^{\ell-1}$ distinct integers $k^{\ell-1}_{p,1},\ldots,k^{\ell-1}_{p,\mathfrak{e}_p^{\ell-1}}$ in $\{1,\ldots,|\bm{\pi}_{\ell-1}|\}$ such that for any $j\in\{k^{\ell-1}_{p,1},\ldots,k^{\ell-1}_{p,\mathfrak{e}_p^{\ell-1}}\}$, the block $\bm{B}_j^{\ell-1}$ of the partition $|\bm{\pi}_{\ell-1}|$ is the union of $b_{\ell-1}(\bm{B}_j)$ block(s) of the partition $\bm{\pi}_{\ell}$ of elements of $F^{\ell}_p$. We set $F^{\ell-1}_p:=\{k^{\ell-1}_{p,1},\ldots,k^{\ell-1}_{p,\mathfrak{e}_p^{\ell-1}}\}$ so $I_1^{\ell-1},\ldots,I_{\mathfrak{e}}^{\ell-1}$ form a partition of $\{1,\ldots,|\bm{\pi}_{\ell-1}|\}$. Now, let $i\in\{\tau^{\ell}+1,\ldots,\ell\}$ and assume that $F^i_p$ has been built. By definition of $\tau^{\ell}$, for any $p\in\{1,\ldots,\mathfrak{e}\}$, there exists an integer $\mathfrak{e}_p^{i-1}\geq 1$ and $\mathfrak{e}_p^{i-1}$ distinct integer $k^{i-1}_{p,1},\ldots,k^{i-1}_{p,\mathfrak{e}_p^{i-1}}$ in $\{1,\ldots,|\bm{\pi}_{i-1}|\}$ such that for any $j\in\{k^{i-1}_{p,1},\ldots,k^{i-1}_{p,\mathfrak{e}_p^{i-1}}\}$, the block $\bm{B}_j^{i-1}$ of the partition $|\bm{\pi}_{i-1}|$ is the union of $b_{i-1}(\bm{B}_j)$ block(s) of the partition $\bm{\pi}_{i}$ of elements of $I^{i}_p$. We set $I^{i-1}_p:=\{k^{i-1}_{p,1},\ldots,k^{\ell-1}_{p,\mathfrak{e}_p^{i-1}}\}$ so $I_1^{i-1},\ldots,I_{\mathfrak{e}}^{i-1}$ form a partition of $\{1,\ldots,|\bm{\pi}_{i-1}|\}$. \\
Hence, noticing that 
$$ f^{\ell}_{\bm{t},\Pi}\un_{\Upsilon^{k,\mathfrak{s}_n}_{\mathfrak{g}_{l-1,n}}\setminus\Upsilon^{k,\mathfrak{s}_n}_{\mathfrak{g}_{l,n}}}(\Vect{x})\leq f^{\ell}_{\bm{t},\Pi}(\Vect{x})\un_{\{\mathfrak{g}_{l,n}\leq t_{\tau^{\ell}(\Vect{x})}-1<\mathfrak{g}_{l-1,n}\}}=f^{\ell}_{\bm{t},\Pi}(\Vect{x})\un_{\{\mathfrak{g}_{l,n}\leq t_{\tau^{\ell}}-1<\mathfrak{g}_{l-1,n}\}}, $$
it is enough to show \eqref{MultiExcuControl} for $\mathfrak{g}_{l,n}\leq t_{\tau^{\ell}}-1<\mathfrak{g}_{l-1,n}$. We then have
\begin{align*}
    &\Eb\Big[\sum_{\Vect{x}\in\Delta^k_m}\un_{\{\underline{V}(\Vect{x})\geq-B\}}f^{\ell}_{\bm{t},\Pi}\un_{\Upsilon^{k,\mathfrak{s}_n}_{\mathfrak{g}_{l-1,n}}\setminus\Upsilon^{k,\mathfrak{s}_n}_{\mathfrak{g}_{l,n}}}(\Vect{x})\prod_{p=1}^{\mathfrak{e}}\P^{\mathcal{E}}\big(\max_{i\in I_p}T_{\VectCoord{x}{i}}<T^1\big)\big|\mathcal{F}_{t_{\tau^{\ell}}}\Big] \\ & \leq\mathfrak{C}_{\ref{LemmMultiExcur},2}\sum_{\Vect{u}\in\Delta^{|\bm{\pi}_{\tau^{\ell}}|}_{t_{\tau^{\ell}}}}\un_{\{\underline{V}(\Vect{u})\geq-B\}}f^{\tau^{\ell}}_{\bm{t}^{\tau^{\ell}},\Pi^{\tau^{\ell}}}(\Vect{u})\prod_{p=1}^{\mathfrak{e}}\P^{\mathcal{E}}\big(\max_{i\in I^{\tau^{\ell}}_p}T_{\VectCoord{u}{i}}<T^1\big)\prod_{j=1}^{|\bm{\pi}_{\tau^{\ell}}|}(H_{\VectCoord{u}{j}})^{|\bm{B}^{\tau^{\ell}}_j|},
\end{align*}
for some constant $\mathfrak{C}_{\ref{LemmMultiExcur},2}>0$ where $\bm{t}^{\tau^{\ell}}$ and $\Pi^{\tau^{\ell}}$ are defined in Example \ref{ProcedurePartition}. \\
Note that $t_{\tau^{\ell}}-1$ is the first generation (backwards in time) at which a coalescence between two or more vertices visited during distinct excursions occurs so there exists a subset $J_{\ell}$ of $\{1,\ldots,|\bm{\pi}_{t_{\tau^{\ell}-1}}|\}$ and a collection $\{\alpha_i;\; i\in J_{\ell}\}$ of $|J_{\ell}|$ integers satisfying $\alpha_i\geq 1$ for all $i\in J_{\ell}$ and $\sum_{i\in J_{\ell}}\alpha_i\leq k$ such that 
\begin{align*}
    &\Eb\Big[\sum_{\Vect{x}\in\Delta^k_m}\un_{\{\underline{V}(\Vect{x})\geq-B\}}f^{\ell}_{\bm{t},\Pi}\un_{\Upsilon^{k,\mathfrak{s}_n}_{\mathfrak{g}_{l-1,n}}\setminus\Upsilon^{k,\mathfrak{s}_n}_{\mathfrak{g}_{l,n}}}(\Vect{x})\prod_{p=1}^{\mathfrak{e}}\P^{\mathcal{E}}\big(\max_{i\in I_p}T_{\VectCoord{x}{i}}<T^1\big)\big|\mathcal{F}_{t_{\tau^{\ell}}-1}\Big] \\ & \leq\mathfrak{C}_{\ref{LemmMultiExcur},3}\sum_{\Vect{z}\in\Delta^{|\bm{\pi}_{\tau^{\ell}-1}|}_{t_{\tau^{\ell}}-1}}f^{\tau^{\ell}-1}_{\bm{t}^{\tau^{\ell}-1},\Pi^{\tau^{\ell}-1}}(\Vect{z})\P^{\mathcal{E}}(T_{\Vect{z}}<T^1)\prod_{j=1}^{|\bm{\pi}_{\tau^{\ell}-1}|}(H_{\VectCoord{u}{j}})^{|\bm{B}^{\tau^{\ell}-1}_j|}\prod_{i\in J_{\ell}}e^{-\alpha_iV(\VectCoord{z}{i})}\un_{\{\underline{V}(\Vect{z})\geq-B\}}.
\end{align*}
Note that 
\begin{align*}
    \prod_{i\in J_{\ell}}e^{-\alpha_iV(\VectCoord{z}{i})}\un_{\{\underline{V}(\Vect{z})\geq-B\}}\leq&\prod_{i\in J_{\ell}}e^{-\alpha_iV(\VectCoord{z}{i})}\un_{\{\min_{i\in J_{\ell}}V(\VectCoord{z}{i})\geq-B,\;\min_{|z|=t_{\tau^{\ell}}-1}V(z)<\delta_0(t_{\tau_{\ell}}-1)\}} \\ & + e^{-\min_{|z|=t_{\tau^{\ell}-1}}V(z)}\un_{\{\min_{|z|=t_{\tau^{\ell}}-1}V(z)\geq\delta_0(t_{\tau_{\ell}}-1)\}},
\end{align*}
so $\Eb[\sum_{\Vect{x}\in\Delta^k_m}\un_{\{\underline{V}(\Vect{x})\geq-B\}}f^{\ell}_{\bm{t},\Pi}\un_{\Upsilon^{k,\mathfrak{s}_n}_{\mathfrak{g}_{l-1,n}}\setminus\Upsilon^{k,\mathfrak{s}_n}_{\mathfrak{g}_{l,n}}}(\Vect{x})\prod_{p=1}^{\mathfrak{e}}\P^{\mathcal{E}}(\max_{i\in I_p}T_{\VectCoord{x}{i}}<T^1)]$ is smaller than
\begin{align*}
    \Eb\Big[\mathfrak{C}_{\ref{LemmMultiExcur},3}\sum_{\Vect{z}\in\Delta^{|\bm{\pi}_{\tau^{\ell}-1}|}_{t_{\tau^{\ell}}-1}}f^{\tau^{\ell}-1}_{\bm{t}^{\tau^{\ell}-1},\Pi^{\tau^{\ell}-1}}(\Vect{z})&\P^{\mathcal{E}}(T_{\Vect{z}}<T^1)\prod_{j=1}^{|\bm{\pi}_{\tau^{\ell}-1}|}(H_{\VectCoord{u}{j}})^{|\bm{B}^{\tau^{\ell}-1}_j|} \\ & \times\big(e^{kB}\un_{\{\min_{|z|=t_{\tau^{\ell}}-1}V(z)<\delta_0(t_{\tau_{\ell}}-1)\}}+e^{-3\delta_0(t_{\tau_{\ell}}-1)}\big)\Big].
\end{align*}
Using the same argument as the one we used in the proof of Lemma \ref{EspProbaTAPartition} together with the Cauchy–Schwarz inequality, we obtain that the previous mean is smaller than
\begin{align*}
    &\mathfrak{C}_{\ref{LemmMultiExcur},3}\sup_{d\in\N^*}\Eb[(H_{d-1}^S)^{4k-1}]\Big(e^{kB}\P\big(\min_{|z|=t_{\tau^{\ell}-1}}V(z)<\delta_0(t_{\tau_{\ell}}-1)\big)^{1/2}+e^{-3\delta_0(t_{\tau_{\ell}}-1)}\Big) \\ & \leq\mathfrak{C}_{\ref{LemmMultiExcur},3}\sup_{d\in\N^*}\Eb[(H_{d-1}^S)^{4k-1}](e^{kB}+1)e^{-k\Lambda_{l}(\mathfrak{L}_n)},
\end{align*}
where we have used Lemma \ref{MinPotential} with $\zeta=\delta_0t_{\tau_{\ell}}$ and the fact that $t_{\tau_{\ell}}-1\geq\mathfrak{g}_{l,n}$. \\
Back to \eqref{DecompoMultiExcu} together with what we have just obtained and the fact that for all $\mathfrak{j}\in\{1,\ldots,\tau_{\ell}\}$, $t_{\mathfrak{j}}\leq\mathfrak{g}_{l-1,n}$, $\E[\sum_{\Vect{x}\in\Delta^k,\; |\Vect{x}|=\Vect{p}}\un_{\{\underline{V}(\Vect{x})\geq-B\}}\un_{\Upsilon^{k,\mathfrak{s}_n}_{\mathfrak{g}_{l-1,n}}\setminus\Upsilon^{k,\mathfrak{s}_n}_{\mathfrak{g}_{l,n}}}(\Vect{x})\un_{\mathfrak{S}^{k,\mathfrak{s}_n}\cap\mathfrak{E}_2^{k,\mathfrak{s}_n}}(\Vect{x})]$ is smaller than
\begin{align*}
    \mathfrak{C}_{\ref{LemmMultiExcur},3}\sup_{d\in\N^*}\Eb[(H_{d-1}^S)^{4k-1}](e^{kB}+1)e^{-k\Lambda_{l}(\mathfrak{L}_n)}\sum_{\ell=1}^{k-1}\;\sum_{\Pi\textrm{ increasing}}(\mathfrak{g}_{l-1,n})^{\tau_{\ell}}(\mathfrak{L}_n)^{\ell-\tau_{\ell}}.
\end{align*}
Note that $\tau_{\ell}\leq\ell<k$. Moreover, by definition, $\ell-\tau_{\ell}$ is smaller than the total number of coalescences occurring between two or more vertices which indexes belong to the same set $I_p^{\ell}$ and this number is smaller than $\sum_{p=1}^{\mathfrak{e}}(|I_p^{\ell}|-1)=k-\mathfrak{e}$ thus giving
\begin{align*}
    \E\Big[\underset{|\Vect{x}|=\Vect{p}}{\sum_{\Vect{x}\in\Delta^k}}\un_{\{\underline{V}(\Vect{x})\geq-B\}}\un_{\Upsilon^{k,\mathfrak{s}_n}_{\mathfrak{g}_{l-1,n}}\setminus\Upsilon^{k,\mathfrak{s}_n}_{\mathfrak{g}_{l,n}}}(\Vect{x})\prod_{p=1}^{\mathfrak{e}}Y_p\Big]\leq\mathfrak{C}_{\ref{LemmMultiExcur},1}\big(\Lambda_{l-1}(\mathfrak{L}_n)e^{-\Lambda_l(\mathfrak{L}_n)}\big)^k(\mathfrak{L}_n)^{k-\mathfrak{e}},
\end{align*}
which, by definition of $\Lambda_l(\mathfrak{L}_n)$, is equal to $\mathfrak{C}_{\ref{LemmMultiExcur},1}(\mathfrak{L}_n)^{k-\mathfrak{e}}$ and it yields \eqref{MultiExcuControl}. \\
In the same way, we can prove that 
\begin{align}\label{MultiExcuControlBis}
    \E\Big[\underset{|\Vect{x}|=\Vect{p}}{\sum_{\Vect{x}\in\Delta^k}}\un_{\{\underline{V}(\Vect{x})\geq-B\}}\un_{\Upsilon^{k,\mathfrak{s}_n}_{\mathfrak{g}_{l,n}}}(\Vect{x})\prod_{p=1}^{\mathfrak{e}}Y_p\Big]\leq\mathfrak{C}'_{\ref{LemmMultiExcur},1}\big(1+\Lambda_{l+1}(\mathfrak{L}_n)^k\big)(\mathfrak{L}_n)^{k-\mathfrak{e}},
\end{align}
for some constant $\mathfrak{C}'_{\ref{LemmMultiExcur},1}>0$. Putting together \eqref{EspSup}, \eqref{MultiExcuControl} and \eqref{MultiExcuControlBis}, we obtain, for some constant $\mathfrak{C}_{\ref{LemmMultiExcur},4}>0$
\begin{align*}
    &\E\big[\frac{1}{(\sqrtBis{n}\bm{L}_n)^k}\sup_{s\leq\mathfrak{s}_n}\mathcal{A}^k(\mathcal{D}_{n,T^{s}},\un_{\mathfrak{S}^{k,s}\cap\mathfrak{E}^{k,s}_2}\un_{\{\underline{V}(\cdot)\geq-B\}\cap\mathcal{C}^k_{\mathfrak{g}_{0,n}}})\big] \\ & \leq \mathfrak{C}_{\ref{LemmMultiExcur},4}\sum_{\mathfrak{e}=2}^{k-1}\Big(\frac{\mathfrak{L}_n}{\sqrtBis{n}}\Big)^{k-\mathfrak{e}}\big(2+\Lambda_{l_0+1}(\mathfrak{L}_n)^k\big).
\end{align*}
Using the fact that $\Lambda_{l_0+1}(\mathfrak{L}_n)^k=(\log \Lambda_{l_0}(\mathfrak{L}_n))^k$, we obtain \textit{\ref{MultiExcu}} thanks to the Assumption \ref{AssumptionSmallGenerations}.
\end{proof}

\noindent We are now ready to prove Proposition \ref{GENPROPCONV1}:

\begin{proof}[Proof of Proposition \ref{GENPROPCONV1}]
Let $\varepsilon'>0$. First, note that thanks to Lemma \ref{LemmeExc1} and \textbf{Fact 1} \eqref{Fact1} there exists $a_{\varepsilon'}>0$ such that we can restrict our study to the $k$-tuples of vertices in the set $\mathfrak{S}^{k,s}\cap\{\underline{V}(\cdot)\geq-a_{\varepsilon'}\}$
\begin{align*}
    \lim_{\varepsilon'\to 0}\limsup_{n\to\infty}\P^*\Big(\sup_{s\leq\mathfrak{s}_n}\mathcal{A}^k\big(\mathcal{D}_{n,T^{s}},\un_{\Delta^k\setminus\mathfrak{E}^{k,s}}(1-\un_{\mathfrak{S}^{k,s}\cap\{\underline{V}(\cdot)\geq-a_{\varepsilon'}\}})\big)>\varepsilon(\sqrtBis{n}\bm{L}_n)^k\Big)=0,
\end{align*}
where we recall that $\mathfrak{s}_n=\sqrtBis{n}/\varepsilon_1$. Then, note that $\mathcal{A}^k(\mathcal{D}_{n,T^{s}},\un_{\Delta^k\setminus\mathfrak{E}^{k,s}}\un_{\mathfrak{S}^{k,s}\cap\{\underline{V}(\cdot)\geq-a_{\varepsilon'}\}})$ is smaller than
\begin{align*}
    \mathcal{A}^k(\mathcal{D}_{n,T^{s}},\un_{\mathfrak{S}^{k,s}\cap\mathfrak{E}^{k,s}_1})+\mathcal{A}^k(\mathcal{D}_{n,T^{s}},\un_{\mathfrak{S}^{k,s}\cap\mathfrak{E}^{k,s}_2}\un_{\{\underline{V}(\cdot)\geq-a_{\varepsilon'}\}}).
\end{align*}
Hence, by Markov inequality, the result follows using Lemma \ref{LemmMultiExcur} with $B=a_{\varepsilon'}$.
\end{proof}

\vspace{0.5cm}

\noindent\begin{merci}
    I would like to express my sincere thanks to an anonymous referee for her/his very careful reading of the paper, her/his relevant and precise remarks that were very useful to help improve the presentation of the paper.
\end{merci}

\bibliographystyle{alpha}
\bibliography{thbiblio}

\end{document}